\documentclass[reqno,oneside]{amsart}
\usepackage[a4paper,
  textwidth=17.0cm,
  textheight=24.0cm,
  ]{geometry}
\usepackage[dvipsnames]{xcolor}
\usepackage{mathrsfs}
\usepackage{enumitem}
\usepackage{mathtools}
\usepackage[final]{hyperref}
\usepackage{esint}
\usepackage{bm}
\usepackage{subcaption}
\usepackage{pgfplots}
\pgfplotsset{compat = newest}
\usepackage{soul}
\usepackage[english]{babel}
\usepackage[markup=underlined]{changes}
\definechangesauthor[color=BrickRed]{RB}
\definechangesauthor[color=blue]{MF}
\definechangesauthor[color=ForestGreen]{SA}
\usepackage{todonotes}

\newtheorem{theorem}{Theorem}[section]
\newtheorem{lemma}[theorem]{Lemma}
\newtheorem{proposition}[theorem]{Proposition}
\newtheorem{corollary}[theorem]{Corollary}

\theoremstyle{definition}
  \newtheorem{definition}[theorem]{Definition}

\theoremstyle{remark}

\newcommand{\N}{\mathbb{N}}
\newcommand{\Z}{\mathbb{Z}}
\newcommand{\R}{\mathbb{R}}

\newcommand{\Id}{\mathbf{Id}}
\newcommand{\id}{\mathbf{id}}
\newcommand{\eps}{\varepsilon}
\newcommand{\vphi}{\varphi}
\newcommand{\weakly}{\rightharpoonup}
\newcommand{\weaklystar}{\stackrel{*}{\rightharpoonup}}
\newcommand{\defas}{\coloneqq}
\newcommand{\sym}{\mathrm{sym}}
\newcommand{\elen}{\mathcal{W}^{\mathrm{el}}}
\newcommand{\cplen}{\mathcal{W}^{\mathrm{cpl}}}
\newcommand{\hyen}{\mathcal{H}}
\newcommand{\mechen}{\mathcal{M}}
\newcommand{\Deltatwo}{\Delta^{2}}
\newcommand{\toten}{\mathcal{E}}

\newcommand{\inten}{W^{\mathrm{in}}}

\newcommand{\Yid}{\mathcal{Y}_{\id}}
\newcommand{\Yidreg}{\mathcal{Y}^{\rm reg}_{\id}}
\newcommand{\hypotsclr}{\mathfrak{h}}

\def\Xint#1{\mathchoice
    {\XXint\displaystyle\textstyle{#1}}%
    {\XXint\textstyle\scriptstyle{#1}}%
    {\XXint\scriptstyle\scriptscriptstyle{#1}}%
    {\XXint\scriptscriptstyle\scriptscriptstyle{#1}}%
    \!\int}
\def\XXint#1#2#3{{\setbox0=\hbox{$#1{#2#3}{\int}$}
      \vcenter{\hbox{$#2#3$}}\kern-.5\wd0}}

\def\mint{\Xint-}

\newcommand{\Om}{\Omega}

\newcommand{\diss}{\mathcal{R}}

\usepackage{bbm}
\usepackage{dsfont}
\newcommand{\indic}{\mathds{1}}
\newcommand{\elpot}{W^{\mathrm{el}}}
\newcommand{\cplpot}{W^{\mathrm{cpl}}}
\newcommand{\hypot}{H}
\newcommand{\felpot}{W}
\newcommand{\disspot}{R}

\newcommand{\pl}{\partial}

\newcommand{\yst}[1]{y_{\tau}^{(#1)}}
\newcommand{\ysts}[1]{y_{\tau}^{(#1)}}
\newcommand{\tst}[1]{\theta_{\tau}^{(#1)}}
\newcommand{\tsts}[1]{\theta_{\tau}^{(#1)}}
\newcommand{\lst}[1]{f_{\tau}^{(#1)}}
\newcommand{\wst}[1]{w_{\tau}^{(#1)}}
\newcommand{\wsts}[1]{w_{\tau}^{(#1)}}

\newcommand{\fst}[1]{f_\tau^{(#1)}}

\newcommand{\bt}{\theta_\flat}

\newcommand{\btst}[1]{\theta_{\flat, \tau}^{(#1)}}
\newcommand{\hc}{\mathbb{K}}
\newcommand{\hcm}{\mathcal{K}}

\newcommand{\drate}{\xi}

\newcommand{\haus}{\mathcal{H}}
\newcommand{\ddif}{\delta_\tau}
\newcommand{\aC}{C_0}

\newcommand{\intQ}{\int_I\int_\Omega}

\usepackage{stackengine}

\newcommand*{\di}{\mathop{}\!\mathrm{d}}

\DeclareMathOperator{\diver}{div}

\DeclarePairedDelimiterX\setof[1]\{\}{#1}
\DeclarePairedDelimiterX\abs[1]\lvert\rvert{#1}
\DeclarePairedDelimiterX\norm[1]\lVert\rVert{#1}
\DeclarePairedDelimiterX\sprod[2]\langle\rangle{#1, #2}

\newcommand{\sa}{\color{black}}
\newcommand{\rb}{\color{black}}

\newcommand{\ee}{\color{black}}

\definecolor{darkblue}{rgb}{0.0, 0.0, 0.55}
\newcommand{\seb}[1]{{\color{black} #1}}
\newcommand{\ste}{\color{black}}

\newcommand{\BBB}{\color{black}}
\newcommand{\EEE}{\color{black}}

\newcommand{\RB}{\color{black}}

\newcommand{\MMM}{\color{black}}

\newcommand{\ZZZ}{\color{black}}

\usepackage{xargs}

\numberwithin{equation}{section}

\makeatletter
\@namedef{subjclassname@2020}{%
 \textup{2020} Mathematics Subject Classification}
\makeatother

\begin{document}
\title{Thermo-elastodynamics of nonlinearly viscous solids}

\author[S. Almi]{Stefano Almi}
\address[Stefano Almi]{Department of Mathematics and Applications ``R.~Caccioppoli'', University of Naples Federico II, Via Cintia, Monte S. Angelo, 80126 Napoli, Italy.}
\email{stefano.almi@unina.it}

\author[R.~Badal]{Rufat Badal}
\address[Rufat Badal]{
  Department of Mathematics \\
  Friedrich-Alexander Universit\"at Erlangen-N\"urnberg \\
  Cauerstr.~11, D-91058 Erlangen, Germany
}
\email{rufat.badal@fau.de}

\author[M.~Friedrich]{Manuel Friedrich} 
\address[Manuel Friedrich]{%
  Department of Mathematics \\
  Friedrich-Alexander Universit\"at Erlangen-N\"urnberg \\
  Cauerstr.~11, D-91058 Erlangen, Germany 
}
\email{manuel.friedrich@fau.de}

\author[S.~Schwarzacher]{Sebastian Schwarzacher} 
\address[Sebastian Schwarzacher] 
 {%
 Department of Mathematics, Uppsala University, Ångströmlaboratoriet, Lägerhyddsvägen  1, 752 37 UPPSALA \\  \&  Department of Mathematical Analysis, Charles University, Sokolovská 83, 186 75 Praha 8 
}
\email{schwarz@karlin.mff.cuni.cz}


\maketitle

\ZZZ 

\begin{abstract}
 In this paper, we study the thermo-elastodynamics of nonlinearly viscous solids in the Kelvin-Voigt rheology where both the elastic and the viscous stress tensors comply with the frame-indifference principle. The system features a force balance including inertia in the frame  of nonsimple materials and a heat-transfer equation which is governed by the Fourier law in the deformed configuration.  Combining  a staggered   minimizing  movement scheme  for quasi-static thermoviscoelasticity  \cite{tve_orig, tve} with a variational approach to hyperbolic PDEs developed in \cite{veaccel2}, our main result consists in establishing the existence of weak solutions in the dynamic case. This is first achieved by including an additional higher-order regularization for the dissipation. Afterwards, this regularization can be removed by passing to a weaker formulation of the heat-transfer equation which complies with a total energy balance.  The latter description hinges on regularity theory for   the fourth order $p$-Laplacian  which induces regularity  estimates of the deformation  {beyond} the standard  estimates available from energy bounds. Besides being crucial for the proof, these  extra regularity properties might be of independent interest and seem to be new in the setting of nonlinear viscoelasticity, also in the static or quasi-static  case.  
\end{abstract}

\EEE

\section{Introduction}
Understanding the coupling  between mechanical and thermal phenomena in viscoelastic solids has been a mainstay in the mathematical and physical literature over the last decades. Even at  small strains, the problem is notoriously difficult since the heat-transfer equation has  no obvious variational structure due to  the low regularity of data. In fact, after the pioneering work of {\sc Dafermos} \cite{dafermos1, dafermos2, dafermos3, dafermos4} in one space dimension, new fundamental ideas related to the  existence theory for parabolic equations with measure-valued data developed in \cite{Boccardoetal, BoccardoGallouet89Nonlinear} were needed to obtain results in  three dimensions  \cite{Blanchard, Bonetti, Roubicek09}. At large strains, the problem is still considered to be extremely difficult even in the isothermal case, due to the highly nonlinear nature of models   respecting   material frame indifference \cite{Antmann98Physically}. For some results without temperature coupling, we refer to  \cite{potier-ferry-1,potier-ferry-2} for existence of  global-in-time  weak solutions for  initial data sufficiently  close  to a smooth equilibrium  and to a    local-in-time existence result \cite{Lewick}. By now, more general settings can only be treated by passing to weaker solution concepts such as  measure-valued solutions \cite{demoulini,  DST, gala}.     Resorting to energy densities with  higher-order spatial gradients, i.e., to    so-called nonsimple materials \cite{Toupin62Elastic, Toupin64Theory},   existence of weak solutions  has been shown in  \cite{FiredrichKruzik18Onthepassage, tve_orig} for the quasi-static case (without inertia) and in \cite{veaccel2}  for the dynamic case (with inertia). The variational approach adopted in these papers is quite flexible and has led to various extensions in the last years, ranging from  models for dimension reduction \cite{FK_dimred, MFLMDimension2D1D, MFLMDimension3D1D}, \seb{to  problems with   self-contact \cite{gravina, gravina2, kroemrou}, approximability~\cite{cesik}, diffusion  \cite{liero}, or  homogenization \cite{gahn}, to  applications for  fluid-structure interactions \cite{veaccel2, veaccel, breit, sperone}.}

 Nonlinear frame-indifferent  models in  thermoviscoelasticity  were  analyzed only very recently  \cite{tve, RBMFLM, RBMFMKLM, tve_orig}, again adopting the concept of nonsimple materials,   yet neglecting inertial effects.    The goal of this work is to extend this analysis to the setting of thermo-elastodynamics including inertia. 
While our work follows the Lagrangian perspective, let us mention that in the last years several works appeared in the isothermal and nonisothermal framework which employ the alternative Eulerian approach instead,   see  \cite{Roubicek23Eulerian3, Roubicek23Eulerian,    Roubicek23Eulerian2, Roubicek24Eulerian, Roubicek23Eulerian4}. In  this context,   higher-order gradients are involved rather in the dissipative than in the conservative part, which sometimes is referred to as  multipolar viscous solids. Besides adopting the Lagrangian framework, a main motivation of our work is to establish an existence result \emph{without} higher-order regularization of the dissipation. 

We now introduce the large-strain model in more detail.   In the Kelvin-Voigt  rheology, the force balance of a nonlinearly viscoelastic material in a setting of nonsimple materials is given by the system 
\begin{equation}\label{viscoel}
   f  =  \rho \partial^{2}_{tt} y - {\rm div} \big( \partial_{F} W ( \nabla y, \theta) + \partial_{\dot{F}} R(\nabla y, \partial_{t} \nabla y, \theta)  - \nabla  ({D H}(\Delta y))  \big) \qquad \text{in $[0,T]\times \Omega$.}
\end{equation}
Here, $[0, T]$ is a process time interval with $T > 0$, $\Omega \subset \R^d$ \ZZZ ($d=2,3$) \EEE  denotes the \ZZZ reference \EEE configuration, $y \colon [0, T] \times \Omega \to \R^d$ is the time-dependent  deformation,   $\theta \colon [0, T] \times \Omega \to [0,\infty)$ denotes the temperature, and $f \colon [0, T] \times \Omega \to \R^d$ is a volume density of external forces acting on $\Omega$.
  The free energy density   $W \colon \R^{d\times d} \times [0, \infty) \to \R \cup \setof{+\infty}$ depends on the deformation gradient $\nabla y$ (with placeholder $F \in \R^{d\times d}$) and  respects frame indifference under rotations  as well as  positivity of the determinant of $\nabla y$. Additionally, adopting the framework of nonsimple materials,  the stored energy features a contribution depending on the Laplacian $\Delta y$ given in terms of a convex potential   $H\colon \R^d \to \R$ with $p$-growth for some $ p>d$.   Finally, $R  \colon \R^{d \times d} \times \R^{d \times d} \times [0, \infty) \to \R\BBB$ denotes a (pseudo)potential of dissipative forces ($\dot F$ is the time derivative of $F$). As observed by {\sc Antman} \cite{Antmann98Physically}, $R$ must comply with  a   time-continuous frame indifference principle meaning that $R$ can be written in terms of the right Cauchy-Green tensor $C \defas F^T F$  and its time derivative $\dot C \defas \dot F^T F + F^T \dot F$, see \ref{D_quadratic} below for details.

 The system \eqref{viscoel} is coupled   to a heat-transfer equation of the form
\begin{equation}\label{heat}
  c_V(\nabla y,\theta) \, \partial_t \theta =
    \diver(\mathcal{K}(\nabla y, \theta) \nabla\theta)
    + \partial_{\dot F} R(\nabla y, \nabla \partial_t  y, \theta) : \nabla \partial_t y
    + \theta \seb{\partial_{F \theta}^2} \ZZZ W \EEE (\nabla y, \theta) : \nabla \partial_t  y \qquad  \text{in $[0,T]\times \Omega$},  
\end{equation}
 where     $c_V(F,\theta) = -\theta \ZZZ \partial^2_{\theta\theta} \EEE W(F, \theta)$ is the  heat capacity, $\mathcal{K}$ denotes the matrix of the heat-conductivity coefficients, and the last term plays the role of an {adiabatic heat} source. This corresponds to  a heat transfer modeled by the Fourier law in the deformed configuration  which is pulled back to the reference configurations and thus includes dependence on the deformation gradient.   
The coupled system \eqref{viscoel}--\eqref{heat} is complemented with suitable initial and boundary conditions, see  \eqref{e:strong-bdry}--\eqref{iiiniita} below. 

The goal of this article is to establish an existence result for weak solutions  to the nonlinear thermo-elastodynamic system  \eqref{viscoel}--\eqref{heat}, see Theorem \ref{thm:main-thermal-elasto-unregu}.   Our proof strategy heavily  hinges on two recent advances in the variational analysis of nonlinearly elastic solids: we combine  the staggered  \ZZZ minimizing \EEE movement scheme for proving  existence results in quasi-static thermoviscoelasticity  \cite{tve_orig, tve} with a variational approach to hyperbolic PDEs \cite{veaccel2}  which allows to include inertia.

In the following, we describe the main ingredients for the proof in more detail. The fundamental idea in \cite{veaccel2}  consists in  replacing the acceleration term $\rho   \pl_{tt}^2   y$ by a discrete difference $\rho \frac{\pl_t y - \pl_t y(\cdot - h)}{h}$ which allows to  turn the hyperbolic problem \eqref{viscoel} into a parabolic one. The latter time-delayed problem can  be approximated by a time-discretized scheme as in   \cite{tve_orig, tve} with time step $\tau>0$. Then, given solutions to the discretized problems with two different length scales $\tau$ and $h$  (called  the velocity and the acceleration time scale, respectively), one first passes to $\tau \to 0$ and afterwards to  $h \to 0$ to obtain a weak solution for   \eqref{viscoel}. As in \cite{tve_orig},  a generalized version of Korn's inequality \cite{pompe}  relying on the second-order regularization is essential in order to tame the nonlinearity arising from  the frame indifference of the dissipation term. Concerning the coupling to the heat-transfer equation, the approach in \cite{tve_orig, tve} crucially relies on the  theory of parabolic equations with measure-valued right-hand side \cite{BoccardoGallouet89Nonlinear}. A delicate part of the proof lies in the passage to  the limit $\tau \to 0$ in the dissipation term $\partial_{\dot F} R(\nabla y, \nabla \partial_t  y, \theta) : \nabla \partial_t y$, see \eqref{heat}. For this, strong convergence of the time-discrete approximations $\nabla \partial_t y_\tau$ is indispensable which is guaranteed by   exploiting the convergence of a mechanical energy balance, cf.\ \cite[Proposition 5.1]{tve_orig} for details.

Although all techniques mentioned above are crucial ingredients in our work, it turns out that they do not suffice in the setting with heat coupling \emph{and} inertia. The main reason lies in missing regularity which impedes the derivation of a mechanical energy balance. To explain this issue, let as consider the simplified problem 
\begin{align}\label{expli}
\rho \partial_{tt}^2  y -\Delta \partial_t  y +\Delta (|\Delta y|^{p-2} \Delta y)       =f,
\end{align}
which arises from \eqref{viscoel} by neglecting the first Piola-Kirchhoff stress tensor $\partial_{F} W$, and considering a linear variant of $\partial_{\dot{F}} R $ as well as  a $p$-homogeneous variant of $H$.  In the quasi-static case $\rho = 0$ or in the time-delayed problem where  $\rho  \pl_{tt}^2   y$ is replaced  by the discrete difference $\rho \frac{\pl_t y - \pl_t y(\cdot - h)}{h}$, a test of the time-discretized problem  with   $\partial_t y$ and an integration by parts (neglecting boundary terms) leads to the natural energy bounds $\Delta y\in L^\infty([0,T]; L^p(\Omega))$ and $\nabla \partial_t  y \in L^2([0,T];L^2(\Omega))$. Then, in the case $\rho = 0$, a mechanical energy balance is achieved by testing \eqref{expli} with $\partial_t  y $, cf.\ \cite[Equation (4.11)]{tve}. In this context, the term $\Delta (|\Delta y|^{p-2} \Delta y)  $ might in principle not have the correct duality coupling to apply the chain rule. However,  since the other two terms $f$ and $\Delta \partial_t  y$ are in duality, also the delicate fourth-order term can be handled  by comparison. In contrast, for $\rho>0$, the two terms $\rho \partial_{tt}^2  y$ and $\Delta (|\Delta y|^{p-2} \Delta y)$ are  not in duality and the chain rule (and thus the mechanical  energy balance) may fail. 

This fundamental issue  has already been observed in \ZZZ \cite[Remark 6.6]{tve_orig}. \EEE A possible workaround lies in  adding an additional regularization for the dissipation, see \eqref{e:mechanical-strong}, which  in the  simplified setting reads as 
\begin{align}\label{expli2}
\rho \partial_{tt}^2  y -\Delta \partial_t  y +\Delta (|\Delta y|^{p-2} \Delta y)     - \eps \partial_t  \Delta^3 y     =f.
\end{align}
With the test $\partial_t y$, this induces the energy bound $\nabla \Delta \partial_t  y \in L^2([0,T]; L^2(\Omega))$ which is strong enough to recover the  chain rule. In this case, a mechanical energy balance can be guaranteed and we can follow the strategy devised in \cite{tve_orig, tve} and \cite{veaccel2}, see Theorem \ref{thm:main-thermal-elasto-regu}. (Note that we choose a simple higher-order regularization which does not comply with the principle of dynamical frame indifference. A frame-indifferent regularization would necessarily be very nonlinear.) This regularized setting is related to \cite{Roubicek24Eulerian} where existence results under higher-order regularizations of the dissipation have been derived in a Eulerian settting.  Yet, a main motivation of our work is to derive an existence result \emph{without} such regularization.

Our strategy relies on passing to a weaker formulation of the heat-transfer equation \eqref{heat} which is inspired by the derivation of a total energy balance (see \cite[Equation (2.21)]{tve_orig} or \eqref{final energy balance} below) and does not feature the delicate dissipation term $\partial_{\dot F} R(\nabla y, \nabla \partial_t  y, \theta) : \nabla \partial_t y$, see \eqref{e:new-thermal-equation-lim} for details. On a formal level, the idea is to test \eqref{viscoel} with $\partial_t y$ which allows to replace the dissipation term in  \eqref{heat}. As discussed above, however, this test is actually   not allowed in \eqref{expli}. Therefore, we    perform this replacement  first on the regularized level \eqref{expli2}, and  afterwards we  pass to the limit $\eps \to 0$.   This procedure leads to a modified weak formulation of the system   which  does not guarantee a mechanical energy balance but has the essential feature that   the total energy is in equilibrium with the work by external body forces and   heat sources, see \eqref{final energy balance}. Moreover, the solution concept introduced here becomes a standard weak solution or a strong solution once the necessary regularity properties for the deformation and the temperature are available.

Although curing the issue with the dissipation,  the passage to the weaker modified  setting causes a new problem: the resulting weak formulation features a third-order term $\nabla (D{H}(\Delta y ))$ which is not compatible with the available energy bound $\Delta y  \in L^\infty([0,T];L^p(\Omega))$, see below \eqref{expli}. Therefore,  it is necessary to improve the regularity of the deformation. Loosely speaking, this is achieved by testing \eqref{expli} with  $-\Delta y$ which after integration by parts  (omitting any boundary terms) leads to an  {\em elliptic estimate}. In fact, using $\nabla \partial_t  y \in L^2([0,T];L^2(\Omega))$, the  first  term   $|\int   \partial_{tt}^2  y \Delta y \di t \di x| \le C \norm{\nabla \partial_t  y }_{L^2([0,T];L^2(\Omega))}^2 \le C$ is controlled. Assuming $p=2$ for simplicity here, the second and third term can be controlled as
\begin{align*}
&\Big|\int_0^T\int_\Omega \Delta \partial_{t}  y \cdot \Delta y  \di t \di x\Big| \le C\norm{\nabla  \partial_{t}  y}_{ \in L^2([0,T];L^2(\Omega))} \norm{\nabla \Delta y}_{  L^2([0,T];L^2(\Omega))}\le  C\norm{\nabla \Delta y}_{  L^2([0,T];L^2(\Omega))}, \\
& \Big|\int_0^T\int_\Omega \Delta ( \Delta y) \cdot \Delta y\di t \di x\Big| \ge \tfrac{1}{C}\norm{ \nabla\Delta  y}^2_{  L^2([0,T];L^2(\Omega))}.  
\end{align*}
 This allows to obtain the control  $\nabla \Delta y \in  L^2([0,T];L^2(\Omega))$ which suffices to give sense to the term $\nabla (D{H}(\Delta y ))$ in the weak formulation. Again, on a rigorous level, this test is performed for the  regularized problem  \eqref{expli2} with $\eps$-independent bounds, and then the regularity for $y$ is obtained in the limit of vanishing regularization $\eps \to 0$, see Proposition \ref{cor:higher_reg_y} and Lemma \ref{l:laplace-regularity}  for details. \ZZZ More \EEE precisely, for  given $p>d$, the additional regularity    reads as  $(1+\abs{\Delta y})^\frac{p-2}{2}\abs{\nabla \Delta y}^2\in L^2([0,T];L^2(\Omega))$, see Theorem \ref{thm:main-thermal-elasto-unregu}.  
  Even in the nonlinear case of $p\neq 2$, the regularity estimates introduced here rely on the theory for the Laplace operator only and as such are {\em independent of nonlinear regularity techniques}. This seems to be a special feature of the fourth order $p$-Laplacian, which was observed by the authors very much to their surprise. Up to their knowledge, it has not been used before.
 
Besides being crucial for our proof, the result might be of independent interest and improves the known regularity properties also for results in the quasi-static case ($\rho =0$) \cite{tve_orig, tve} or in the isothermal case \cite{FiredrichKruzik18Onthepassage}. \seb{It seems that even in the static case of elastic minimizers this extra regularity property has not been shown previously}.  Let us mention, however, that compared to \cite{tve, FiredrichKruzik18Onthepassage, tve_orig}  the regularity issues force us to impose Dirichlet conditions on the \emph{entire} boundary $\partial \Omega$.

The plan of the paper is as follows. Section~\ref{sec:tau_to_zero}  introduces the nonlinear  model   and states our main results.
Then, the results are proved in Sections \ref{s:min-mov-scheme}--\ref{sec: epstozero}. We start by considering the $\eps$-regularized problem and introduce  a discretized solution with time stepping $\tau$ of the parabolic approximation with time delay  $h>0$. This introduces a three layer approximation, and we successively pass to the limits in the layers, namely first in $\tau$ (Section~\ref{s:min-mov-scheme}), then in $h$ (Section \ref{sec:vanish_h}), and eventually in the regularization $\eps$ (Section \ref{sec: epstozero}). It is important to mention that all essential a priori bounds are already established on the   $\tau$-level in Section \ref{s:min-mov-scheme} and transfer over to the limits $\tau \to 0$, $h \to 0$, and $\eps \to 0$.  Moreover, some additional higher-order bounds based on  elliptic regularity theory   are provided  (see Lemma \ref{l:laplace-regularity}) that  blow up in the limit $\eps \to 0$. Still, they are  crucial to perform the final limiting passage in the heat-transfer equation to control some $\eps$-dependent terms resulting from the regularization, see Proposition \ref{thm:mu_convergence_thermal} for details. 
\EEE

\section{The model and main results}\label{sec:tau_to_zero}

\subsection{Notation}
 Denoting by $d \in \setof{2, 3}$ the dimension, we indicate  by $\Omega \subset \R^d$ an open bounded set with $C^5$-boundary  and \BBB fix  $p \in (d, 2^*)$, where $2^* = \infty $ for $d=2$ and $2^* = 6$ for $d=3$. \EEE   
In what follows, we use standard notation for Lebesgue, Sobolev, \BBB and Bochner \EEE spaces. \BBB By $\indic_J$ we denote the indicator function of a set $J \subset \R$ \ZZZ or $J \subset \Omega$. \EEE  
The lower index $_+$ means nonnegative elements, i.e., $L^2_+(\Omega)$ denotes the convex cone of nonnegative functions belonging to $L^2(\Omega)$ and a similar definition is used for $H^1_+(\Omega)$. \BBB Mean integrals  are denoted by $\mint$. \EEE 
We also set $\R_+\defas [0,+\infty)$.
Let $a \wedge b\defas \min\setof{a, b}$ for $a,b \,\in \R$. 
 Moreover, we \ee let $\Id \in \R^{d \times d}$ be the identity matrix, and $\id(x) \defas x$ stands for the identity map on $\R^d$.
We define the subsets $SO(d) \defas \setof{A \in \R^{d \times d} \colon A^T A = \Id, \, \det A = 1 }$, $GL^+(d) \defas \setof{F \in \R^{d \times d} \colon \det(F) > 0}$, and $\R^{d \times d}_\sym \defas \setof{A \in \R^{d \times d} \colon A^T = A}$.
Furthermore,  for a matrix $F \in \R^{d \times d}$ we write \ee $F^{-T} \defas (F^{-1})^T=(F^T)^{-1}$, and given a tensor  $G$ \ee (of arbitrary dimension  and order\ee),  $\abs{G}$ \ee will denote its Frobenius norm.
We write the scalar product between vectors and  matrices \BBB as $\cdot$ and  $:$, \EEE  respectively. \BBB The tensor product of two vectors $v_1,v_2 \in \R^d$ is denoted by $v_1 \otimes v_2 \in \R^{d \times d}$. \EEE 
As usual, in the proofs generic constants $C$ are strictly positive and may vary from line to line.  \BBB
If not stated otherwise, all constants   only depend \ZZZ on \EEE  $d$,  $p$,  $\Omega$, and the potentials  and data  defined in Subsection \ref{sec:setting} \ZZZ below. \EEE  
We frequently use a scaled version of Young's inequality with constant $\lambda \in (0,1)$ by which we mean $ab \le \lambda a^q + C_\lambda b^{q'}$ for $a, b \,\ge 0$, exponents $q, \, q' > 1$ with $1/q + 1/q' = 1$, \BBB and a suitable constant $C_\lambda>0$. \EEE

For $\Omega \subset \R^d$ and $p$ as above, we introduce the set of \emph{admissible deformations} by 
\begin{equation}
\label{def_Yid}
  \Yid \defas \setof*{
    y \in W^{2, p}(\Omega; \R^d) \colon
    y = \id \text{ on }  \partial \Omega \ee, \,
     \det(\nabla y) > 0 \ee \text{ in } \Omega
  },
\end{equation}
and we say that the \emph{absolute temperature} $\theta$ is admissible if $\theta \in L^1_+(\Omega)$.

\subsection{Energies and their respective potentials}\label{sec:setting}
The variational setting described in the sequel  mostly \ee coincides with the one from \cite{tve}, up to a more special choice of \BBB the \EEE strain-gradient energy. \BBB In the following, let \EEE $\aC \ZZZ \ge 1 \EEE$ be some fixed positive constant.

\noindent \textbf{Mechanical energy and coupling energy:}
The \emph{elastic energy} $\elen \colon \Yid \to \R_+$ is given by
\begin{equation}\label{purely_elastic}
  \elen(y) \defas \int_\Omega \elpot(\nabla y) \di x,
\end{equation}
where $\elpot \colon GL^+(d) \to \R_+$ is a frame indifferent elastic energy potential with the usual assumptions in nonlinear elasticity.
More precisely, we require:
\begin{enumerate}[label=(W.\arabic*)]
  \item \label{W_regularity} $\elpot$ is \BBB $C^2$; \EEE
  \item \label{W_frame_invariace} Frame indifference: $\elpot(QF) = \elpot(F)$ for all $F \in GL^+(d)$ and $Q \in SO(d)$;
  \item \label{W_lower_bound} Lower bound: $\elpot(F) \ge \frac{1}{\aC} \big(|F|^2 + \det(F)^{-q}\big) - \aC$ for all $F \in GL^+(d)$, where $q \ge \frac{pd}{p-d}$.
\end{enumerate}
Adopting the concept of 2nd-grade nonsimple materials, see \cite{Toupin62Elastic, Toupin64Theory}, we also consider a \emph{strain-gradient energy term} $\hyen \colon \Yid \to \R_+$, defined as
\begin{equation}\label{hyperelastic}
  \hyen(y) \defas \int_\Omega \hypot(\Delta y) \di x.
\end{equation}
Here, $\hypot \colon \R^{d} \to \R_{+}$ is of the form
\begin{equation}\label{e:Hdef}
  \hypot(v) = \hypotsclr( |v| )
\end{equation}
for $v \in \R^d$, where $\hypotsclr \colon \R_+ \to \R_+$ is defined as
\begin{equation}
\label{e:psi}
\hypotsclr(s) \defas \int_{0}^{s} \max\{2\sigma \ee, p\sigma^{p-1}\} \di \sigma.
\end{equation}
\BBB The definition of $\hypotsclr$ ensures that $H$ is uniformly convex and has $p$-growth. \EEE  More precisely, we have
\begin{enumerate}[label=(H.\arabic*)]
  \item \label{H_regularity} $\hypot$ is \BBB uniformly \EEE convex and $C^1$;
  \item \label{H_frame_indifference} Frame indifference: \BBB $  H(Q \Delta y)  =   H(\Delta y)$ in $\Omega$ for all $y \in \Yid$ \EEE and $Q \in SO(d)$;
  \item \label{H_bounds} $\abs{v}^p\ee \leq H(v) \leq \aC \BBB  \abs{v}^p \EEE $ and $\abs{ \BBB D H \EEE (v)} \leq \aC  \BBB \abs{v}^{p-1} \EEE$ for all $v \in \R^{d}$,
\end{enumerate}
where $DH(v) \defas (\partial_{v_i} H(v))_{i=1}^d = \max\{2, p |v|^{p-2}\} v$ is the gradient of $H$ with respect to $v$.
The \emph{mechanical energy} $\mechen \colon \Yid \to \R_+$ is then defined as
\begin{equation}\label{mechanical}
  \mechen(y) \defas \elen(y) + \hyen(y).
\end{equation} 
Besides the mechanical energy, we introduce a \emph{coupling energy} $\cplen \colon \Yid \times L^1_+(\Omega) \to \R$ given by
\begin{equation}\label{koppl}
  \cplen(y, \theta) \defas \int_\Omega \cplpot(\nabla y, \theta) \di x,
\end{equation}
where $\cplpot \colon GL^+(d) \times \R_+ \to \R$ describes mutual interactions of mechanical and thermal effects, and satisfies
\begin{enumerate}[label=(C.\arabic*)]
  \item \label{C_regularity} $\cplpot$ is continuous, and $C^2$ in $GL^+(d) \times (0, \infty)$;
  \item \label{C_frame_indifference} $\cplpot(QF, \theta) = \cplpot(F, \theta)$ for all $F \in GL^+(d)$, $\theta \geq 0$, and $Q \in SO(d)$;
  \item \label{C_zero_temperature} $\cplpot(F, 0) = 0$ for all $F \in GL^+(d)$;
  \item \label{C_lipschitz} $|\cplpot(F,\theta) - \cplpot(\tilde{F}, \theta)| \le \aC(1 + |F| + |\tilde{F}|)|F - \tilde{F}|$ for all $F, \, \tilde F \in GL^+(d)$, and $\theta \geq 0$;
  \item \label{C_bounds} For all $F \in GL^+(d)$ and $\theta > 0$ it holds that
  \begin{align*}
    \abs{ \ZZZ \partial_{FF}^2 \EEE W^{\rm cpl}(F,\theta)} &\le \aC, &
    \abs{\RB\partial_{F\theta}^2\EEE\cplpot(F, \theta)} &\leq \frac{\aC(1+|F|)}{\max\lbrace \theta,1\rbrace}, &
    \frac{1}{\aC} & \leq -\theta \ZZZ \pl_{\theta\theta}^2 \EEE \cplpot(F, \theta) \leq \aC.
  \end{align*}
\end{enumerate}
Notice that, by \ref{C_zero_temperature} and the second bound in \ref{C_bounds}, $\pl_F \cplpot$ can be continuously extended to zero temperatures with $\pl_F W^{\rm cpl}(F, 0) = 0$. For $F \in GL^+(d)$ and $\theta \geq 0$, we define the \emph{total free energy potential} 
\begin{equation}\label{eq: free energy}
  \felpot(F, \theta) \defas \elpot(F) + \cplpot(F, \theta).
\end{equation}
\noindent\textbf{Dissipation potential:}
The \emph{dissipation functional} $\diss \colon \Yid \times \BBB H^1(\Omega;\R^d) \EEE \times  L^1_+(\Omega) \ee \to \R_+$ is defined as
\begin{equation}\label{dissipation}
  \diss( y, \BBB \tilde y, \EEE \theta)
  \defas \int_\Omega \disspot(\nabla  y,  \BBB  \nabla \tilde y, \EEE \theta) \di x,
\end{equation}
where $\disspot \colon \R^{d \times d} \times \R^{d \times d} \times \R_+ \to \R_+$ is the \emph{potential of dissipative forces} satisfying
\begin{enumerate}[label=(D.\arabic*)]
  \item \label{D_quadratic} $\disspot(F, \dot F, \theta) \defas \frac{1}{2} D(C, \theta)[\dot C, \dot C] \defas \frac{1}{2} \dot C : D(C, \theta) \dot C$, where $C \defas F^T F$, $\dot C \defas \dot F^T F + F^T \dot F$, and $D \in C(\R^{d \times d}_\sym \times \R_+; \R^{d \times d \times d \times d})$ with $D_{ijkl} = D_{jikl}= D_{klij}$ for $1 \le i,j,k,l \le d$;
  \item \label{D_bounds} $\frac{1}{\aC} \abs{\dot C}^2 \leq \dot C : D(C, \theta) \dot C \leq \aC \abs{\dot C}^2$ for all $C, \, \dot C \in \R^{d \times d}_\sym$ \RB and \EEE $\theta \geq 0$.
\end{enumerate}
Notice that Assumption \ref{D_quadratic} implies that the viscous stress $\pl_{\dot F} \disspot(F, \dot F, \theta)$ is linear in the time derivative $\dot C$ as well as \BBB (see e.g.~\cite[(2.8)]{tve}) \EEE
\begin{equation}\label{chain_rule_Fderiv}
  \partial_{\dot F} R(F, \dot F, \theta) = 2 F (D(C, \theta) \dot C).
\end{equation}
We also define the associated \emph{dissipation rate} $\drate \colon \R^{d \times d} \times \R^{d \times d} \times \R_+ \to \R_+$ as
\begin{equation}\label{diss_rate}
  \drate(F, \dot F, \theta)
  \defas \pl_{\dot F} \disspot(F, \dot F, \theta) : \dot F
  = 2 R(F, \dot F,\theta),
\end{equation}
where the second identity follows from \eqref{chain_rule_Fderiv} and Assumption \ref{D_quadratic}, \BBB see also \cite[(2.9)]{tve}. \EEE

\BBB
Below, for technical reasons explained in \eqref{expli2}, we will also consider a \emph{regularized version of the dissipation}  $\diss_\eps \colon \Yid \times  H^3(\Omega;\R^d)   \times  L^1_+(\Omega) \ee \to \R_+$, defined as
\begin{equation}\label{Repps}
  \diss_\eps( y, \BBB \tilde y, \EEE \theta)
  \defas \int_\Omega \disspot(\nabla  y,  \BBB  \nabla \tilde y, \EEE \theta) \di x + \frac{\eps}{2} \int_\Omega  |\nabla \Delta \tilde{y}|^2 \di x
\end{equation}
for a small regularization parameter $\eps>0$. \EEE

\ZZZ
\noindent\textbf{Heat conductivity:} \EEE
The map $\hc \colon \R_+ \to \R^{d \times d}_\sym$ \BBB denotes \EEE the \emph{heat conductivity tensor} of the material in the deformed configuration.
We require that $\hc$ is continuous, symmetric, uniformly positive definite, and bounded.
More precisely, for all $\theta \geq 0$ it holds that
\begin{equation}\label{spectrum_bound_K}
  \frac{1}{\aC} \leq \hc(\theta) \leq \aC,
\end{equation}
where the inequalities are meant in the eigenvalue sense.
We define the pull-back $\hcm \colon GL^+(d) \times \R_+ \to \R^{d \times d}_\sym$ of $\hc$ into the reference configuration by (see \cite[(2.24)]{tve_orig}\RB)\EEE
\begin{equation*}
  \hcm(F, \theta) \defas \det(F) F^{-1} \hc(\theta) F^{-T}.
\end{equation*}

\ZZZ
\noindent\textbf{Thermal energy and total internal energy:} \EEE Following \cite{tve, RBMFMKLM,  tve_orig}, the \textit{(thermal part of the) internal energy} $\inten \colon GL^+(d) \times (0, \infty) \to \R$  is defined as
\begin{equation}\label{Wint}
  \inten(F, \theta) \defas \cplpot(F, \theta) - \theta \pl_\theta \cplpot(F, \theta).
\end{equation}
Using \ref{C_zero_temperature} and the third bound in \ref{C_bounds}, we see that $\inten$ can be continuously extended to zero temperatures by setting $\inten(F, 0) = 0$ for all $F \in GL^+(d)$.
Furthermore, by the third bound in \ref{C_bounds} we have that  
\begin{align*}
  \partial_{\theta} \inten (F, \theta)
  = -\theta \ZZZ \pl_{\theta\theta}^2 \EEE \cplpot(F, \theta) \in [\aC^{-1}, \aC]
  \qquad \text{for all $F \in GL^+(d)$ and $\theta > 0$}.
\end{align*}
Along with \ref{C_zero_temperature} this shows that the internal energy is controlled by the temperature in the  \BBB sense that \EEE
\begin{equation}\label{inten_lipschitz_bounds}
  \frac{1}{\aC} \theta \leq \inten(F, \theta) \leq \aC \theta.
\end{equation}
Finally, we define \emph{total \BBB internal \EEE energy functional}   $\toten \colon \Yid \times L^1_+(\Omega) \to \R_+$ by
\begin{equation}\label{toten}
  \toten(y, \theta) \defas \mechen(y) + \mathcal{W}^{\rm in}(y,\theta) \qquad \text{ with } \mathcal{W}^{\rm in}(y,\theta) \defas \int_\Omega \inten(\nabla y, \theta) \di x.
\end{equation}
\BBB We remark that the above assumptions on the potentials  coincide with the ones in \cite[Section 2.1]{tve}, up to the fact that, differently to \cite[(2.4)]{tve}\RB, \BBB we only allow the potential of the strain-gradient energy to depend on the norm of \RB the diagonal \BBB $\Delta y$ of \RB $\nabla^2 y$\BBB. Moreover, for $d=3$, the range of   $p \in (3,6)$ is restricted as we need the Sobolev embedding $H^3(\Omega;\R^d) \subset \subset W^{2,p}(\Omega;\R^d)$.  Eventually, in contrast to \cite{tve}, in the definition of admissible deformations, see \eqref{def_Yid}, we need to impose Dirichlet conditions on the \emph{entire} boundary $\partial \Omega$ as this  allows us to apply elliptic regularity results.    We refer to \cite[Examples 2.4 and 2.5]{tve_orig} for a class of   potentials satisfying all assumptions   \EEE above.

\subsection{Equations of nonlinear thermoviscoelasticity with inertia: Existence of weak solutions}

\BBB Let \EEE $I \defas [0, T]$ where $T > 0$ denotes a time horizon, let $\rho > 0$ be a constant \emph{mass density} in the reference configuration, let $\kappa \geq 0$ be a constant \emph{heat-transfer coefficient}, let $f \in W^{1, 1}(I; L^2(\Omega; \R^d))$ be a time-dependent \emph{dead force}, and let $\bt \in  W^{1, 1}(I; L^2_+(\partial \Omega))$ be an \emph{external temperature}.  \BBB Moreover,  let $\eps \ge 0$ be a regularization parameter, where $\eps = 0$ corresponds to the setting without regularization.  \EEE In the strong form, we \BBB study the system \EEE
\begin{subequations}
    \label{e:strong}
\begin{align}
   \label{e:mechanical-strong}
   f & =  \rho \partial^{2}_{tt} y - {\rm div} \big( \partial_{F} W ( \nabla y, \theta) + \partial_{\dot{F}} R(\nabla y, \partial_{t} \nabla y, \theta)  \BBB- \nabla \EEE ({D H}(\Delta y)) \BBB +  \eps \partial_t \nabla \Deltatwo y\EEE  \big) \,,
   \\
   \label{e:thermal-strong}
   -\theta \partial^{2}_{\theta\theta} W^{\rm cpl} (\nabla y, \theta) \partial_{t} \theta & = {\rm div} \big( \mathcal{K} (\nabla y, \theta) \nabla \theta \big) + \xi (\nabla y, \partial_{t} \nabla y, \theta)  + \theta \partial^{2}_{F \theta} W^{\rm cpl} (\nabla y, \theta) : \partial_{t} \nabla y  \BBB +   \eps  |\partial_t \nabla \Delta y|^2\EEE \,,
\end{align}
\end{subequations}
coupled with the boundary conditions
\begin{subequations}
\label{e:strong-bdry}
\begin{align}
    y  = \id & \qquad \text{in $\RB I \EEE \times \partial\Om$},   \label{e:laplaceboundary} \\
    \label{e:laplaceboundary1.5}
    {D H}(\Delta y)   = 0 &  \qquad \text{in $\RB I \EEE \times \partial\Om$},\\
    \BBB  \eps \partial_\nu \Delta y = \eps \Delta^2 y   = 0 &  \qquad \text{in $\RB I \EEE \times \partial\Om$},  \label{e:laplaceboundary2}\\ \EEE
    \label{e:thermal-bdry} \mathcal{K}(\nabla y, \theta) \nabla \theta \cdot \nu + \kappa \theta  = \kappa \bt &  \qquad \text{in $\RB I \EEE\times\partial\Om$}
\end{align}
\end{subequations}
and subject to the initial conditions  
\begin{equation}\label{iiiniita}
   y(0) = y_0, \qquad \partial_t y(0) = y_0', \qquad \theta(0) = \theta_0,
\end{equation}
for initial values $y_0 \in \Yid$, $y_0' \in H^1_0(\Omega; \R^d)$, and $\theta_0 \in \BBB L^2_+ \EEE (\Omega)$. \BBB We refer  to \cite[Section 2]{tve_orig} for a thorough explanation of this model. We highlight  that, compared to   \cite{tve_orig}, we include inertial effects, i.e., the mechanical equation features the term $ \rho \partial^{2}_{tt} y$. Moreover, for $\eps>0$ there are regularizing terms both in \eqref{e:mechanical-strong} and \eqref{e:thermal-strong}, complemented with the additional natural boundary condition \eqref{e:laplaceboundary2}. In the regularized setting, we will assume stronger initial conditions for the deformations, namely $y_{0,\eps} \in    \Yidreg  $ and  $y'_{0,\eps}\in  H^{3}(\Om; \R^{d}) \cap H^1_0(\Omega;\R^d) $, where 
  \begin{equation}\label{eq: yidre}
 \Yidreg \defas \big\{ y \in \Yid \cap H^4(\Omega;\R^d) \colon   \partial_\nu \Delta  y(t) = \Delta  y(t)   = 0  \text{ $\haus^{d-1}$-a.e.~in } \partial \Omega\big\}.
  \end{equation} 
  We now first treat the case $\eps>0$ and afterwards we address the system without regularization. \EEE

\subsection*{Existence of weak solutions for \ZZZ the \EEE regularized system}
 We introduce the notion of weak solutions related to \eqref{e:mechanical-strong}--\eqref{e:thermal-bdry} for $\eps>0$. \EEE

\begin{definition}[\BBB Weak solutions to the regularized  thermo-elastodynamic system for viscous solids]
\label{def:solution-elastodyn-thermal} \EEE
Let $y_{0,\eps} \in    \Yidreg  $, $y'_{0,\eps}\in \BBB H^{3}(\Om; \R^{d}) \cap H^1_0(\Omega;\R^d) \EEE $, $\theta_{0} \in \BBB L^{2}_{+}(\Om) \EEE $, $f \in W^{1, 1}(I; L^{2}(\Om; \R^{d}))$, and $\bt \in  W^{1, 1}(I; L^2_+(\partial \Omega))$. We say that a pair $(y_\eps, \theta_\eps)$ with 
\begin{align*}
y_\eps &\in  L^{\infty}(\BBB I ;  \EEE  \Yid) \cap H^1\big(I;H^3(\Omega;\R^d)   \big) \BBB \cap H^2\big(I; (H^3(\Omega;\R^d) \cap H^1_0(\Omega;\R^d))^*\big) , \EEE  \\\  
 \theta_\eps & \in \BBB L^2(I; H^{1}_+(\Omega)) \EEE
 \end{align*}
is a solution \BBB to the regularized thermo-elastodynamic system \EEE  with initial conditions $(y_{0,\eps}, y'_{0,\eps}, \theta_{0})$ if  \BBB $y_\eps(0) = \ZZZ y_{0,\eps} \EEE $, $\partial_t y_\eps(0) = y_{0,\eps}'$,   the internal energy $w_\eps \defas W^{\rm in} (\nabla y_\eps , \theta_\eps )$ lies in $L^2(I; H^1(\Omega)) \cap H^1(I; (H^1(\Omega))^*)$ and satisfies $w_\eps(0) = w_{0,\eps} \defas W^{\rm in}(\nabla \ZZZ y_{0,\eps}, \EEE \theta_0)$, and \EEE the following equations are satisfied for every $z \in C^{\infty} (\RB I \EEE \times \overline{\Omega}; \R^{d})$ with $z = 0$ on $\RB I \EEE \times \partial \Om$   and for every $\varphi \in C^{\infty} (\RB I \EEE\times \overline\Om)$:
\begin{align}
  &\begin{aligned}
    0= & \BBB \int_I \EEE \int_{\Om} \partial_{F} W(\nabla y_\eps, \theta_\eps) : \nabla z \di x \di t
    +   \BBB \int_I \EEE \int_{\Om} {D H} ( \Delta  y_\eps) \cdot \Delta z \di x \di t  + \BBB  \eps \int_I   \int_{\Om} \partial_t \nabla \Delta  y_\eps : \nabla \Delta z \di x \di t \EEE \\ 
    & \quad + \BBB \int_I \EEE \int_{\Om} \partial_{\dot{F}} R(\nabla y_\eps, \partial_t \nabla y_\eps, \theta_\eps) : \nabla z \di x \di t  + \rho \BBB \int_I \BBB  \langle  \partial^2_{tt} y_\eps  , z \rangle \EEE \di t
    - \BBB \int_I \EEE  \int_{\Om} f \cdot z \di x \di t   \,, 
  \end{aligned} \label{mechanical_equation} \\
  &\begin{aligned}
    0 = &\int_I \int_\Omega \hcm(\nabla y_\eps, \theta_\eps) \nabla \theta_\eps \cdot \nabla \vphi \di x \di t
 + \int_I \langle \partial_t w_\eps, \vphi  \rangle \di t
    - \kappa \int_I \int_{\partial \Omega} (\bt - \theta_\eps) \vphi \di \haus^{d-1} \di t    \\
    &\quad    - \int_I \int_\Omega \Big( \drate(\nabla y_\eps, \partial_t \nabla y_\eps, \theta_\eps)
        ) 
        + \pl_F \cplpot(\nabla y_\eps, \theta_\eps) : \pl_t \nabla y_\eps  \BBB +  \eps  |\partial_t \nabla \Delta y_\eps|^2\EEE
      \Big) \vphi \di x \di t   ,  
  \end{aligned}\label{e:regularized-thermal}
\end{align}  
where \BBB $\langle \cdot , \cdot \rangle$ denotes the dual pairing of $H^3(\Omega;\R^d) \cap H^1_0(\Omega;\R^d)$ and its dual or of  $H^1(\Omega)$ and its dual, respectively.   
\end{definition}
\BBB Following \EEE the lines of~\cite[(2.28)--(2.29)]{tve_orig}, one can show \ZZZ that \eqref{e:mechanical-strong} \BBB together with \eqref{e:laplaceboundary}--\eqref{e:laplaceboundary2} \EEE is equivalent to \eqref{mechanical_equation}.   \BBB Besides the regularizing term, \EEE the only difference in~\eqref{e:mechanical-strong} \BBB compared to \cite{tve_orig} \EEE is the presence of the inertial term. Arguing as in~\cite[(2.16)--(2.17)]{tve_orig}, we  can  rewrite the heat-transfer  equation~\eqref{e:thermal-strong} in terms of the internal energy  \BBB $w_\eps$ as    \EEE
\begin{equation*}
    \partial_{t} w_\eps = {\rm div} \big( \mathcal{K} (\nabla y_\eps, \theta_\eps) \nabla \theta_\eps \big) + \xi (\nabla y_\eps, \partial_{t} \nabla y_\eps, \theta_\eps) \BBB +\eps  |\partial_t \nabla \Delta y_\eps|^2\EEE +  \partial_{F} W^{\rm cpl} (\nabla y_\eps, \theta_\eps) : \partial_{t} \nabla y_\eps,
\end{equation*}
where   we have used  \BBB \eqref{Wint} and the identity \EEE 
\begin{align*}
\partial_{t} w_\eps = \partial_{F} W^{\rm cpl} (\nabla y_\eps, \theta_\eps) : \partial_{t} \nabla y_\eps  - \theta_\eps  \BBB \partial_{F\theta}^2 \EEE  W^{\rm cpl} (\nabla y_\eps, \theta_\eps) : \partial_{t} \nabla y_\eps - \theta_\eps \ZZZ \partial^{2}_{\theta\theta} \EEE W^{\rm cpl} (\nabla y_\eps, \theta_\eps) \partial_{t} \theta_\eps\,. 
\end{align*}
\BBB Taking also \eqref{e:thermal-bdry} into account, this yields \RB the \BBB weak formulation \eqref{e:regularized-thermal}. \EEE

The first main results of the paper read as follows.

\begin{theorem}[\BBB Existence of weak solutions to the regularized system\EEE]
    \label{thm:main-thermal-elasto-regu}
       Let $p \in (2, +\infty)$ if $d=2$ or $p \in (3, 6)$ for $d=3$. \BBB Assume that  \ref{W_regularity}--\ref{W_lower_bound}, \ref{H_regularity}--\ref{H_bounds},~\ref{C_regularity}--\ref{C_bounds},~\ref{D_quadratic}--\ref{D_bounds}, and~\eqref{spectrum_bound_K} hold true. Let $\eps>0$. \EEE   Let $y_{0,\eps} \in \Yidreg$, $y'_{0,\eps}\in H^{3}(\Om; \R^{d}) \cap H^{1}_0(\Om; \R^{d})$, $\theta_{0} \in \ZZZ L^{2}_{+}(\Om) \EEE $, $f \in W^{1, 1}(\ZZZ I; \EEE L^{2}(\Om; \R^{d}))$, and $\bt \in  W^{1, 1}(\ZZZ I; \EEE L^2_+(\partial \Omega))$. Then, there exists a \BBB weak solution \EEE  $(y_\eps, \theta_\eps)$  to the \ste regularized \EEE thermo-elastodynamic system with initial data $(y_{0,\eps}, y'_{0,\eps}, \theta_{0})$ in the sense of Definition~\ref{def:solution-elastodyn-thermal}.  
\end{theorem}

\BBB For weak solutions, we can derive  energy balances and some regularity properties. Recall \eqref{mechanical} and \eqref{toten}. \EEE

\BBB 
\begin{theorem}[Regularity of solutions and total energy balance]  \label{thm:main-thermal-elasto-regu2}
  In the setting of Theorem \ref{thm:main-thermal-elasto-regu}, we find weak solutions $(y_\eps, \theta_\eps)$ \EEE  satisfying  $\ste y_{\eps} \EEE \in L^\infty(I; \Yidreg)$. \MMM  Moreover, for each $t \in I$,  the system satisfies 
the mechanical energy balance 
\begin{align}
\label{e:limit-energy-2-NNN-for result }
  \mathcal{M}(y_\eps(t))  &   +  \frac{\rho}{2}   \Vert \partial_t y_\eps(t) \Vert^2_{L^2(\Omega)}    + \int_0^t  \int_\Omega \Big(   \drate(\nabla y_\eps, \partial_t \nabla y_\eps, \theta_\eps)
        )   +  \eps  |\partial_t \nabla \Delta y_\eps|^2  \di x  +  \pl_{F} W^{\rm cpl} (\nabla y_{\eps}, \theta_{\eps}) 
         : \partial_{t}\nabla y_{\eps}   \Big)     \di s   \nonumber
          \\
          & =  \mechen(y_{0,\eps})   +   \frac{\rho}{2}   \Vert y_{0,\eps}' \Vert^2_{L^2(\Omega)} 
                      +  \int_0^t \int_\Omega f \cdot \partial_{t} y_{\eps}   \di x \di s ,
\end{align}
the \ZZZ thermal \EEE energy balance 
\begin{align}\label{internal energy balance-for result }
  \int_\Omega w_\eps(t)  \di x  &= \int_\Omega w_{0,\eps}   \di x  +   \int_0^t \int_\Omega \Big( \drate(\nabla y_\eps, \partial_t \nabla y_\eps, \theta_\eps)
          +  \eps  |\partial_t \nabla \Delta y_\eps|^2       + \pl_F \cplpot(\nabla y_\eps, \theta_\eps) : \pl_t \nabla y_\eps       \Big)   \di x \ZZZ \di s \EEE \\
    &\quad   + \kappa \int_0^t \int_{\partial \Omega} (\bt - \theta_\eps)   \di \haus^{d-1} \ZZZ \di s \EEE
      \,, \nonumber
\end{align}
and the total energy balance
\begin{align}\label{eq:energy-old}
\begin{aligned}
  &\toten(y_\eps(t), \theta_\eps(t)) +  \frac{\rho}{2}   \Vert \partial_t y_\eps(t) \Vert^2_{L^2(\Omega)} \\
  &\quad=   \toten(y_{0,\eps},\theta_0)  +   \frac{\rho}{2}   \Vert y_{0,\eps}' \Vert^2_{L^2(\Omega)}  + \int_0^t \int_{\partial \Omega}\kappa(\bt-\theta_\eps)\, \BBB \di \mathcal{H}^{d-1}  \di s \EEE + \int_0^t\int_\Omega f\cdot \partial_t  y_\eps \di x \di s .
\end{aligned} 
\end{align}
\end{theorem}

We emphasize that the energy balances are well-defined pointwise for each $t \in I$ since the regularity of the deformation and the temperature imply $ y_\eps \in C(I;W^{2,p}(\Omega;\R^d))$, $\partial_t y_\eps \in C(I;L^2(\Omega;\R^d))$, and $ \ZZZ w_\eps \EEE \in C(I;L^2(\Omega))$, where we use that $p < 2^*$ and \cite[Lemma 7.3]{Roubicek-book}.  The total energy balance \eqref{eq:energy-old} arises by summing \eqref{e:limit-energy-2-NNN-for result } and \eqref{internal energy balance-for result }. In particular,  we observe that the system is closed for  $\kappa=0$ and $f = 0$. Still,  an exchange of mechanical energy and thermal energy   is possible due to the \ZZZ (regularized) \EEE \emph{dissipation rate} $\drate(\nabla y_\eps, \partial_t \nabla y_\eps, \theta_\eps) 
        )   +  \eps  |\partial_t \nabla \Delta y_\eps|^2$  and the \emph{adiabatic heat source} $ \pl_{F} W^{\rm cpl} (\nabla y_{\eps}, \theta_{\eps}) 
         : \partial_{t}\nabla y_{\eps} $ which cancel out in the summation of \eqref{e:limit-energy-2-NNN-for result } and \eqref{internal energy balance-for result }.

\subsection*{Existence of weak solutions for the system without regularization}

Our goal is to remove the regularization by passing to the limit $\eps \to 0$ for weak solutions $(y_\eps,\theta_\eps)$ in the sense of Definition \ref{def:solution-elastodyn-thermal}. Unfortunately, the available a priori bounds and compactness results yielding a limit $(y,\theta)$, see Lemma \ref{l:mu-compactness} below, are not strong enough as they guarantee convergence of all terms  in \eqref{mechanical_equation}--\eqref{e:regularized-thermal} except  for the acceleration $\partial^2_{tt} y_\eps $ in \eqref{mechanical_equation} and  the dissipation rate  $\drate(\nabla y_\eps, \partial_t \nabla y_\eps, \theta_\eps)$ in \eqref{e:regularized-thermal}. Accordingly, also the validity of the mechanical and thermal energy balances \eqref{e:limit-energy-2-NNN-for result }--\eqref{internal energy balance-for result } cannot be expected in the limit $\eps \to 0$ as  
$$\liminf_{\eps \to 0} \int_0^t \int_\Omega \drate(\nabla y_\eps, \partial_t \nabla y_\eps, \theta_\eps) \di x \di s  > \int_0^t \int_\Omega \drate(\nabla y, \partial_t \nabla y, \theta) \di x \di s      $$
is \ZZZ possible  under the available \EEE  compactness results. In \cite[Lemma~4.5]{tve} and \cite[Propositions~5.1 and~6.6]{tve_orig}, equality was guaranteed by a chain rule for the mechanical energy (see \cite[Proposition~3.6]{tve_orig}) which also allowed to derive a  mechanical energy balance. Due to the presence of the inertial term, it appears to be impossible to  adapt this strategy to the current setting.

We overcome this difficulty by appealing   to a weaker formulation of the mechanical and the heat-transfer equation.
 In \eqref{mechanical_equation}, it suffices to perform an integration by parts in time to deal with the term $\partial^2_{tt} y_\eps$.  The passage to a weaker form of \eqref{e:regularized-thermal} is based on the observation that the delicate dissipation term cancels in the summation of \eqref{e:limit-energy-2-NNN-for result }  and \eqref{internal energy balance-for result }. \ZZZ More precisely, this passage \EEE is achieved by testing  \eqref{mechanical_equation} with $z =  \partial_{t} y \varphi$ and adding the result to \eqref{e:regularized-thermal}. This leads to the following notion of weak solution whose form will be explained in more detail by a formal computation in~\eqref{e:thermal-strong-3}--\eqref{e:thermal-strong-7}   below.

\begin{definition}[\BBB Weak solutions to thermo-elastodynamic system for viscous solids]
\label{def:solution-elastodyn-thermal-unreg} \EEE
Let $y_{0} \in  \Yid$, $y'_{0}\in  H^{1}(\Om; \R^{d}) \EEE $, $\theta_{0} \in \BBB L^{2}_{+}(\Om) \EEE $, $f \in W^{1, 1}(\ste I \EEE ; L^{2}(\Om; \R^{d}))$, and $\bt \in  W^{1, 1}( \ste I \EEE ; L^2_+(\partial \Omega))$. We say that a pair $(y, \theta)$ with 
\begin{align*}
y \in  L^{\infty}(\BBB I ;  \EEE  \Yid) \cap H^1(I;H^1(\Omega;\R^d) ),  \quad  \theta  \in L^1(I; W^{1, 1}_+(\Omega))
 \end{align*}
is a solution \BBB to the thermo-elastodynamic system \EEE  with initial conditions $(y_0, y'_0, \theta_{0})$ if the following equations are satisfied for every $z \in C^{\infty} (\RB I \EEE\times \overline{\Omega}; \R^{d})$ with $z = 0$ on $\RB I \EEE\times \partial \Om$ \BBB and $z(T) = 0$, \EEE  and for every $\varphi \in C^{\infty} (\RB I \EEE\times \overline\Om)$ with $\varphi(T) = 0$:
\begin{align} \label{mechanical_equation_final-def-unreg}
    0= & \BBB \int_I \EEE \int_{\Om} \partial_{F} W(\nabla y, \theta) : \nabla z \di x \di t
    +   \BBB \int_I \EEE \int_{\Om} {D H} ( \Delta  y) \cdot \Delta z \di x \di t  
   \\
    &      \quad + \BBB \int_I \EEE \int_{\Om} \partial_{\dot{F}} R(\nabla y, \partial_t \nabla y, \theta) : \nabla z \di x \di t  - \rho \BBB \int_I \EEE \int_{\Om} \partial_t y \cdot \partial_t z \di x \di t
    - \BBB \int_I \EEE  \int_{\Om} f \cdot z \di x \di t - \rho \int_{\Om} y'_{0} \cdot z(0) \di x \,, \notag
  \end{align}
  \begin{align} \label{e:new-thermal-equation-lim}
  0= &  \intQ \hcm(\nabla y, \theta) \nabla \theta \cdot \nabla \varphi  \di x \di t  - \kappa \ZZZ \int_I \EEE \int_{\partial \Omega} (\bt-\theta  ) \varphi \di \haus^{d-1} \di t 
  -  \intQ \varphi \, f \cdot \partial_{t} y \di x \di t \notag \\ &
    - \ZZZ \int_I \EEE \int_\Omega \big( \elpot (\nabla y )  + H(\Delta y ) + w  \BBB + \frac{\rho}{2} |\partial_t y|^2  \EEE\big) \partial_{t} \varphi \di x \di t
 - \int_{\Om} \big( \elpot (\nabla y_{0}) + H(\Delta y_{0} ) + w_0 + \frac{\rho}{2} |y_0'|^2 \big) \varphi (0) \sa \di x 
 \notag\\
 &
\quad +\intQ 
      \Big( \pl_F  W (\nabla y , \theta ) + \pl_{\dot{F}} R (\nabla y , \nabla \partial_{t} y , \theta ) \Big) 
     : ( \partial_{t} y \otimes \nabla \varphi) \di x \di t \notag \\
&
\quad  - \intQ D{H}(\Delta y) \cdot \partial_{t} y \Delta \varphi \, \di x \di t - 2 \intQ \nabla (D{H}(\Delta y )) : ( \partial_{t}y  \otimes \nabla \varphi) \di x \di t  ,
  \end{align}
  where \BBB for shorthand we set \EEE  $w \defas W^{\rm in} (\nabla y , \theta )$ \BBB and $w_0 \defas W^{\rm in} (\nabla y_0 , \theta_0 )$. \EEE 
\end{definition}

 \BBB
 In particular, we observe that, due to the lack of   regularity of $\partial^2_{tt} y$ and  $\partial_{t} w$,  the initial conditions of   $\partial_t  y $ and $w$  are  given implicitly in a weak form, relying on an integration by parts in time. An important aspect of the weak formulation \eqref{e:new-thermal-equation-lim} is that it directly guarantees the total energy balance. Indeed, for each $t \in I$ such that 
\begin{align*}
&  \lim_{\delta \to 0} \RB\mint_{t-\delta}^{t+\delta}\BBB  \int_\Omega y   \di x \di s =   \int_\Omega y(t,x)  \di x \in W^{2,p}(\Omega;\R^d), \quad  \lim_{\delta \to 0} \RB\mint_{t-\delta}^{t+\delta}\BBB  \int_\Omega  \partial_t y   \di x \di s =   \int_\Omega \partial_t y(t,x)  \di x \in L^2(\Omega;\R^d),  \\
&\lim_{\delta \to 0} \RB\mint_{t-\delta}^{t+\delta}\BBB  \int_\Omega w   \di x \di s =   \int_\Omega w(t,x)  \di x \in L^1(\Omega)
\end{align*}
(and thus for a.e.\ $t \in I$),   \BBB we can  test \eqref{e:new-thermal-equation-lim}  
with  $\varphi$ given by $\varphi \equiv 1$ on $(0,t-\delta)$, $\varphi \equiv 0$ on $(t+\delta,T)$ and $\varphi' \equiv -\frac{1}{2\delta}$ on $(t-\delta, t + \delta)$. In the limit $\delta \to 0$, \ZZZ after rearranegment, \EEE this yields
\begin{align}\label{final energy balance}
  \int_\Omega & \big( \elpot (\nabla y(t) )  + H(\Delta y(t) ) + w(t)   + \frac{\rho}{2} |\partial_t y(t)|^2  \EEE\big)  \di x \notag   \\
 & =  \ZZZ  \int_{\Om} \big( \elpot (\nabla y_{0}) + H(\Delta y_{0} ) + w_0 + \frac{\rho}{2} |y_0'|^2 \big)    \di x \EEE +  \kappa \int_0^t \int_{\partial \Omega} (\bt-\theta  )  \di \haus^{d-1} \di s 
  +  \int_0^t \int_\Omega  f \cdot \partial_{t} y \di x \di s  
 \end{align}
which is exactly the total energy balance, cf.\ also \eqref{eq:energy-old}.  \ZZZ Whereas  a total energy balance  still holds,   the \EEE respective form of \eqref{e:limit-energy-2-NNN-for result } and \eqref{internal energy balance-for result } \ZZZ may \EEE become {\em inequalities}. In some sense, \ZZZ this weaker form based on a \EEE replacement is inspired \ZZZ by \EEE fluid-mechanics for compressible heat conduction fluids, where the conservation of the total energy is guaranteed by transferring the heat equation into an inequality for the entropy~\cite{feireisl1,feireisl2}.

\begin{theorem}[\BBB Existence \seb{and regularity} of weak solutions\EEE]
    \label{thm:main-thermal-elasto-unregu}
       Let $p \in (2, +\infty)$ if $d=2$ or $p \in (3, 6)$ for $d=3$. \BBB Assume that  \ref{W_regularity}--\ref{W_lower_bound}, \ref{H_regularity}--\ref{H_bounds},~\ref{C_regularity}--\ref{C_bounds},~\ref{D_quadratic}--\ref{D_bounds}, and~\eqref{spectrum_bound_K} hold true. \EEE   Let $y_{0} \in \Yid$, $y'_{0}\in H^{1}_{0}(\Om; \R^{d})$, $\theta_{0} \in \ZZZ L^{2}_{+}(\Om) \EEE $, $f \in W^{1, 1}(\RB I \EEE; L^{2}(\Om; \R^{d}))$, and $\bt \in  W^{1, 1}(\RB I \EEE; L^2_+(\partial \Omega))$. Then, there exists a \BBB weak solution \EEE  $(y, \theta)$   to the thermo-elastodynamic system with initial data $(y_{0}, y'_{0}, \theta_{0})$ in the sense of Definition~\ref{def:solution-elastodyn-thermal-unreg}.   \BBB The weak solution \EEE satisfies   $y\in L^2(I; H^3(\Omega;\R^d))$ and     $(1+\abs{\Delta y})^\frac{p-2}{2}\abs{\nabla \Delta y}^2\in L^2(I\times \Omega)$. 
\end{theorem}

Note that \eqref{final energy balance} holds and \eqref{e:laplaceboundary1.5} is satisfied in the sense of traces.

\subsection*{Formal derivation of the weak formulation}\label{formal deriv}
Let us close this section with a formal derivation of  equation \eqref{e:new-thermal-equation-lim}  \BBB which  will be made precise below in  Proposition~\ref{p:new_formulation}  for the regularized system.  
  Assuming sufficient regularity for $y$ and $\theta$, let us check that  the formulations in \eqref{e:new-thermal-equation-lim} and \eqref{e:regularized-thermal} (for $\eps = 0$) coincide.  First, an integration by parts in time shows that\RB, for $\eps = 0$, \EEE \eqref{e:regularized-thermal} is equivalent to \EEE
\begin{align}
\begin{aligned}
  0 &=   \intQ \mathcal{K} (\nabla y, \theta) \nabla \theta \cdot \nabla \varphi \, \di x \di t    - \intQ \big( \xi (\nabla y, \partial_{t} \nabla y, \theta) +  \partial_{F} W^{\rm cpl} (\nabla y, \theta) : \partial_{t} \nabla y\big) \varphi \, \di x \di t
    \\
    &\phantom{=}\quad 
    + \kappa \int_{\RB I} \int_{\partial\Om} (\theta - \bt ) \varphi \, \di \mathcal{H}^{d-1} \di t - \intQ w \partial_{t} \varphi  \di x \di t 
 - \int_{\Om} \BBB w_0 \EEE \varphi(0)  \di x  
\end{aligned}  
\label{e:thermal-strong-3}
\end{align}
for every $\varphi \in C^{\infty} ( \ste I \EEE \times \overline{\Om})$ with $\varphi(T) = 0$. Now, we test~\eqref{mechanical_equation} with~$z\defas \partial_{t} y \varphi$  \BBB for  $\varphi \in C^{\infty} (I\times \overline{\Om})$ with $\varphi(T) = 0$.  Using \eqref{eq: free energy}, \eqref{diss_rate}, expanding $\nabla ( \partial_{t} y \varphi)$,  and \EEE rearranging the terms, we obtain
\begin{equation}\label{e:thermal-strong-4}
  \begin{aligned}
     &- \intQ \xi (\nabla y, \partial_{t} \nabla y, \theta) \varphi \di x \di t - \intQ  \partial_{F} \cplpot(\nabla y, \theta) : \partial_{t} \nabla y \varphi \di x \di t 
     \\
     &\quad= \intQ  \partial_{F} \elpot(\nabla y, \theta) : \partial_{t} \nabla y \varphi \di x \di t 
     + \intQ  \partial_{F} W(\nabla y, \theta) : (\partial_{t} y \otimes \nabla \varphi) \di x \di t
     \\
     &\phantom{=}\quad + \intQ \partial_{\dot{F}} R(\nabla y, \partial_{t} \nabla y, \theta) : (\partial_{t} y \otimes \nabla \varphi) \di x \di t 
     + \intQ {D H}(\Delta y) \cdot \Delta (\partial_{t} y \varphi) \di x \di t
     \\
     &\phantom{=}\quad  - \intQ f \cdot \partial_{t} y \varphi \di x \di t   \BBB + \rho  \int_I  \int_\Omega \varphi  \partial^2_{tt} y  \cdot \partial_t y \di x \EEE \di t.
  \end{aligned}
\end{equation}
\BBB By the chain rule, the fundamental theorem of calculus, and  $\varphi(T) = 0$ we get 
\begin{align}
\label{e:thermal-strong-5}
\int_I  \int_\Omega  \varphi  \partial^2_{tt} y  \cdot \partial_t y  \di x \EEE \di t & = \frac{1}{2}\int_I  \int_\Omega \frac{\di}{\di t} \big( \varphi  | \partial_t y|^2 \big)\di x \EEE \di t - \frac{1}{2}\int_I  \int_\Omega \partial_t \varphi  | \partial_t y |^2 \di x \EEE \di t \notag \\ &
=   - \frac{1}{2} \int_{\Om} |y'_{0}|^{2} \varphi (0) - \frac{1}{2}\int_I  \int_\Omega \partial_t \varphi  | \partial_t y |^2 \di x \EEE \di t.
\end{align} \EEE
Moreover, by integration by parts in $\Om$ and since $\partial_{t} y = 0$ in $\RB I \EEE \times \partial \Om$ (recall that $y (t) \in \Yid$), we have that
\begin{align}
\label{e:thermal-strong-6}
    &\intQ {D H}(\Delta y) \cdot \Delta (\partial_{t} y \varphi) \di x \di t  = \int_I  \int_\Omega {D H}(\Delta y) : \big( \varphi \partial_{t}\Delta  y + 2 \partial_{t} \nabla y \nabla \varphi + \partial_{t} y \Delta \varphi \big) \di x \di t 
    \\
    &\quad
    = \RB\int_I \int_\Omega\EEE {D H}(\Delta y) :   \partial_{t}\Delta  y \varphi    \di x \di t - \RB\int_I \int_\Omega\EEE {D H}(\Delta y) :  \partial_{t} y \Delta \varphi  \di x \di t    - 2 \intQ \nabla ({D H}(\Delta y)) : (\partial_{t} y \otimes \nabla \varphi ) \di x \di t . \nonumber
\end{align}
\ZZZ Eventually, by  \EEE the \sa chain rule and by integration by parts with $\varphi(T) = 0$  we get \BBB
\begin{align}\label{e:thermal-strong-7}
&\intQ \varphi   \big( \partial_{F} W^{\rm el} (\nabla y) : \partial_{t} \nabla y  + {D H}(\Delta y) \cdot \partial_{t} \Delta y \big ) \di x \di t = \intQ \varphi  \frac{\di }{\di t}  \big( W^{\rm el} ( \nabla y) + H(\Delta y) \big) \di x \di t \notag \\ 
&\quad =  - \int_{\Om} \varphi(0) \big(  W^{\rm el} ( \nabla y_{0}) + H(\Delta y_{0}) \big)\di x -  \intQ \partial_{t} \varphi \big( W^{\rm el} ( \nabla y) + H(\Delta y) \big) \di x \di t \,.
\end{align} \EEE
Combining~\eqref{e:thermal-strong-3}--\eqref{e:thermal-strong-7} we infer~\eqref{e:new-thermal-equation-lim}.

\BBB
 
\subsection*{Outline} 
The rest of the paper is structured as follows.  In Section \ref{s:min-mov-scheme} we consider a  time-delayed parabolic system for a time-delay $h >0$ whose existence is established by a minimizing movement scheme with time discretization $\tau >0$. In Section \ref{sec:vanish_h} we pass to the limit $h \to 0$ and prove existence of solutions to the regularized system, see Theorem \ref{thm:main-thermal-elasto-regu}. Eventually in Section \ref{sec: epstozero} we pass to the limit $\eps \to 0$ and show Theorem \ref{thm:main-thermal-elasto-unregu}. Importantly, all relevant a priori estimates are established already on the level $\tau >0$ in Section \ref{s:min-mov-scheme} and immediately transfer to the limits \ZZZ $\tau\to 0$, \EEE $h\to 0$, and $\eps \to 0$.  \EEE

\section{Minimizing movements for time-delayed parabolic system}
\label{s:min-mov-scheme}

We \BBB fix \EEE  \ste a \EEE \emph{regularization parameter} $\eps \ee > 0$ and   a \emph{time-delay} $h > 0$. \BBB For   convenience, without further notice\RB, \BBB we assume that   $T / h \in \N$. In this section, we include the $\eps$-dependent regularizing term \ZZZ in \eqref{e:strong} \EEE and we suppose more regular initial conditions $y_0$, $y_0'$, denoted by $y_{0,\eps}$, $y_{0,\eps}'$. \EEE  
As   done in \cite{veaccel2},  by \EEE replacing the acceleration term $\rho \BBB \pl_{tt}^2 \EEE y$ by a discrete difference $\rho \frac{\pl_t y - \pl_t y(\cdot - h)}{h}$, we turn the hyperbolic problem \eqref{e:mechanical-strong} \ee into a parabolic one. \BBB   The main goal of this section is to prove  \ZZZ  the following existence result for the resulting problem,   where for convenience we use the notation $I_h\defas [-h,T]$.  \EEE

\begin{theorem}[Weak solutions of the time-delayed regularized problem]
\label{thm:weak_sol_time_delayed}
  Let $T, \, h , \, \eps> 0$, $y_{0,\eps} \in \BBB \Yidreg \EEE $, $y_{0,\eps}' \in  H^{ 3}(\Om; \R^{d}) \cap H^1_{0}(\Omega; \R^d) $, and $\theta_0 \in \BBB L^2_+(\Omega)$. \EEE
  Then, there exist $\BBB y_h  \in L^\infty( \BBB I; \EEE \Yidreg) \cap H^1(I_h; H^3(\Omega; \R^d))$ \EEE with $y_h(t) = y_{0,\eps} + t y_{0,\eps}'$ for all $t \in [-h, 0]$ and \BBB  $\theta_h \in L^2(I; H^1_+(\Omega))  $ such that     $w_h  \defas \inten(\nabla y_h, \theta_h) \in L^2(I; H^1(\Omega)) \cap H^1(I; (H^1(\Omega))^*)$ with  $\ZZZ w_h(0) = \EEE w_{0,\eps} \defas \inten(\nabla y_{0,\eps}, \theta_0)$ and \EEE  the following holds true:
  For all $z \in C^\infty(\BBB I \EEE \times \overline \Omega; \R^d)$ satisfying $z = 0$ on $\BBB I \EEE \times \partial \Omega$ we have
  \begin{subequations}
  \begin{equation}\label{weak_sol_time_del_y}
  \begin{aligned}
    &\intQ
      \Big(
        \pl_F \felpot(\nabla y_h, \theta_h)
        + \pl_{\dot F} \disspot(\nabla y_h, \partial_t \nabla y_h, \theta_h)
      \Big) : \nabla z + D\hypot(\Delta y_h) \cdot \Delta z 
      + \eps  \partial_t  \nabla \Delta {y}_h   :  \nabla \Delta z \ee \di x \di t \\
    &\quad= \intQ f \cdot z \di x \di t
      - \frac{\rho}{h}  \BBB \int_I \EEE \int_\Omega (\ee\partial_t y_h(t) - \partial_t y_h(t - h)) \cdot \ee z  \di x \ee \di t,
  \end{aligned}
  \end{equation}
  and for all $\vphi \in C^\infty(\BBB I \EEE \times \overline \Omega)$  \BBB it \EEE holds that
  \begin{equation}
  \label{weak_sol_time_del_theta}
  \begin{aligned}
    &\BBB \int_I \EEE \int_\Omega \hcm(\nabla y_h, \theta_h) \nabla \theta_h \cdot \nabla \vphi
      \BBB -\left(
        \drate(\nabla y_h, \partial_t \nabla y_h, \theta_h)
                + \pl_F \cplpot(\nabla y_h, \theta_h) : \pl_t \nabla y_h
      \right) \vphi \di x \di t \EEE  \\
    &\quad   - \eps \int_{0}^{T} \int_{\Om} |\rb \partial_{t} \nabla \Delta \ZZZ {y}_h \EEE \ee|^{2} \varphi \di x \di t  \BBB +  \int_I \langle \partial_t w_h , \vphi  \rangle  \EEE  \di t
    = \kappa \BBB \int_I \EEE \int_{\partial \Omega} (\bt - \theta_h) \vphi \di \haus^{d-1} \di t ,     
  \end{aligned}
  \end{equation}
  \end{subequations}
 \BBB  where $\langle \cdot, \cdot \rangle $ denotes the dual pairing between $H^1(\Omega)$ and $(H^1(\Omega))^*$.\EEE 
   \end{theorem}

A similar notion of weak solutions \BBB has already been \EEE considered in \cite{tve_orig,tve}. The main differences of the above equations \eqref{weak_sol_time_del_y}--\eqref{weak_sol_time_del_theta} to, e.g.,~\cite[(2.19)-(2.20)]{tve} is the presence of the additional $h$-dependent terms arising from the time-discretization of the acceleration  as well as the regularizing  terms \BBB depending on $\eps$, which induce better regularity properties of the solutions.  The solutions also depend on  $\eps$, which we  however do not include in the notation for simplicity.  \EEE The proof of existence follows along the lines of the reasoning from \cite[Sections 3 and 4]{tve}  \BBB and is based on a minimizing movement scheme. \EEE To keep the presentation concise, \BBB many  proofs in this section will only be sketched, \EEE highlighting primarily the differences to the \BBB arguments \EEE  in  \cite{tve}.

\BBB 

\subsection{Staggered minimizing movement scheme and its well-definedness}\label{MMscheme}

We \EEE introduce a \emph{discrete time-step} $\tau \in (0,h)$ and without further notice we assume that $h/\tau \in \N$. This also implies $T/\tau \in \N$. \EEE If not stated otherwise, all constants encountered in this section are independent of $\tau$, $h$, and $\eps$. Given any sequence $(a_k)_{\BBB k \in \Z}$, \BBB we \EEE introduce the   notation for discrete differences \BBB as \EEE
\begin{equation}\label{difff}
  \ddif a_k \defas \frac{a_k - a_{k-1}}{\tau}, \qquad \BBB k \in \Z. \EEE
\end{equation}
Theorem \ref{thm:weak_sol_time_delayed} will be shown via a \emph{staggered} \ee minimizing movements scheme.
 Let $\ste y_{0,\eps}  \in \Yidreg \EEE $,  $\BBB y_{0,\eps}' \EEE \in   H^3(\Omega; \R^d) \cap   H^1_0(\Omega;\R^d)$, and $\theta_0 \in \BBB L^2_+(\Omega)$. \EEE  We \BBB first define the initial conditions of the scheme by \EEE
\begin{equation*}
\BBB y_{\tau}^{(k)}  \EEE \defas y_{0,\eps} + k \tau y_{0,\eps}' \quad \text{ for } k \in \{-h/\tau, \, \ldots, \, 0\} \quad \text{ and } \quad    \BBB   \theta_{\tau}^{(0)}  \EEE \defas \theta_0.
\end{equation*}
\BBB Note that the time-discrete deformation is also defined for negative times, which will allow us to prove that the    solution \ZZZ $y_h$ \EEE in \eqref{weak_sol_time_del_y}--\eqref{weak_sol_time_del_theta} satisfies  $\ZZZ y_h(t) \EEE = y_{0,\eps} + t y_{0,\eps}'$ for all $t \in [-h, 0]$. \EEE

\BBB Now, suppose \EEE that for $k \in \{1, \ldots, T / \tau\}$ we have already constructed \BBB $(y_\tau^{(0)}, \theta_\tau^{(0)}), \ldots, (y_\tau^{(k-1)},\theta_\tau^{(k-1)})$. (The solutions also depend on $h$ and $\eps$, which we do not include in the notation for simplicity.)  Let $\BBB f_\tau^{(k)} \EEE \defas \mint_{(k-1)\tau}^{k\tau} f(t) \di t \defas \tau^{-1} \int_{(k-1)\tau}^{k\tau} f(t) \di t$, \EEE and \BBB for shorthand we \ste denote by $(\cdot, \cdot)_{2}$ the scalar product in $L^{2}(\Om; \R^{d})$. \EEE
\BBB Recalling also \eqref{mechanical}, \eqref{koppl}, and \eqref{dissipation}, \EEE the next deformation \ZZZ $y_\tau^{(k)}$ \EEE is  defined \BBB as \EEE a solution of the minimization problem
\begin{equation}\label{mechanical_step}
\begin{aligned}
  \min_{y \in \Yid \cap H^{ 3} (\Om; \R^{d})} & \Bigg\{
    \mechen(y)  \ + \cplen\big(y, \tsts{k-1})
    + \frac{1}{\tau} \diss\big(\ysts{k-1}, y - \ysts{k-1}, \tsts{k-1}\big)
     \\
    &  \ \   + \frac{\eps}{2\tau} \|\rb \nabla \Delta y - \nabla\Delta \ysts{k-1} \ee\|_{L^{2}(\Om)}^{2}   
     - \BBB  ( f_\tau^{(k)} , y )_2 \EEE
 +  \frac{\rho \tau}{2 h} \Big\Vert \frac{y - \BBB y_{\tau}^{(k-1)} \EEE }{\tau} - \ddif \BBB y^{(k - h/\tau)}_\tau \EEE \ee \Big\Vert_{L^2(\Omega)}^2
  \Bigg\}.
\end{aligned}
\end{equation}
Supposing that \BBB $y_\tau^{(k)}$ \EEE exists, we define $\tsts{k}$ as a solution \ste to \EEE the minimization problem
\begin{align}\label{thermal_step}
  &\min_{\theta \in H^1_+(\Omega)} \Bigg\{
    \int_\Omega \int_0^\theta \frac{1}{\tau}\left(
      \inten(\nabla \ysts{k},s)
      - \inten(\nabla \ysts{k-1}, \tsts{k-1})
    \right) \di s \di x
    + \frac{1}{2} \int_\Omega \hcm(\nabla \ysts{k-1}, \tsts{k-1}) \nabla \theta \cdot \nabla \theta \di x \notag \\
    &\quad-\int_\Omega \left(
      \pl_F \cplpot(\nabla \ysts{k-1}, \tsts{k-1}) : \ddif \nabla \ysts k
      + \drate(\nabla \ysts{k-1}, \ddif \nabla \ysts k, \tsts{k-1}) +   \eps | \delta_{\tau} \nabla \Delta \ZZZ y_{\tau}^{(k)}  \EEE |^{2}  \wedge \tau^{-1} 
    \right) \theta \di x \notag\\
    &\qquad    + \frac{\kappa}{2} \int_{{\partial \Omega}} (\theta - \btst{k})^2 \di \haus^{d-1}
  \Bigg\}, 
\end{align}
where $\btst k \defas \mint_{(k-1)\tau}^{k\tau} \bt(t) \di t$.

\BBB The minimization problem \eqref{mechanical_step} differs from the one used in \cite{veaccel2} due to the presence of the additional $\eps$-regularizing term and a different discretization of the acceleration term. In \cite[Definition 3.3, Theorem~3.5]{veaccel2}, the term 
$${ \frac{\rho \tau}{2 h} \Big\Vert \frac{y - \BBB y_{\tau}^{(k-1)} \EEE }{\tau} - \RB\mint_{(k-1)\tau}^{k\tau}\BBB v \di t  \EEE \ee  \Big\Vert_{L^2(\Omega)}^2}$$
for a \RB generic \BBB $v \in L^2(0,h)$ is used. By replacing $v$ with  $\partial_t y(\cdot - h)$, one can then construct  a $\tau$-discretized solution in the small time interval $(0,h)$. Then, a successive repetition of the argument  yields  a time-discrete solution on $[0,T]$. Here, instead, we simply  use the discretized solution $\ddif y^{(k - h/\tau)}_\tau $   which directly allows us to construct  time-discrete solutions in the entire time horizon $[0,T]$.  The minimization problem \eqref{thermal_step} coincides with the thermal step used in \cite{tve}, except for the regularizing term $\eps | \delta_{\tau} \nabla \Delta \ZZZ y_{\tau}^{(k)} \EEE |^{2} \wedge \tau^{-1}$. In this context, the truncation by $\tau^{-1}$ is necessary to guarantee well-posedness of the problem, see Proposition \ref{prop:existence_thermal_step} below.  \EEE

We  now show the well-definedness of the above minimization problems.
In this regard, the following \BBB  properties \EEE  of the mechanical energy are useful.
\begin{lemma}[\BBB Coercivity  of $\mechen$\EEE]\label{lem:pos_det}
  Given $M > 0$ there exists a constant $C_M > 0$ such that for all $y \in \Yid$ with $\mechen(y) \leq M$ it holds that
  \begin{align}\label{pos_det}
    \norm{y}_{W^{2, p}(\Omega)} &\leq C_M, &
    \norm{y}_{C^{1, 1-d/p}(\Omega)} &\leq C_M, &
    \norm{(\nabla y)^{-1}}_{C^{1 - d/p}(\Omega)} &\leq C_M, &
    \det(\nabla y) \geq \frac{1}{ C_M} \text{ in } \Omega.
  \end{align}
\end{lemma}

\begin{proof}
  By the definition of $\hyen$, the first inequality in \ref{H_bounds} and \ref{W_lower_bound} we see that
  \begin{equation*}
    - C_0 |\Omega| + \int_\Omega |\Delta y|^p \di x \leq \mechen(y) \leq M,
  \end{equation*}
  and hence
  \begin{equation*}
    \|\Delta y\|_{L^p(\Omega)}^p  \leq M + C_0 |\Omega| \eqqcolon \tilde M. 
  \end{equation*}
  As $y - \BBB \id \EEE \in W^{2, p}(\Omega; \R^d) \cap W^{1, p}_0(\Omega; \R^d)$ by the definition of $\Yid$, we then derive from \cite[Lemma 9.17]{gt} \ZZZ and the regularity of $\partial \Omega$ \EEE  that
  \begin{equation*}
    \|y \BBB - \id \EEE \|_{W^{2, p}(\Omega)} \leq C \|\Delta y\|_{L^p(\Omega)} \leq \BBB C \EEE \tilde M^{1/p}
  \end{equation*}
  for a constant $C$ independent of $M$. Our choice of \BBB the \EEE elastic potential $\elpot$ satisfies all \BBB assumptions imposed  in \EEE \cite{tve_orig}.   Consequently, \cite[Theorem 3.1]{tve_orig} applies which directly leads to \eqref{pos_det}.
\end{proof}

\begin{proposition}[Existence of the mechanical step]\label{prop:existence_mechanical_step}
   
  For any $M > 0$ there exists $\tau_0 \in (0, 1]$ such that for all $\tau \in (0, \tau_0)$ and $k \in \{1, \ldots, T/\tau\}$ the following holds:
  Let $\ysts{k-1} \in \Yid \cap H^3(\Omega; \R^d)$ satisfy $\mechen(\ysts{k-1}) \leq M$ and let $\tsts{k-1} \in H^1_+(\Omega)$.
  Then, the minimization problem \eqref{mechanical_step} attains a \BBB solution \EEE $\ysts{k} \in \Yid \cap H^{3}(\Om; \R^{d})$ solving \BBB the \EEE corresponding Euler-Lagrange equation, i.e., for all $z \in  H^3(\Omega; \R^d) \cap H^1_0(\Omega;\R^d)$ it holds  that   \BBB
  \begin{align}\label{mechanical_step_el}
    &\int_\Omega
      \left(
        \pl_F \felpot(\nabla \yst k, \tst{k-1})
        + \pl_{\dot F} \disspot(\nabla \yst{k-1}, \ddif \nabla\yst{k}, \tst{k-1})
      \right) : \nabla z
      +   D H  (\Delta \yst{k}) \cdot \Delta z
      + \eps  \ddif  \nabla\Delta y^{(k)}_{\tau} : \nabla \Delta z \ee \di x \notag \\
    &\quad=   \int_\Omega f_\tau^{(k)} \cdot  z  \di x  - \frac{\rho}{h} \int_\Omega (\ddif y_\tau^{(k)} - \ddif y_\tau^{(k - h/\tau)}) \cdot z \di x.
  \end{align}
\EEE
Moreover, there exists a constant $C_M \ZZZ >0 \EEE $, possibly depending on $M$, \BBB such that \EEE
  \begin{equation}\label{bad_mechen_bound_2}
    \begin{aligned}
  &\mechen(\ysts k)
    + \frac{\tau}{C_M} \norm{\ddif \nabla \ysts k}_{L^2(\Omega)}^2
    +  \frac{\eps \tau  }{2} \|  \ddif \nabla \Delta  \ysts k\ee\|_{L^{2}(\Om)}^{2}
 \leq
   \ZZZ C_M( 1 +  \Vert f \Vert_{L^2(I \times \Omega)}^2   ) \EEE
    + \frac{\rho \tau}{h} \Vert \ddif \yst{k - h/\tau} \Vert_{L^2(\Omega)}^2.
    \end{aligned}
  \end{equation} 
\end{proposition}

\begin{proof}
  The proof is similar to the one in \cite[Proposition 3.5]{tve}.
  We start by showing \BBB compactness. To this end, let \EEE  $(y_n)_n \subset \Yid \cap H^{ 3}(\Omega; \ZZZ \R^d) \EEE$ be a minimizing sequence for the minimization problem in \eqref{mechanical_step}. \BBB Using $y_{\tau}^{(k-1)}$ as a competitor  in  \eqref{mechanical_step}, we may suppose that each $y_n$ satisfies  
  \begin{align*}
    &\mechen(y_n) + \cplen(y_n, \theta_{\tau}^{(k-1)}) + \frac{1}{\tau} \diss( \ZZZ y_{\tau}^{(k-1)}, \EEE y_n - y_{\tau}^{(k-1)},  \ZZZ\theta_{\tau}^{(k-1)}) \EEE
      + \frac{\eps}{2\tau} \| \nabla \Delta y_{n} - \nabla \Delta y_{\tau}^{(k-1)}\ee\|_{L^{2}(\Om)}^{2}  \\
    &\quad -   ( f_\tau^{(k)} , y_n )_2    + \frac{\rho \tau}{2 h} \Big\Vert \frac{y_n - y_{\tau}^{(k-1)}}{\tau} - \ddif \yst{k - h/\tau} \Big\Vert_{L^2(\Omega)}^2    \\
    &\qquad \leq  \BBB \mechen(y_{\tau}^{(k-1)})   + \cplen(y_{\tau}^{(k-1)}, \theta_{\tau}^{(k-1)}) -   ( f_\tau^{(k)} , y_{\tau}^{(k-1)} )_2    + \frac{\rho \tau}{2 h} \Vert \ddif \yst{k - h/\tau} \Vert_{L^2(\Omega)}^2.
  \end{align*} \EEE
   As the mechanical energy $\mechen$ satisfies the same coercivity \BBB properties \EEE as the one in \cite{tve}, see Lemma \ref{lem:pos_det}, we can apply the generalized Korn's inequality in the form \cite[Corollary 3.4]{tve_orig} (see also \cite{pompe} \BBB for its original formulation). \EEE  Consequently, reasoning \BBB similarly \EEE to the proof of \cite[Proposition 3.5]{tve}  \BBB for the terms involving $\mechen$, $\cplen$, $\diss$, and $f_\tau^{(k)}$, see \cite[Equation (3.9)]{tve}, \EEE we can find $C_M > 0$ \BBB and $\tau_0 \in (0, 1)$, possibly depending on $M$, such that for $\tau \in (0,\tau_0) $ \EEE
  \begin{align*}
    &\BBB (1 - C_M \tau) \mechen(y_n) + \frac{1}{C_M \tau} \Vert \nabla y_n \BBB - \nabla y_{\tau}^{(k-1)} \EEE  \Vert_{L^2(\Omega)}^2 +   \frac{\eps}{2\tau } \| \nabla\Delta \EEE y_{n} - \nabla\Delta y_{\tau}^{(k-1)} \ee\|_{L^{2}(\Om)}^{2}
    \notag \\
    &\quad+ \frac{\rho \tau}{2 h} \Big\Vert \frac{y_n - y_{\tau}^{(k-1)}}{\tau} - \ddif \yst{k - h/\tau} \Big\Vert_{L^2(\Omega)}^2 \notag \\
    &\qquad\leq  (1 + C_M \tau) \BBB \mechen(y_{\tau}^{(k-1)})  + C_M \tau \big( 1 + \Vert f_\tau^{(k)} \Vert_{L^2(\Omega)}^2 \big) + \frac{\rho \tau}{2h} \Vert \ddif \yst{k - h/\tau} \Vert_{L^2(\Omega)}^2.
  \end{align*}
  By the definition of $\lst{k}$ and Jensen's inequality we see that
  \begin{equation*}
 \tau \|\fst{k}\|_{L^2(\Omega)}^2
    = \tau \int_\Omega \Big|\mint_{(k-1)\tau}^{k\tau} f \di t\Big|^2 \di x
    \leq \tau \int_\Omega \mint_{(k-1)\tau}^{k\tau} |f|^2 \di t \di x
    \leq \|f\|_{L^2(I \times \Omega)}^2.
  \end{equation*}  
\BBB Then, choosing $\tau_0$ small enough such that $C_M \tau_0 \le \frac{1}{2}$ and using   $\mechen(\ysts{k-1}) \leq M$,    we see that, after possibly increasing $C_M$, it follows   \EEE
  \begin{align}  \label{comparison_yn_2}
    &\mechen(y_n) + \frac{1}{C_M \tau} \Vert \nabla y_n- \nabla y_{\tau}^{(k-1)}    \Vert_{L^2(\Omega)}^2 + \frac{\eps}{2\tau} \| \nabla\Delta y_{n} - \nabla\Delta y_{\tau}^{(k-1)} \ee\|_{L^{2}(\Om)}^{2}
    + \frac{\rho \tau}{2 h} \Big\Vert \frac{y_n - y_{\tau}^{(k-1)}}{\tau} - \ddif \yst{k - h/\tau} \Big\Vert_{L^2(\Omega)}^2 \notag \\
    &\qquad\leq  \ZZZ C_M  (1 + \Vert f  \Vert_{L^2(I \times \Omega)}^2 ) \EEE
    + \frac{\rho \tau}{h} \Vert \ddif \yst{k - h/\tau} \Vert_{L^2(\Omega)}^2.
  \end{align} 
By Young's inequality with power $2$ and constant $\lambda \in (0, 1)$ we derive that
  \begin{align*}
    \|\nabla\Delta y_{n} - \nabla\Delta \ZZZ  y_{\tau}^{(k)} \EEE \|_{L^{2}(\Om)}^{2}
    &= \|\nabla\Delta y_{n}\|_{L^{2}(\Om)}^{2} - 2 \int_\Omega \nabla\Delta y_{n} : \nabla\Delta \ysts{k-1} \di x + \|\nabla\Delta \ysts{k-1}\|_{L^{2}(\Om)}^{2} \\ \BBB
    &\geq \BBB (1 - \lambda) \|\nabla\Delta y_{n}\|_{L^{2}(\Om)}^{2} - (-1 + 1 / \lambda) \|\nabla\Delta \ysts{k-1}\|_{L^{2}(\Om)}^{2}.
  \end{align*} \EEE
  Choosing $\lambda = 1/2$ above and \BBB combining \EEE with \eqref{comparison_yn_2} this leads to
  \begin{equation*}
    \mechen(y_n) + \BBB \frac{\eps}{4\tau} \EEE \|\Delta \nabla y_{n}\|_{L^{2}(\Om)}^{2}
    \leq \ZZZ
    + C_M  (1 + \Vert f  \Vert_{L^2(I \times \Omega)}^2) \EEE
    + \BBB \frac{\eps}{2\tau} \EEE \|\nabla\Delta \ysts{k-1}\|_{L^{2}(\Om)}^{2}
    + \frac{\rho \tau}{h} \Vert \ddif \yst{k - h/\tau} \Vert_{L^2(\Omega)}^2.
  \end{equation*} \EEE
\BBB This implies $\sup_{n \in \N}\|\nabla \Delta y_n\|_{L^2(\Omega)} < \infty$ and, \ZZZ in view of \ref{H_bounds}, \EEE in particular \EEE shows $\sup_{n \in \N} \| \Delta y_n \|_{H^1(\Omega)} < \infty$.   As $\Omega$ was assumed to have a   $C^5$-boundary   and $(y_n)_n \subset \Yid$, elliptic regularity \BBB for the operator $\Delta$ implies 
\begin{align*}
  \Vert y_n - \id \Vert_{H^3(\Omega)} \le C  \| \Delta y_n \|_{H^1(\Omega)} ,
  \end{align*}
   and thus \EEE $\sup_{n \in \N} \| y_n \|_{H^3(\Omega)} < \infty$.
  Consequently, \BBB as $p <2^*$, \EEE we can select a subsequence (without relabeling) such that $y_n \to y$ \BBB strongly \EEE  in $W^{2, p}(\Omega; \R^d)$ as well as $y_n \weakly y$ weakly in $H^3(\Omega; \R^d)$.  

  Existence of minimizers \BBB then \EEE  follows by standard lower semicontinuity arguments.  \BBB  Also, recalling \eqref{difff},  the derivation of the Euler-Lagrange equation  \eqref{mechanical_step_el} is standard, see the proof of \cite[Proposition 3.5]{tve} for some details. \EEE \BBB  Eventually,      estimate  \eqref{bad_mechen_bound_2} \EEE directly follows \BBB from  \eqref{comparison_yn_2}, \EEE  after passing to the limit $n \to \infty$ and using standard lower semicontinuity arguments.
\end{proof}

\begin{proposition}[Existence of the thermal step]\label{prop:existence_thermal_step}
   For any $M > 0$ there exists $\tau_0 \in (0, 1]$ such that for all $\tau \in (0, \tau_0)$ and $k \in \{1, \ldots, T/\tau\}$ the following holds:
  Let $\ysts{k-1}, \, \ysts{k} \in   \Yid$ \BBB be \EEE such that $\mechen(\yst{k-1}) \leq M$ and let $\tsts{k-1} \in H^1_+(\Omega)$.
  Then, the minimization problem \eqref{thermal_step} attains a solution $\BBB \theta_\tau^{(k)}  \EEE \in H^1_+(\Omega)$ solving \BBB the \EEE corresponding Euler-Lagrange equation, i.e.,  for all $\vphi \in H^1(\Omega)$ it holds that 
  \begin{align}\label{thermal_step_el}
    0 = & \int_\Omega  \ddif \wst{k} \ee \vphi \di x + \int_\Omega
    \hcm(\nabla \yst{k-1}, \tst{k-1}) \nabla \tst k \cdot \nabla \vphi \di x
    + \kappa \int_{\partial \Omega} (\tst k - \btst k) \vphi \di \haus^{d-1}  \\
        & - \int_\Omega \big(
    \pl_F \cplpot(\nabla \yst{k-1}, \tst{k-1}) : \ddif \nabla \yst k
    + \drate(\nabla \yst{k-1}, \ddif \nabla \yst k, \tst{k-1}) \BBB + \eps |\delta_{\tau}   \nabla \Delta y^{(k)}_{\tau} |^{2}  \wedge \tau^{-1} \big) \vphi \di x\,,  \notag
  \end{align}
  where we \RB shortly write \BBB $\wst{k-1} \defas \inten(\nabla \yst{k-1}, \tst{k-1})$,  $\wst k \defas \inten(\nabla \yst k, \tst k)$, \BBB and \EEE  $\ddif \wsts{k} \defas \frac{\wst k - \wst{k-1}}{\tau}$.
\end{proposition}

\BBB 
\begin{proof}
The thermal step differs from the one used in \cite{tve} only by the regularization of the dissipation term $ \drate(\nabla \yst{k-1}, \ddif \nabla \yst k, \tst{k-1})$ by $\eps |\delta_{\tau}   \nabla \Delta \ZZZ y^{(k)}_{\tau}  \EEE |^{2} \wedge \tau^{-1}$. Therefore,  the statement immediately follows from \cite[Proposition~3.8]{tve} since an inspection of its  proof  shows that for the existence of $\theta_\tau^{(k)}$ it is enough that the dissipation term lies in $L^\infty(\Omega)$ (see Steps 1 and 2) and for nonnegativity of $\theta_\tau^{(k)}$ it is enough that the dissipation is larger than $c|(\ddif \nabla \yst k)^T\nabla \yst{k-1} + (\nabla \yst{k-1})^T \ddif \nabla \yst k|^2$, see  \cite[Remark 3.9]{tve} and \ref{D_quadratic}.
\end{proof}
\EEE 

\subsection{Existence of \ZZZ time-discrete solutions\EEE}

\BBB In the previous subsection, we focused on one step in the staggered scheme. \EEE Our next goal is to prove the existence of \ZZZ time-discrete solutions \EEE and to derive \BBB first \EEE a priori bounds independent of $\tau$, $h$, and $\eps$\ee.
\BBB  We \EEE set
\begin{equation}\label{def_Cf}
  C_{f} \ee
  \defas \norm{f}_{W^{1,1}(I; L^2(\Omega))},
\end{equation}
and note that by the fundamental theorem of calculus it holds \BBB that
\begin{equation}\label{ell_bound}
  \Vert f(t) \Vert_{L^2(\Omega)} 
  \leq C_T C_{f} \quad \text{for all $t \in I$},
\end{equation}
where \BBB here and in the following \EEE $C_T$ denotes a constant possibly depending on $T$.

Given the sequences $\yst 0, \, \ldots,\, \yst k $ and $\tst 0, \, \ldots, \,  \tst k $ for some $k \in \setof{1, \, \ldots, \, T/\tau}$, as described \BBB in Subsection~\ref{MMscheme}, \EEE    for $l \in \setof{0, \, \ldots, \, k}$   we define
\begin{equation*}
  \mathcal{F}^{(l)} \defas \toten(\yst l, \tst l) - ( f(l \tau), \yst l )_2,
\end{equation*}
\BBB where the total internal energy $\toten$ is defined in \eqref{toten}. Then, using \ref{W_lower_bound} we find \EEE 
\begin{equation}\label{diff_F_E}
  |( f(l \tau), \yst l )_2|
  \leq \min \{ \mathcal{F}^{(l)}, \toten(\yst l, \tst l) \} + C_T C_{f}^2 \BBB + C(1+C_0), \EEE
\end{equation}
see also \cite[Lemma 3.10]{tve}. Finally,  for $l \in \{ 0, \, \ldots, \, \ZZZ k \EEE \}$ we also define   
\begin{equation}\label{Gl}
  \mathcal G^{(l)} \defas \mathcal F^{(l)} + \frac{\rho}{2} \frac{\tau}{h} \sum_{m = l-h/\tau+1}^l   \Vert \ddif \yst m \Vert_{L^2(\Omega)}^2   ,
\end{equation}
\BBB which corresponds to adding also a suitably \ZZZ averaged \EEE  kinetic energy. \EEE The main result of this subsection reads as follows. \EEE
\begin{proposition}[Existence of  \ZZZ time-discrete solutions\EEE]\label{prop:minmoves_existence}
  Let $ C_f\ee$ be as in \eqref{def_Cf} and $\mathcal G^{(0)}$ be as in \eqref{Gl}.
  For any $T > 0$, there exists a constant $\overline C_T \ZZZ  >0 \EEE $, corresponding constants
  \begin{equation*}
    M' \defas
    2 e^{\overline C_T C_{f}} \Big(
      \mathcal{G}^{(0)}
      + \overline C_T (1 + C_{f}^3)
      + \kappa  \BBB \int_I \EEE \int_{\partial \Omega} \theta_\flat \di \haus^{d-1} \di t
    \Big), \qquad
    M \defas 2M' + \overline C_T C_{f}^2,
  \end{equation*}
  as well as a constant $C_M \ZZZ >0 \EEE$ and scalar $\tau_0 \in (0, h]$ only depending on $M$ above such that the following holds true:
  For each $\tau \in (0, \tau_0)$ such that   $h / \tau \in \N$ the sequences $\yst 0, \ldots, \yst{T / \tau}$ and $\tst 0, \ldots, \tst{T / \tau}$  \BBB constructed in Subsection \ref{MMscheme} \EEE  exist, and for all $k \in \{ 0, \ldots, T / \tau \}$ it holds that
  \begin{align}
    \label{mechen_apriori_bound} 
    \toten(\yst k, \tst k) \BBB   
    + \frac{\rho}{2} \BBB \frac{\tau}{h} \EEE \sum_{l = k - h/\tau + 1}^k   \Vert \ddif \yst{l} \Vert_{L^2(\Omega)}^2 &\leq   M,   \\
    \label{velocity_apriori_bound}
    \sum_{l = 1}^{k} \tau \Big( \Vert \ddif \nabla \yst l \Vert_{L^2(\Omega)}^2    + \eps   \| \delta_{\tau} \BBB \nabla \Delta y^{(l)}_{\tau} \EEE \|_{L^{2}(\Om)}^{2}
        \Big) &\leq C_M (M \BBB (1+T) \EEE + \overline C_T C_{f}^2).
  \end{align}
\end{proposition}

\BBB The first estimate corresponds to a bound on the total  energy   and the second one is a bound on the (regularized) strain rate.  \EEE  The proof  relies on the following two lemmas.

%

\begin{lemma}[Inductive bound on the total energy]\label{lem:inductive_toten_bound}
  For any $M,\, \BBB T \EEE > 0$ there \BBB exist constants $C_M >0$ and $C_T >0$ only depending on  $M$ and $T$, respectively, \EEE  such that the following holds true:
  Suppose $\tau \in (0, 1)$ is chosen such that for $k \in \setof{1, \, \ldots, \, T / \tau}$ the sequences $\yst 0, \, \ldots, \, \yst k$ and $\tst 0, \, \ldots, \, \tst k$ \BBB constructed in Subsection \ref{MMscheme} \EEE  exist.
  Moreover, assume that $\mathcal{G}^{(l)} \leq M$ for all $l \in \{ 0, \, \ldots, \, k-1 \}$ with $\mathcal{G}^{(l)}$ as in \eqref{Gl}.
  Then, it holds that
  \begin{equation*}
  \begin{aligned}
    \mathcal{G}^{(k)}
    &\leq
     \BBB  \mathcal{G}^{(0)} \EEE
           + C_M \tau V_k
      + C_T(1 + C_{f}^3)
      + \kappa \int_0^{k \tau} \int_{\partial \Omega} \theta_\flat \di \haus^{d-1} \di t  
      + C\sum_{l = 0}^k
      \mathcal{G}^{(l)} \BBB \int_{(l-1)\tau}^{(l+1)\tau} \EEE  
             \Vert \partial_t f(t) \Vert_{L^2 (\Omega)}    \di t,
  \end{aligned}
  \end{equation*}
  where $C >0$ is a universal constant,  and
  \begin{equation}\label{def_Vk}
    V_k \defas \sum_{m = 1}^k \tau \int_\Omega \abs{\ddif \nabla \yst m}^2 \di x.
  \end{equation}
\end{lemma}

\begin{lemma}[Inductive bound on the strain rates]\label{lem:Vk_bound}
  Given $M, \, T > 0$, there exist a constant $C_M >0$ and $\tau_0 \in (0, 1]$ only depending on $M$, and a constant $C_T >0$ only depending on $T$ such that for $\tau \in (0, \tau_0)$ the following holds:
  Suppose that the sequences $\yst 0, \, \ldots, \, \yst k$ and $\tst 0, \, \ldots, \, \tst k$ for some $k \in \{1, \ldots, T / \tau\}$ \BBB constructed in Subsection~\ref{MMscheme} \EEE  exist.
  Moreover, suppose that $ \MMM \mathcal  G^{(l)}\EEE  \leq M$ for all $l \in \{0, \, \ldots, \, k-1\}$ with $\mathcal G^{(l)}$ as defined in \eqref{Gl}.
  Then,
  \begin{equation}\label{Vk_bound}
  \begin{aligned}
    &  \sum_{l = 1}^{k} \tau \Big( \Vert \ddif \nabla \yst l \Vert_{L^2(\Omega)}^2    + \eps   \| \delta_{\tau} \BBB \nabla \Delta y^{(l)}_{\tau} \EEE \|_{L^{2}(\Om)}^{2} \Big)\\
      &\quad\leq
      C_M \left(
        \mechen(y_{0,\eps}) + \frac{\rho}{2} \Vert y_{0,\eps}' \Vert_{L^2(\Omega)}^2  + C_T C_{f}^2
      \right) \BBB +    C_M \tau \sum_{l = 0}^{k-1}   \big(1+ \mechen(\yst {l})\big) \EEE .
  \end{aligned}
  \end{equation}
\end{lemma}

\begin{proof}[Proof of Proposition \ref{prop:minmoves_existence}]
 \BBB  Once  Lemmas \ref{lem:inductive_toten_bound}--\ref{lem:Vk_bound} are proved,  the proof of Proposition \ref{prop:minmoves_existence} follows \EEE  by an inductive argument using the discrete Gronwall's inequality  and Propositions \ref{prop:existence_mechanical_step}--\ref{prop:existence_thermal_step}.
  We refer to \cite[Theorem 3.13]{tve} for \BBB all \EEE details  \ZZZ (cf.\ also \cite[Lemma  3.11, Lemma 3.12]{tve}). \EEE  
\end{proof}

\BBB 
We now proceed with the proof of the lemmas.  As an auxiliary result, we show $\Lambda$-convexity of $\elpot$ and  $\cplpot$. \seb{The result is closely related to the estimate in \cite[Subsection 2.3]{cesik}, which improved upon \cite[Proposition  3.2]{tve_orig}, where \emph{local} $\Lambda$-convexity has been shown. }

\begin{lemma}[$\Lambda$-convexity of $\elpot$ and  $\cplpot$]\label{lemma_ lambda convex}
For any $M > 0$ there exists a constant $C_M>0$ such that for all $y_1, y_2 \in \Yid $ with $\mechen(y_1), \mechen(y_2) \le M$ and $\theta \in \L^1(\Omega)$, we have 
\begin{align*}
  \int_\Omega \elpot(y_{2}) \, \di x  \ge   \int_\Omega  \elpot(y_1) \di x +   \int_\Omega
      \pl_F \elpot(\nabla y_1) : \big( \nabla y_2 - \nabla y_1 \big) \di x - C_M  \Vert \nabla y_2 - \nabla y_1 \Vert^2_{L^2(\Omega)}. 
      \end{align*}
      \begin{align*}
  \int_\Omega \cplpot(y_{2},\theta) \, \di x  \ge   \int_\Omega  \cplpot(y_1,\theta) \di x +   \int_\Omega
      \pl_F \cplpot(\nabla y_1,\theta) : \big( \nabla y_2 - \nabla y_1 \big) \di x - C_M  \Vert \nabla y_2 - \nabla y_1 \Vert^2_{L^2(\Omega)}. 
      \end{align*}
\end{lemma}

 \EEE
 
\begin{proof}
\BBB We first prove the statement for $\elpot$. \EEE By Lemma \ref{lem:pos_det} there exists a constant \BBB  $C^*_M$ \EEE  depending on  $M$ such that for all $x \in \Omega$ and $l \in \{ 1,2\}$ it holds that
  \begin{equation}\label{strain_det_bounds}
    |\nabla y_l(x)| \leq C_M^*, \quad \det(\nabla y_l(x)) \geq \frac{1}{C_M^*}.
  \end{equation}
By the continuity of the determinant we can find $\delta_M > 0$ such that for all $\lambda \in (0, 1)$ and $F_1, \, F_2 \in GL^+(d)$ with \ZZZ $|F_1|, |F_2| \le C_M^*$, \EEE $\det(F_1), \, \det(F_2) \geq \frac{1}{C_M^*}$ and $|F_1 - F_2| \leq \delta_M$ we have that $\det(\lambda F_1 + (1 - \lambda) F_2) \geq \frac{1}{2 C_M^*}$. \BBB The \EEE set
  \begin{equation*}
    K \defas \left\{F \in GL^+(d) \colon |F| \leq C_M^*, \, \det(F) \geq \frac{1}{2 C_M^*}\right\}
  \end{equation*}
  is a compact subset of $GL^+(d)$.
  Hence, by the $C^2$-regularity of $\elpot$ it follows that
 \begin{align*}
    C_M \defas \sup_{F \in K} | \ZZZ \partial_{FF}^2 \EEE \elpot(F)| < \infty.
  \end{align*}
  With $\delta_M$ as above,  let us  define the set \BBB $G \defas \{x \in \Omega \colon  |\nabla y_2(x) - \nabla y_{1}(x)| \leq \delta_M \}$.  Moreover, we define $y_t = (1-t)y_1 + ty_2$ for $t \in [0,1]$. \EEE   By the previous reasoning, for every $x \in G$  it holds that $\nabla y_t(x) \in \BBB K \EEE$ and
  \begin{equation}\label{dist_pa_pc_interp}
  \begin{aligned}
    &\abs{\pl_F \elpot(\nabla y_t(x)) - \pl_F \elpot(\nabla y_1(x))} \leq C_M \abs{\nabla y_t(x) - \nabla y_1(x)} \le  C_M  \abs{\nabla y_2(x) - \nabla y_1(x)}.
  \end{aligned} 
  \end{equation}
  On the one hand,   by the Gateaux differentiability of $\mathcal{W}^{\rm el}$ (see \cite[Proposition 3.2]{tve_orig}) \BBB and the chain rule  we have that
  \begin{align*}
\int_G \elpot(\nabla y_2) \, \di x- \int_G \elpot(\nabla y_1) \, \di x & = \int_0^1  \int_G  \partial_F\elpot(\nabla y_t)  : \big( \nabla y_2 - \nabla y_1\big) \, \di x \ste \di t .
  \end{align*}\EEE
  On the other hand, using \eqref{dist_pa_pc_interp} we can estimate
  \begin{align*}
    &\Bigg|
       \int_0^1  \int_G  \partial_F\elpot(\nabla y_t)  : \big( \nabla y_2 - \nabla y_1\big) \, \di x \ste \di t \EEE
        - \int_0^1  \int_G  \partial_F\elpot(\nabla y_1)  : \big( \nabla y_2 - \nabla y_1\big)   \di x \ste \di t \EEE
      \Bigg| 
    \leq C_M \Vert \nabla y_2 - \nabla y_1\Vert_{L^2(\Omega)}^2.
  \end{align*}
  By the definition of $G$ we also see that, by possibly increasing $C_M$,  it holds
  \begin{align} \label{for explanation2}
  &  \Big|\int_{\Omega\setminus G} \elpot(\nabla y_2) \, \di x - \int_{\Omega \setminus G} \elpot(\nabla y_1) \di x   \BBB -  \int_{\Omega \setminus G}
      \pl_F \elpot(\nabla y_1) : \big( \nabla y_2 - \nabla y_1 \big) \EEE \di x\Big|
    \notag \\  & \quad \leq \int_{\Omega \setminus G} \frac{|\elpot(\nabla y_1)| + |\elpot(\nabla y_2)| \BBB +  \delta_M|\pl_F \elpot(\nabla y_1)| \EEE }{\delta_M^2} |\nabla y_2 - \nabla y_1|^2 \di x \leq C_M \Vert \nabla y_2 - \nabla y_1 \Vert_{L^2(\Omega)}^2,
  \end{align}
 \BBB where we used that $\elpot(\nabla y_1) $, $\elpot(\nabla y_2)$, and $|\pl_F \elpot(\nabla y_1)|$  are uniformly bounded by  \eqref{strain_det_bounds}.  The combination of the  \EEE last three estimates gives  the statement for $\elpot$.

\BBB The argument for $\cplpot$ is similar. However, the bounds in \eqref{dist_pa_pc_interp} and \eqref{for explanation2} do not follow immediately from the compactness of $K$ due to the presence of the temperature $\theta \in L^1(\Omega)$.  To obtain the analogous bounds, we use \ref{C_lipschitz}, the first inequality in \ref{C_bounds}, and    \eqref{strain_det_bounds}. \EEE    
\end{proof}

\BBB As a second auxiliary result, we establish a bound on the mechanical energy. \EEE 
  
\begin{lemma}[Mechanical energy bound in the
time-discrete setting]\label{lemma: mechenbound}

Given $M > 0$, there exists a constant $C_M>0$ such that the  following holds:   Suppose   that for $\tau \in (0,\tau_0)$ and  $k \in \setof{1, \, \ldots, \, T / \tau}$ the sequences $\yst 0, \, \ldots, \, \yst k$ and $\tst 0, \, \ldots, \, \tst k$ \BBB constructed in Subsection \ref{MMscheme} \EEE  exist, and that  $ \mathcal{G}^{(l)}\leq M$ for all $l \in \{0, \, \ldots, \, k-1\}$.  Then,   it holds that 
 \begin{align*} 
 & \mechen(\yst k)   + \frac{\rho \tau }{2h} \sum_{l=k - h/\tau + 1}^k \Vert \ddif \yst l \Vert_{L^2(\Omega)}^2 + \sum_{l = 1}^k  \tau \Big(   2\diss_\eps( \yst{l-1}, \ddif  \yst l, \tst{l-1})   +      \frac{\rho }{2h}  \Vert \ddif \yst l - \ddif \yst{l - h/\tau} \Vert_{L^2(\Omega)}^2    \Big)
     \\ &\leq
      \mechen(y_{0,\eps})     + \EEE  \frac{\rho}{2} \Vert \BBB y_{0,\eps}' \EEE \Vert_{L^2(\Omega)}^2 -     \tau \sum_{l=1}^k  \int_\Omega
      \pl_F \cplpot(\nabla \yst {l-1}, \tst{l-1}) : \ddif\nabla \yst l \di x
      + \tau \sum_{l = 1}^k ( \lst l, \ddif \yst l )_2    + C_M \tau V_k, 
  \end{align*} 
  \BBB where $\mathcal{R}_\eps$ is given in \eqref{Repps} and  $V_k$  in \eqref{def_Vk}. \EEE
\end{lemma}

\begin{proof}
  Using Proposition \ref{prop:existence_mechanical_step} for $l$ in place of $k$, \eqref{diss_rate}, and testing \eqref{mechanical_step_el} with $z = \ddif \yst l$ it follows that
  \begin{align}
  \label{el_mech_test_step_l}
    0
    = &  \int_\Omega
      \pl_F \felpot(\nabla \yst l, \tst{l-1}) : \ddif \nabla \yst l
      +   D H \ee(\Delta \yst l) \cdot \ddif \Delta \yst l  \BBB + \eps  \delta_{\tau}  \nabla \Delta y^{(l)}_{\tau} :  \delta_{\tau}  \nabla \Delta y^{(l)}_{\tau}    \EEE \di x 
        \\
    & + \int_\Omega \drate(\nabla \yst{l-1}, \ddif \nabla \yst l, \tst{l-1}) \di x  - \int_\Omega \lst l \cdot  \ddif \yst l \di x \BBB + \frac{\rho}{h} \int_\Omega \big(\ddif \yst l - \ddif \yst{l-h/\tau} \big) \cdot  \ddif \yst l \di x.   \nonumber \EEE
   \end{align}
 By the convexity of $\hypot$ (see \ref{H_regularity}) it follows for  $l \in \{ 1, \ldots, k \}$  that
  \begin{align}\label{hypot_sum_lower_bound}
    \hyen(\Delta \yst{l-1})& \geq
    \hyen(\Delta \yst l)
      +  \int_\Omega  D H\ee(\Delta \yst l) \cdot (\Delta \yst{l-1} - \Delta \yst l)  \, \di x   =  \hyen(\Delta \yst l)
      \BBB - \EEE    \tau    
      \int_\Omega  D H \ee(\Delta  \yst l) \cdot \ddif \Delta \yst l \di x, \EEE
  \end{align}
  where we recall the notation in \eqref{hyperelastic}. \BBB 
   We now perform a similar argument for the elastic energy. \EEE  \ZZZ Using     \eqref{def_Cf},   \eqref{diff_F_E}--\eqref{Gl},   and $\mathcal{G}^{(l)}\leq M$ for $l \in \{ 0, \, \ldots, \,  k - 1 \}$, we get  
\begin{align}\label{usedlateragain}
 \mechen(\yst {l-1}) \le \mathcal{G}^{(l-1)} +  ( f((l-1) \tau), \yst{l-1} )_2 \le 2 \mathcal{G}^{(l-1)}  + C + C_T C_{f}^2 \le 2 M  + C + C_T C_{f}^2 
 \end{align}
  for all $l \in \{ 1, \, \ldots, \, k \}$. Thus, we can apply Proposition \ref{prop:existence_mechanical_step} for all $l \in \{ 1, \, \ldots, \, k \}$,  where now $C_M$ may also depend on $T$ and $f$. Then,  using again  $\mathcal{G}^{(l)}\leq M$ for $l \in \{ 0, \, \ldots, \,  k - 1 \}$, by \eqref{bad_mechen_bound_2} we find  for all $l \in \{ 1, \, \ldots, \, k \}$ that 
  \begin{equation}\label{eq: MMM}
    \mechen(\yst l)
    \leq C_M   \big(
        1
        +   \Vert f  \Vert_{L^2(I \times \Omega)}^2  \EEE 
    \big)  \ste+  \ZZZ   2\mathcal{G}^{(l-1)}   
    \leq 2 M + C_M(1 + C_T C_{f}^2).\EEE
  \end{equation}
  Consequently, by   \eqref{difff},  \ZZZ  \eqref{usedlateragain}--\eqref{eq: MMM}, and \EEE Lemma \ref{lemma_ lambda convex}  applied for $y_1 = \yst {l}$ and $y_2 = \yst {l-1}$ we find \EEE 
  \begin{equation}\label{elpot_sum_lower_bound}
    \tau  \int_\Omega
      \pl_F \elpot(\nabla \yst l) : \ddif\nabla \yst l \di x
    \geq \elen(\yst l) - \elen( \yst {l-1}) - C_M \tau^2 \int_\Omega \abs{\ddif \nabla \yst l}^2 \di x
  \end{equation} 
  \ZZZ for a possibly larger $C_M$, where we recall the definition in \eqref{purely_elastic}. \EEE 
 Multiplying \eqref{el_mech_test_step_l} by $\tau$, using \eqref{hypot_sum_lower_bound} and \eqref{elpot_sum_lower_bound},  and   summing over $l \in \{ 1, \ldots, k \}$,  we conclude   
 \begin{align} \label{preliaus}
 & \mechen(\yst k)   + \frac{\tau\rho}{h} \sum_{l=1}^k  \int_\Omega \big(   \ddif \yst l - \ddif \yst{l-h/\tau} \big) \cdot  \ddif \yst l \di x  \notag \\
 &\quad+   \sum_{l = 1}^k\tau \Big( \int_\Omega \drate(\nabla \yst{l-1}, \ddif \nabla \yst l, \tst{l-1}) \di x   +\eps  \Vert \delta_{\tau}  \nabla \Delta y^{(l)}_{\tau} \Vert^2_{L^2(\Omega)} \Big)
    \notag  \\ &\qquad\leq
      \mechen(y_{0,\eps}) 
       -     \tau \sum_{l=1}^k  \int_\Omega
      \pl_F \cplpot(\nabla \yst l, \tst{l-1}) : \ddif\nabla \yst l \di x
      + \tau \sum_{l = 1}^k ( \lst l, \ddif \yst l )_2    + C_M \tau V_k, 
  \end{align} 
  where we employed the definitions in \eqref{mechanical} and \eqref{def_Vk}. Using the identity \EEE
$$
\Pi_l \defas \int_\Omega \big(   \ddif \yst l - \ddif \yst{l-h/\tau} \big) \cdot  \ddif \yst l \di x
    = \frac{1}{2} (\Vert \ddif \yst l \Vert_{L^2(\Omega)}^2 - \Vert \ddif \yst{l - h/\tau} \Vert_{L^2(\Omega)}^2 + \Vert \ddif \yst l - \ddif \yst{l - h/\tau} \Vert_{L^2(\Omega)}^2),
$$
summing over $l \in \{ 1, \ldots, k \}$, and recalling \BBB the definition  $\yst m =  y_{0,\eps} + m \tau y_{0,\eps}'$ \EEE for $m \in \{-h/\tau, \ldots, 0\}$ we derive that
  \begin{align}\label{accelia}
   \BBB  \sum_{l=1}^k  2\tau      \Pi_l  & \BBB = \EEE  \sum_{l=1}^k  \tau \Vert \ddif \yst l - \ddif \yst{l - h/\tau} \Vert_{L^2(\Omega)}^2  +   \sum_{l=k - h/\tau + 1}^k \tau \Vert \ddif \yst l \Vert_{L^2(\Omega)}^2
      - \sum_{l=-h/\tau + 1}^0 \tau \Vert \ddif \yst l \Vert_{L^2(\Omega)}^2 
   \notag  \\ & =  \sum_{l=1}^k  \tau \Vert \ddif \yst l - \ddif \yst{l - h/\tau} \Vert_{L^2(\Omega)}^2 + \sum_{l=k - h/\tau + 1}^k \tau \Vert \ddif \yst l \Vert_{L^2(\Omega)}^2
      - h \Vert  \BBB y_{0,\eps}' \EEE \Vert_{L^2(\Omega)}^2.
  \end{align} \BBB  
 We apply  Lemma \ref{lemma_ lambda convex} first for $y_1 = \yst {l-1}$ and $y_2 = \yst l$,  then for  $y_1 = \yst {l}$ and $y_2 = \yst {l-1}$, and sum the equations to get
 $$ 0 \ge    \int_\Omega
      \Big(
        \pl_F \cplpot(\nabla \yst l, \tst{l-1})
        - \pl_F \cplpot(\nabla \yst {l-1}, \tst{l-1})
      \Big) : \big( \nabla \yst {l-1}   - \nabla \yst {l}    \big)   \di x
     - 2C_M \tau^2 \int_\Omega \abs{\ddif \nabla \yst l}^2 \di x. $$
 Rearranging and summing over $l \in \lbrace 1,\ldots, k\rbrace$ yields 
  \begin{equation}\label{cpldiff_bound}
    \tau \sum_{l = 1}^k \int_\Omega
      \Big(
        \pl_F \cplpot(\nabla \yst l, \tst{l-1})
        - \pl_F \cplpot(\nabla \yst {l-1}, \tst{l-1})
      \Big) : \ddif \nabla \yst l \di x
    \geq - 2C_M \tau V_k.
  \end{equation}
Eventually, recalling  \eqref{diss_rate}--\eqref{Repps}, the combination of  \eqref{preliaus}, \eqref{accelia}, and \eqref{cpldiff_bound}  concludes the proof.   
 \end{proof}

\BBB 
 
 We now proceed with the proofs of Lemma \ref{lem:inductive_toten_bound} and Lemma \ref{lem:Vk_bound}. \EEE
 
 \begin{proof}[Proof of Lemma \ref{lem:inductive_toten_bound}] 
\ZZZ Recalling the argument in \eqref{usedlateragain}, we observe that Proposition \ref{prop:existence_thermal_step} is applicable by passing to a larger value of $M$. We test \eqref{thermal_step_el}  (for $l$ in place of $k$) \EEE with $\vphi = 1$ to obtain
  \begin{align*}
    0
    = &  \int_\Omega
     \Big( \ddif \BBB w_\tau^{(l)} \EEE
      -\pl_F \cplpot(\nabla \yst {l-1}, \tst{l-1}) : \ddif \nabla \yst l
      - \drate(\nabla \yst{l-1}, \ddif \nabla\yst l, \tst{l-1}) 
    \di x   -  \eps |\delta_{\tau}   \nabla \Delta y^{(\tau)}_{k} |^{2} \wedge \tau^{-1} \Big) \di x \nonumber  \\ 
    & \quad    + \kappa \int_{\partial \Omega} (\tst l - \btst l) \di \haus^{d-1}. 
  \end{align*}
Multiplying  this equation by $\tau$, summing over $l \in \{1, \, \ldots, \, k\}$,  \BBB and adding  to  the estimate in Lemma \ref{lemma: mechenbound}, by  \eqref{diss_rate}--\eqref{Repps} \EEE  we discover that
 \begin{align*} 
 & \mechen(\yst k)   + \frac{\rho\tau}{2h} \sum_{l=k - h/\tau + 1}^k  \Vert \ddif \yst l \Vert_{L^2(\Omega)}^2   +  \BBB  \int_\Omega w^{(k)}_\tau \EEE \di x
     \\ &\leq
      \mechen(y_{0,\eps})  +    \int_\Omega  w_{0,\eps}  \di x    + \EEE  \frac{\rho}{2} \Vert \BBB y_{0,\eps}' \EEE  \Vert_{L^2(\Omega)}^2  
      + \tau \sum_{l = 1}^k ( \lst l, \ddif \yst l )_2    + \tau \sum_{l = 1}^k \kappa \int_{\partial \Omega} (\btst l - \tst l) \di \haus^{d-1}   + C_M \tau V_k, 
  \end{align*} 
\BBB where $w_{0,\eps}  =  \inten(\nabla y_{0,\eps}, \theta_0)$. \EEE  Recalling the definition of   $\toten$ in \eqref{toten},   we conclude that
\begin{align}\label{test_all_steps_v2}
  &  \toten(\yst k, \tst k)   + \frac{\rho\tau }{2h} \sum_{l=k - h/\tau + 1}^k \Vert \ddif \yst l \Vert_{L^2(\Omega)}^2 \notag\\ 
    &\leq       \toten(y_{0,\eps}, \BBB \theta_{0} \EEE )   + \frac{\rho}{2} \Vert \BBB y_{0,\eps}' \EEE  \Vert_{L^2(\Omega)}^2   
    + C_M \tau V_k       + \tau \sum_{l = 1}^k ( \lst l, \ddif \yst l )_2
      + \tau \sum_{l = 1}^k \kappa \int_{\partial \Omega} (\btst l - \tst l) \di \haus^{d-1}  .
  \end{align}
\BBB Now, we \EEE estimate the last two terms on the right-hand side of (\ref{test_all_steps_v2}).
  By the nonnegativity of $\tst l$ and the definition of $\btst l$ we can bound 
  \begin{equation}\label{b_surface}
    \tau \sum_{l = 1}^k \kappa \int_{\partial \Omega}(\btst l - \tst l) \di \haus^{d-1}
      \leq  \ZZZ \tau \EEE \sum_{l = 1}^k \kappa \int_{\partial \Omega} \btst l \di \haus^{d-1}
    = \kappa \int_0^{k \tau} \int_{\partial \Omega} \theta_\flat \di \haus^{d-1} \di t.
  \end{equation}
 We define  the piecewise affine function $\hat y_\tau(t) = \frac{t - (l-1)\tau}{\tau}  y_\tau^{(l)} + \frac{l\tau - t}{\tau} y_\tau^{(l-1)}   $ for $t \in [(l-1)\tau, l\tau]$ and   $l \in \{ 1, \ldots, k \}$, and note that \EEE   $\ddif   \yst l  = \partial_t \hat y_\tau(t)$ for   $t \in ((l-1)\tau, l\tau)$.   Consequently, integration by parts yields
  \begin{align}\label{ibp_force_term}
    \sum_{l = 1}^k  \tau ( \lst l, \ddif \yst l )_2
    & = \int_0^{k\tau} ( f(t), \partial_t \hat y_\tau(t))_2 \di t
    =  ( f(k\tau), \hat y_\tau(k \tau) )_2
      - ( f (0), \BBB y_{0,\eps} \EEE )_2
      - \int_0^{k\tau} ( \partial_t f(t), \hat y_\tau(t) )_2 \di t \nonumber \\
    & \leq ( f(k\tau), \hat y_\tau(k \tau) )_2
      - ( f(0), \BBB  y_{0,\eps}\EEE )_2
      + \int_0^{k\tau} \BBB  \Vert \partial_t f(t) \Vert_{L^2(\Omega)} 
        \Vert \hat y_\tau(t) \Vert_{L^2 (\Omega)} \EEE \di t.
  \end{align}
  By Poincar\'e's inequality \EEE and \ref{W_lower_bound} we have for every $t \in ((l-1)\tau, l\tau)$ that
  \begin{equation*}
    \Vert \hat y_\tau(t) \Vert^2_{\BBB L^2 (\Omega)}
    \leq C (
      \Vert \nabla \yst{l-1} \Vert^2_{L^2(\Omega)} + \Vert \nabla \yst l \Vert^2_{L^2(\Omega)}
    )
    \leq C \big(1 + \elen(\yst{l-1}) + \elen(\yst l)\big).
  \end{equation*}
  Therefore, by \eqref{ell_bound},   \eqref{diff_F_E},    the definition of $\mathcal G^{(l)}$ in \eqref{Gl}, and $\sqrt{s} \leq 1 + s$ for all $s \geq 0$ we get
  \begin{align*}
    \int_0^{k \tau}
      \Vert \partial_t f(t) \Vert_{L^2 (\Omega)}
      \Vert \hat y_\tau(t) \Vert_{L^2 (\Omega)} \di t     &\leq C \sum_{l=1}^k \Big(
        1 + \elen(\yst{l-1}) + \elen(\yst l)
      \Big)
      \int_{(l-1) \tau}^{l \tau} \Vert \partial_t f(t) \Vert_{\BBB L^2 (\Omega)} \di t \\
    &\quad\leq C \sum_{l=1}^k \Big(
      \big(\mathcal{G}^{(l-1)} + \mathcal{G}^{(l)} \big)
      \int_{(l-1)\tau}^{l\tau} \Vert \partial_t f(t) \Vert_{\BBB L^2 (\Omega)} \di t
    \Big) +  C_T (C_{f} + C_{f}^3 ). 
  \end{align*}
  Then, using an index shift and   $C_{f} \leq \frac{2}{3} + \frac{1}{3}C_{f}^3$ \EEE we get
  \begin{align*}
    &\int_0^{k\tau} \hspace{-0.1cm} \Vert \partial_{t}{f}(t) \Vert_{L^2(\Omega)}
      \Vert \hat{y}_\tau(t)\Vert_{L^2(\Omega)}  \di t  \leq C \sum_{l=0}^k \Big(
      \mathcal{G}^{(l)} \int_{(l-1)\tau}^{l\tau}
        \big(
          \Vert \partial_{t}{f}(t) \Vert_{L^2 (\Omega)}
          + \Vert \partial_{t}{f}(t + \tau) \Vert_{L^2 (\Omega)}
        \big) \di t
    \Big) + C_T (1  + C_{f}^3)
  \end{align*}
  for a possibly larger $C_T>0$.
  Plugging this into \eqref{ibp_force_term} and using \eqref{b_surface} to estimate the terms on the right-hand side of (\ref{test_all_steps_v2}), \BBB we \EEE conclude the proof by the definition of  $\mathcal G^{(l)}$ in \eqref{Gl}\ee.
\end{proof}

\begin{proof}[Proof of Lemma \ref{lem:Vk_bound}]
 \BBB Applying  Lemma \ref{lemma: mechenbound}  and using the nonnegativity of $\mechen$ we get \EEE 
 \begin{align}\label{forstrain0}
 & \sum_{l = 1}^k  \tau \Big(  \int_\Omega \drate(\nabla \yst{l-1}, \ddif \nabla \yst l, \tst{l-1}) \di x  +     \eps   \| \delta_{\tau}  \nabla \Delta y^{(l)}_{\tau}   \|_{L^{2}(\Om)}^{2}  \Big)
     \\ &\leq
      \mechen(y_{0,\eps})     + \EEE  \frac{\rho}{2} \Vert y_{0,\eps} ' \Vert_{L^2(\Omega)}^2 -     \tau \sum_{l=1}^k  \int_\Omega
      \pl_F \cplpot(\nabla \yst {l-1}, \tst{l-1}) : \ddif\nabla \yst l \di x
      + \tau \sum_{l = 1}^k ( \lst l, \ddif \yst l )_2    + C_M \tau V_k. \notag
  \end{align} 
 \ZZZ As $   \mechen(\yst {l-1}) \le   2 M  + C + C_T C_{f}^2 $   for all $l \in \{ 1, \, \ldots, \, k \}$,    see the argument in \eqref{usedlateragain},  \BBB employing also Lemma \ref{lem:pos_det}  \EEE  we can apply \BBB the generalized Korn's inequality in the form \cite[Corollary 3.4]{tve_orig}, \EEE leading to
  \begin{equation}\label{forstrain1}
    \int_\Omega \xi(\nabla \yst{l-1}, \ddif \nabla \yst l, \tst{l-1}) \di x
    \geq \frac{1}{C_M} \Vert \ddif \nabla \yst l \Vert_{L^2(\Omega)}^2,
  \end{equation}
  \BBB for a constant $C_M$ depending on $M$, $T$, and $f$. \EEE   By H\"older's inequality, Poincar\'e's inequality, and Young's inequality with constant $\lambda \in (0, 1)$ we derive that 
  \begin{align}\label{forstrain2}
    |( f_\tau^{(l)}, \ddif \yst l )_2|
    \leq C \Vert f_\tau^{(l)} \Vert_{L^2(\Omega)} \Vert \ddif \nabla \yst l \Vert_{L^2(\Omega)}  
    &\leq \frac{C}{\lambda} \Vert f_\tau^{(l)} \Vert_{L^2 (\Omega)}^2
      + \lambda \Vert \ddif \nabla \yst l \Vert_{L^2(\Omega)}^2.
  \end{align}
  Choosing $\lambda < \frac{1}{2 C_M}$ above, \BBB summing over $l \in \lbrace 1,\ldots, k \rbrace$ in \eqref{forstrain1}--\eqref{forstrain2}, and plugging into  \eqref{forstrain0}, we find  \EEE 
  \begin{align}\label{oooj}
  \BBB   & \frac{1}{2C_M}V_k   +  \sum_{l = 1}^k  \tau     \eps   \| \delta_{\tau}  \nabla \Delta y^{(l)}_{\tau}   \|_{L^{2}(\Om)}^{2}     = \EEE  \frac{1}{2C_M} \sum_{l = 1}^k \tau \Vert \ddif \nabla \yst l \Vert_{L^2(\Omega)}^2  +  \sum_{l = 1}^k  \tau     \eps   \| \delta_{\tau}  \nabla \Delta y^{(l)}_{\tau}   \|_{L^{2}(\Om)}^{2}  
     \\    &   
    \leq \mechen(y_{0,\eps}  ) + \frac{\rho}{2} \Vert y_{0,\eps} ' \Vert_{L^2(\Omega)}^2 - \BBB     \tau \sum_{l=1}^k  \int_\Omega
      \pl_F \cplpot(\nabla \yst {l-1}, \tst{l-1}) : \ddif\nabla \yst l \di x  + C_M \tau V_k\EEE  + C_M C_T C_{f}^2, \notag
  \end{align}
  \BBB where we also used \eqref{ell_bound}. Since  $|\partial_F W^{\rm cpl}(F,\theta)|  \le 2C_0(1 + |F|) \le  2C_0(2 + |F|^2)$ for all $F \in GL^+(d)$ and $\theta>0$, see \cite[Lemma~3.4]{tve}, applying Young's inequality, \ref{W_lower_bound}, and \eqref{mechanical} we get for some $C_M^*>0$ that
\begin{align*}
\tau \sum_{l=1}^k  \int_\Omega
      \pl_F \cplpot(\nabla \yst {l-1}, \tst{l-1}) : \ddif\nabla \yst l \di x &\le \frac{1}{8C_M} \sum_{l = 1}^k \tau \Vert \ddif \nabla \yst l \Vert_{L^2(\Omega)}^2 +  C^*_M \tau \sum_{l = 1}^k \int_\Omega  (1+ |\nabla \yst {l-1}|^2) \, \di x \\
      & \le \frac{1}{8C_M} V_k  +  C^*_M \tau \sum_{l = 1}^k   \big(1+ \mechen(\yst {l-1})\big) \,.
      \end{align*}
 Plugging this into  \eqref{oooj}, and possibly decreasing $\tau_0$ such that $ C_M^2 \tau_0 \leq \frac{1}{8}$ holds true, we get 
$$ \frac{1}{4C_M}V_k +   \sum_{l = 1}^k  \tau     \eps   \| \delta_{\tau}  \nabla \Delta y^{(l)}_{\tau}   \|_{L^{2}(\Om)}^{2}   \le \mechen(y_{0,\eps} ) + \frac{\rho}{2} \Vert y_{0,\eps} ' \Vert_{L^2(\Omega)}^2   +   C_M \tau \sum_{l = 1}^k   \big(1+ \mechen(\yst {l-1})\big) + C_M C_T C_{f}^2
$$
for $C_M \ZZZ >0 \EEE$ sufficiently large.
  Multiplying both sides with $4C_M$ then leads to the desired estimate \eqref{Vk_bound}.
\end{proof}

\BBB 
 
\subsection{A priori bounds, compactness, and regularity} \label{sec: APB}

Given \EEE $\yst 0, \ldots,  \yst {T / \tau}$ and $\tst 0, \ldots, \tst{T / \tau}$ from Proposition \ref{prop:minmoves_existence}, we define the following interpolations:
for $k \in \setof{ \BBB -h/\tau, \EEE \ldots,   T \ee / \tau}$, let $\overline y_\tau(k\tau) = \underline y_\tau(k\tau) = \hat y_\tau(k\tau)  \defas \yst k$ and for $t \in ((k-1)\tau, k\tau)$ let
\begin{align}\label{interpolations}
  \overline{y}_\tau(t) &\defas \yst k, &
  \underline{y}_\tau(t) & \defas \yst{k-1}, &
  \hat{y}_\tau(t) &\defas \frac{k \tau - t}{\tau} \yst{k-1} + \frac{t - (k-1)\tau}{\tau} \yst k.
\end{align}
A similar notation is employed for $\overline{\theta}_\tau$, $\underline{\theta}_\tau$, and $\hat{\theta}_\tau$, \BBB for \EEE $\overline{w}_\tau$, $\underline{w}_\tau$, and $\hat{w}_\tau$, \BBB and for $\overline{f}_\tau$. \EEE  The next proposition lists several a priori bounds for the sequences of interpolations.

\begin{proposition}[A priori bounds \BBB and compactness\EEE]\label{prop:a_priori_bounds}
  Let $T, \, h, \, \eps > 0$ and $\tau_0$ be as in Proposition \ref{prop:minmoves_existence}.
  Then, there exists a constant $C$ only depending on $T$, \BBB $y_{0,\eps}$, $y_{0,\eps}'$,  $\theta_{0}$,   $f$, and $\theta_\flat$ \EEE such that for all $\tau \in (0, \tau_0)$ with $T /  h \ee \in \N$ and $h / \tau \in \N$ the following bounds hold true:
  \begin{subequations}
  \begin{align}
    \Vert \overline y_\tau \Vert_{L^\infty(\BBB I_h; \EEE  W^{2, p}(\Omega))}  
    + \Vert \det(\nabla \overline y_\tau)^{-1} \Vert_{L^\infty( \BBB I_h \EEE \times \Omega)}
         &\leq C \label{Linfty_W2p_def}, \\
    \Vert \hat y_\tau \Vert_{ H^1( \BBB I_h; \EEE H^1(\Omega))\ee}    + \BBB \sup_{t \in [0, T]} \EEE  \mint_{t - h}^t \Vert \partial_t \hat y_\tau(s) \Vert_{L^2(\Omega)}^2 \di s  \BBB  + \sqrt{\eps}   \| \partial_{t} \rb  \hat{y}_{\tau }\|_{L^2(  I_h;  H^3(\Omega))} \EEE
      &\leq C \label{H1_def}, \\
    \Vert \overline \theta_\tau \Vert_{L^\infty(I; L^1(\Omega))}
    + \Vert \overline w_\tau \Vert_{L^\infty(I; L^1(\Omega))} &\leq C \label{Linfty_L1_temp}.
  \end{align}
  Moreover, for each $q \in [1, \frac{d + 2}{d})$ and $r \in [1, \frac{d+2}{d+1})$ we can find constants $C_q>0$ and $C_r>0$ such that
  \begin{align}
    \Vert \overline \theta_\tau \Vert_{L^q(I \ee \times \Omega)}
    + \Vert \overline w_\tau \Vert_{L^q(I \ee \times \Omega)} \leq C_q \label{Lq_temp}, \\
    \Vert \nabla \overline \theta_\tau \Vert_{L^r(I \ee \times \Omega)}
    + \Vert \nabla \overline w_\tau \Vert_{L^r(I \ee \times \Omega)} \leq C_r. \label{W1r_temp}
      \end{align}
  \end{subequations}
  \BBB Moreover, there exist \EEE  $y_h \in\BBB C \EEE (I_h; \Yid) \cap  H^1( \ZZZ I_h; \EEE H^3(\Omega; \R^d))\ee$   and $\theta_h \in L^1(I; W^{1, 1}(\Omega))$ with \BBB   $y_h(t) = y_{0,\eps} + t y_{0,\eps}'$ for all $t \in [-h, 0]$ and \EEE $\theta \geq 0$ a.e.\ such that, up to taking subsequences (not relabeled), as $\tau \to 0$ it holds  that
  \begin{subequations}
  \begin{align}
    \hat y_\tau &\weakly y_h \quad\text{weakly in } H^1(\BBB I_h; \EEE H^3(\Omega;\R^d)), \label{def_H1_conv} \\
    \overline y_\tau &\to y_h \quad\text{strongly in } L^\infty(I_h; W^{1, \infty}(\Omega; \R^{d})) \BBB \ \ \text{ and strongly in } L^\infty(I_h; W^{2, p}(\Omega; \R^{d})), \EEE \label{def_Linfty_W1infty_conv}  \\
    \overline \theta_\tau &\rightharpoonup \theta_h \quad \text{ and } \quad
    \overline w_\tau \rightharpoonup w_h
    \quad\text{weakly in } L^r(I; W^{1,r}(\Omega)) \text{ for any } r \in [1, \tfrac{d+2}{d+1}), \label{temp_Lr_W1r_conv} \\
    \overline \theta_\tau &\to \theta_h \quad \text{ and } \quad
    \overline w_\tau \to w_h
    \quad\text{strongly in } L^s(I \times \Omega) \text{ for any } s \in [1, \tfrac{d+2}{d}), \label{temp_Ls_conv}
  \end{align}
  \end{subequations}
  where $w_h \defas \inten(\nabla y_h, \theta_h)$.    Note that the convergence \BBB in \EEE  \eqref{def_Linfty_W1infty_conv} also hold true for  $\underline y_\tau$ or $\hat y_\tau$ instead of $\overline y_\tau$.
  Moreover, the convergences in \eqref{temp_Lr_W1r_conv} and \eqref{temp_Ls_conv} remain true after replacing $\overline \theta_\tau$ ($\overline w_\tau$) with $\underline \theta_\tau$ ($\underline w_\tau$) or $\hat \theta_\tau$ ($\hat w_\tau$), respectively.
\end{proposition}

\begin{proof}
\BBB Except for  the estimate on $\| \partial_{t} \rb  \hat{y}_{\tau }\|_{L^2(  I_h;  H^3(\Omega))}$, \EEE
  the bounds \eqref{Linfty_W2p_def}--\eqref{Linfty_L1_temp} are a direct consequence of the a priori bounds \eqref{mechen_apriori_bound}--\eqref{velocity_apriori_bound} from Proposition~\ref{prop:minmoves_existence}, \BBB Lemma \ref{lem:pos_det}, \eqref{inten_lipschitz_bounds}, \EEE and  our definition of the different interpolations in time, \BBB where we particularly use that $ \hat{y}_{\tau}(t) = y_{0,\eps} + t y_{0,\eps}'$ for $t \in (-h,0)$. \EEE
\BBB The remaining estimate in \eqref{H1_def} is based on the bound 
\begin{align}\label{into1} 
    \| \partial_{t}   \nabla  \hat{y}_{\tau }\|_{L^2(I_h \times \Omega)} + \sqrt{\eps}   \| \partial_{t} \rb \nabla \Delta \hat{y}_{\tau }\|_{L^2(I_h \times \Omega)}  \EEE \le C 
\end{align}
  provided by \eqref{velocity_apriori_bound}.  Elliptic regularity   for the operator $\Delta$ and the fact that $ \partial_{t} \hat{y}_{\tau } = 0$ on $I_h\times \partial \Omega $   imply 
\begin{align}\label{into2} 
  \Vert \partial_{t} \hat{y}_{\tau } \Vert_{L^2(I_h; H^3(\Omega))} \le C   \Vert \partial_{t} \Delta \hat{y}_{\tau } \Vert_{L^2(I_h; H^1(\Omega))}.
    \end{align}   
  Eventually, we use the interpolation inequality $\Vert \Delta v \Vert_{L^2(\Omega)}  \le C \Vert   \nabla  \Delta v \Vert_{L^2(\Omega)} + C  \Vert   \nabla   v \Vert_{L^2(\Omega)} $ for all $v \in H^3(\Omega;\R^d)$ which can be shown by the standard contradiction-compactness argument. This along with \eqref{into1}--\eqref{into2} indeed yields the remaining bound in \eqref{H1_def}.  \EEE

  The proof of \eqref{Lq_temp}--\eqref{W1r_temp} relies on a proof of a weighted $L^2$-bound on the temperature gradient, \BBB namely \EEE
  \begin{equation}\label{weighted_L2_temp}
    \int_0^T \int_\Omega
      \frac{\eta}{(1 + \overline w_\tau)^{1 + \eta}} |\nabla \overline w_\tau|^2
    \di x \di t \leq C
  \end{equation}
  for any $\eta \in (0, 1)$, where $C$ is a constant only depending on $T$.
  We refer to e.g.~\cite[Theorem 3.20]{tve} for further details.
  It \BBB thus \EEE remains to show \eqref{weighted_L2_temp}. \BBB
As noticed in the proof of Proposition \ref{prop:existence_thermal_step},  our thermal step coincides with the one from \cite{tve} up to adding the term $\eps |\delta_{\tau}   \nabla \Delta y^{(k)}_{\tau} |^{2} \wedge \tau^{-1}$ to the dissipation $ \drate(\nabla \yst{k-1}, \ddif \nabla \yst k, \tst{k-1})$. \EEE The proof of the weighted $L^2$-bound \eqref{weighted_L2_temp} in \cite[Lemma 3.19]{tve} relies only on a uniform $L^1$-bound on  \BBB the dissipation.   \EEE  In view of  \eqref{H1_def}, a uniform  $L^1$-bound (in $\tau$, $h$, \BBB and $\eps$) \EEE is still  available in the current setting and the arguments in   \cite[Lemma 3.19]{tve} also apply here.

\BBB The a priori bounds along with a diagonal sequence argument show \eqref{def_H1_conv}--\eqref{temp_Ls_conv}. More precisely, \EEE     the convergence \eqref{def_H1_conv} is a straightforward consequence of the a priori \BBB bound \eqref{H1_def} and \EEE Banach's selection principle. \BBB The convergences in \eqref{def_Linfty_W1infty_conv} follow from the embeddings $ H^3(\Omega; \R^d) \subset \subset  W^{2, p}(\Om; \R^{d}) \subset \subset W^{1, \infty}(\Om; \R^{d})$ (recall that $p \in (3, 6)$ for $d=3$), and \eqref{H1_def} along with the  Aubin-Lions' lemma.  \ZZZ Next, \eqref{temp_Lr_W1r_conv} follows from \eqref{Lq_temp}--\eqref{W1r_temp}. Eventually, \EEE    the convergence in \eqref{temp_Ls_conv} \BBB and the identification  $w_h = \inten(\nabla y_h, \theta_h)$ \EEE can be shown using the bounds \eqref{Lq_temp}--\eqref{W1r_temp}, the Aubin-Lions' lemma, and interpolation with the bound in \eqref{Linfty_L1_temp}.
  For further details we refer, e.g., to \cite[Lemma 4.2]{tve}.
\end{proof}
\BBB
 
We proceed with further regularity results for $\hat{y}_\tau$ which hinge on a special case of elliptic regularity, see  Lemma \ref{lem:ellip_reg_spec} in the appendix. This will allow us to prove better regularity for the temperature variable, see Corollary \ref{prop:a_priori_bounds2} below, and it will also be instrumental for the passage $\eps \to 0$ in Section \ref{sec: epstozero}.
 
\begin{proposition}[Higher regularity of the deformation]\label{cor:higher_reg_y}
 Let $T, \, h, \, \eps > 0$ and $\tau \in (0,\tau_0)$ be as in Proposition~\ref{prop:minmoves_existence}. 
  Then, $\hat{y}_\tau \in L^2( I;    H^4(\Omega; \R^d))$, $\partial_t\hat{y}_\tau \in L^2(I;   H^5(\Omega; \R^d))$,  and for each $t \in  I $ the time derivative   $\partial_t \hat{y}_\tau$ satisfies the   boundary conditions  
  \begin{equation}\label{higher_bdry_cond_y}
    \partial_\nu \Delta  \partial_t \hat{y}_\tau(t) = \Delta \partial_t \hat{y}_\tau(t) = \Deltatwo \partial_t \hat{y}_\tau(t) = 0 \qquad \text{$\haus^{d-1}$-a.e.~in } \partial \Omega.
  \end{equation}
\end{proposition}

Note that $\hat{y}_\tau$ only has $H^4$-regularity in space due to the regularity of the initial condition $y_{0,\eps}$, see \eqref{eq: yidre}. The next lemma provides some useful bounds on $\Delta \hat{y}_{\tau}$ and on $\Delta^{2} \hat{y}_{\tau}$, which will be particularly crucial for passing to the limit  $\eps \to 0$ in Section~\ref{s:vanishing-regularisation}. \BBB Define
\begin{align}\label{r-def}
 \varrho =  \frac{2p}{4p - (p-2)d }
 \end{align}
and note that $\varrho \in (\frac{1}{2},1)$ since $p >2$ for $d=2$ and $p \in (3,6)$ for $d=3$.   \EEE

\begin{lemma}[Bounds from regularity]
\label{l:laplace-regularity} 
 \BBB  Let $T, \, h, \, \eps > 0$. Then, for $\tau_0$ sufficiently small depending on $\eps$ and $h$ such that Proposition \ref{prop:minmoves_existence} is applicable  and for    $\tau \in (0,\tau_0)$,  there exists a constant $C>0$ only depending on $T$,  $y_{0,\eps}$, $y_{0,\eps}'$, $\theta_{0}$,   $f$, and $\theta_\flat$ such that 
\begin{align}
\label{e:lap-reg-mu}
  \| \Delta \overline{y}_{\tau}  \|_{L^{2} (I; H^{1}(\Om))}  + \|   |\Delta \overline{y}_{\tau}|^{\frac{p-2}{2}} \nabla(\Delta \overline{y}_{\tau})  \|_{L^{2} (I \times \Omega)}   & \leq C  \big( 1 + \sqrt{\eps } \| y_{0,\varepsilon}\|_{H^{4}(\Om)} + \| y'_{0, \varepsilon}\|_{H^{1}( \Om)} \big), 
 \\
 \label{e:lap-reg-mu1.7}
 \|  \,  \Deltatwo  \hat{y}_{\tau}  \|_{L^{2}(I \times \Omega)} 
& \leq C  \eps^{-1/2} \big( 1 +   \sqrt{\varepsilon} \| y_{0, \eps}\|_{H^{4}(\Om)}  + \| y'_{0, \eps}\|_{H^{1}(\Om)} \ee \big) ,
  \\
\label{e:lap-reg-mu1.5}
\|  \nabla {D H}(\Delta \overline{y}_\tau)\|_{L^2(I;L^{p'}(\Omega))} & 
 \leq C \big( 1 +   \sqrt{\varepsilon} \| y_{0, \eps}\|_{H^{4}(\Om)}  + \| y'_{0, \eps}\|_{H^{1}(\Om)} \ee \big)  , 
  \\
\label{e:lap-reg-mu2}
 \|  \,  \Delta \partial_{t} \hat{y}_{\tau}  \|_{L^{2}(I;H^2(\Omega))} 
& \leq C  \varepsilon^{-\frac{1+\varrho}{2}}  \big( 1 +   \sqrt[4]{\varepsilon} \| y_{0, \eps}\|_{H^{4}(\Om)}  + \ZZZ \| y'_{0, \eps}\|_{H^{1}(\Om)}   \EEE \big).
\end{align}
\end{lemma}
\BBB
We postpone the proofs of the two results to the end of the subsection and first present the following consequence.

\begin{corollary}[Further a priori bounds]\label{prop:a_priori_bounds2}
  Let $T, \, h, \, \eps > 0$ and $\tau_0$ be as in Proposition \ref{prop:minmoves_existence}.
  Then, there exists a constant $C_\eps>0$ only depending on $T$, $\eps $, $y_{0,\eps}$, $y_{0,\eps}'$,  $\theta_{0}$,   $f$, and $\theta_\flat$  such that for all $\tau \in (0, \tau_0)$ with $T /  h \ee \in \N$ and $h / \tau \in \N$ the following bounds hold true:
      \begin{align}
    \Vert  \overline \theta_\tau \Vert_{L^2(I; H^1(\Omega))}
    + \Vert \overline w_\tau \Vert_{L^2(I;H^1(\Omega))} &\leq C_\eps, \label{W1r_temp-new}\\    
\Vert  \hat{w}_\tau \Vert_{H^1(I; (H^1(\Omega))^*)} &\leq C_\eps \label{W1r_temp-new2},\\
 \Vert  \hat{y}_\tau \Vert_{L^2(I; H^4(\Omega))} &\leq C_\eps \label{y-new2}.       
    \end{align}
In particular,   $\theta_h$ and $w_h$ from Proposition \ref{prop:a_priori_bounds} satisfy $\theta_h, w_h \in L^2(I; H^1(\Omega)) $ and $w_h\in H^1(I; (H^1(\Omega))^*)$   with  $w_h(0) =  \inten(\nabla y_{0,\eps}, \theta_0)$. Moreover, $y_h$ lies in $ L^\infty(  I;  \Yidreg)$.
\end{corollary}

\begin{proof}
First, by elliptic regularity along with the boundary conditions $\hat{y}_\tau = \id$ and $\Delta \hat{y}_\tau = 0$ on $I_h \times \partial \Omega$ (see \eqref{def_Yid}, \eqref{eq: yidre}, and \eqref{higher_bdry_cond_y}), and the fact that $ \Omega$ has  $C^5$-boundary, we get 
$$\Vert \hat{y}_\tau(t) \Vert_{H^4(\Omega)} \le C \Vert \Delta \hat{y}_\tau(t) \Vert_{H^2(\Omega)}, \quad \quad  \Vert \Delta \hat{y}_\tau(t) \Vert_{H^2(\Omega)} \le C \Vert \Delta^2 \hat{y}_\tau(t) \Vert_{L^2(\Omega)}  $$      
 for a.e.\ $t \in I_h$. This along with \eqref{e:lap-reg-mu1.7} shows  \eqref{y-new2}. Moreover, $y_h \in L^\infty(  I;  \Yidreg)$ (see \eqref{eq: yidre}) follows from   \eqref{H1_def}, \eqref{higher_bdry_cond_y}, \eqref{e:lap-reg-mu2},  and the fact that $y_{0,\eps} \in \Yidreg$. \EEE
 
 Due to \eqref{e:lap-reg-mu2}, we find \ste that $ \eps | \partial_t \nabla \Delta \hat{y}_\tau |^{2}$ is bounded in $L^2(I_h;L^2(\Omega))$ for a bound depending on $\eps$, but independent of $\tau$ and $h$. \BBB Therefore, the term $
    \pl_F \cplpot(\nabla \underline{y}_\tau, \underline{\theta}_\tau) : \partial_t \nabla \hat{y}_\tau
    + \drate(\nabla \underline{y}_\tau, \partial_t \nabla \hat{y}_\tau, \underline{\theta}_\tau)  + \eps |\partial_t   \nabla \Delta \hat{y}_\tau |^{2}  \wedge \tau^{-1}$ appearing in the second line of \eqref{thermal_step_el} is bounded in \ste $L^{2} (I_{h};L^2(\Omega))$ \BBB for a bound depending on $\eps$, but independent of~$\tau$ and~$h$. This regularity allows us to apply the a priori estimates in \cite[Proposition 4.2]{tve_orig}. This yields \eqref{W1r_temp-new}--\eqref{W1r_temp-new2}, and then the regularity of the limits $\theta_h$ and $w_h$ is a direct consequence of weak compactness. By \cite[Lemma 4.5(iii)]{RBMFMKLM} we get $\theta_h, w_h \in C(I;L^2(\Omega)) $  which along with   \ref{C_regularity} and $y_h \in C  (I_h; \Yid)$  also shows $w_h(0) =  \inten(\nabla y_{0,\eps}, \theta_0)$. This concludes the proof.        
\end{proof}

  \BBB
  We now come to the proofs of Proposition \ref{cor:higher_reg_y} and Lemma \ref{l:laplace-regularity}. As they are purely of technical nature, the reader might want to skip these proofs on first reading of the paper.  
  
\EEE

\begin{proof}[Proof of Proposition \ref{cor:higher_reg_y}] 
We recall the notation in \eqref{interpolations} and for convenience we drop the index $\tau$ in the entire proof. \ZZZ As a preliminary step, we first show $\hat{y} \in H^1(I; H^4  (\Omega; \R^d))$, and afterwards the statement.

  \textit{Step 1 ($\hat{y} \in H^1(I; H^4  (\Omega; \R^d))$):} \EEE    We show that, if 
$\overline{y} \in L^2(I;  W^{2,q}   (\Omega; \R^d))$ for some $q \in [ p, 2(p-1)]$, then 
\begin{equation}\label{neumann_timeder0} 
\hat{y} \in H^1(I;  W^{4,\frac{q}{p-1}}  (\Omega; \R^d)), \quad \quad \overline{y} \in L^2(I;   W^{2,q(1+\eta)}  (\Omega; \R^d)),
\end{equation}
where  $\eta \defas \infty$ for $d=2$ and $\eta \defas  \frac{5}{p-1} - 1>0$ for $d=3$, as well as 
\begin{equation}\label{neumann_timeder}
  \partial_\nu \Delta \partial_t \hat{y}(t) = 0 \qquad \text{\BBB $\haus^{d-1}$-a.e.~on } \partial \Omega.
\end{equation}
Once this has been shown,   this argument can first be applied for $q = p $ by \eqref{Linfty_W2p_def}, and then by  a bootstrapping argument, after a finite number of repetitions depending on $\eta$, we get $\hat{y} \in H^1(I; H^4  (\Omega; \R^d))$. 

Let us show  \eqref{neumann_timeder0}--\eqref{neumann_timeder}. Suppose that $\overline{y} \in L^2(I;  W^{2,q}   (\Omega; \R^d))$ for some $ q \in [ p, 2(p-1)]$.  For $t \in I$, we define $g^{\rm 3rd}(t) \in X_q^*$ in the dual space of $X_q\defas W^{2,(\frac{q}{p-1})'}(\Omega; \R^d) \cap H^1_0(\Omega; \R^d)$ by 
\begin{align}\label{eq: g def}
 \langle g^{\rm 3rd}(t), z \rangle  &  \defas  \int_\Omega   D H(\Delta \overline{y}(t)) \cdot \Delta z
      + \Big(         \pl_F \felpot(\nabla \overline{y}(t), \underline{\theta}(t))
        + \pl_{\dot F} \disspot(\nabla \underline{y}(t), \partial_t \nabla \hat{y}(t), \underline{\theta}(t))
      \Big) : \nabla z \di x \notag \\
      & \quad - \int_\Omega \overline{f}(t) \cdot z \di x  + \frac{\rho}{h}   \int_{\Omega} (\partial_t \hat{y}(t) - \partial_t \hat{y}(t - h)) \cdot z \di x  
      \end{align}
for all $z \in X_q$.    By \ref{H_bounds} and the assumption \ste $\overline{y} \in  L^2 (I;W^{2,q}(\Omega;\R^d))$ \BBB  we get 
\ste $D H(\Delta \overline{y}) \in L^2 (I;L^{\frac{q}{p-1}}(\Omega;\R^d))$  \BBB and thus    $g^{\rm 3rd} \in L^2(I;X_q^*)$ by the bounds from Proposition \ref{prop:a_priori_bounds}.   Fixing $z \in X_q \cap H^3(\Omega;\R^d)$   and using   $z$ as a test function in \eqref{mechanical_step_el}, \EEE we derive that
 \begin{align}\label{z_pde}
  &-\eps \int_{\Om}  \nabla \Delta \BBB \partial_t \hat{y}(t) \EEE : \nabla \Delta z \di x   =   \langle g^{\rm 3rd}(t), z \rangle  ,
\end{align}
i.e., $g^{\rm 3rd}$ represents the regularization of third order. \EEE Due to \eqref{z_pde}, for every $t \in \BBB I\EEE$ we can apply Lemma \ref{lem:ellip_reg_spec}\ref{item:ellip_reg_H4} for $u \defas \partial_t \hat{y}(t)$.
With \eqref{u_H4_bound}--\eqref{neumann_laplace_u} this shows $\partial_t \hat{y} \in L^2(\BBB I ; W^{4,\frac{q}{p-1}}\EEE (\Omega; \R^d))$ as well as  \eqref{neumann_timeder}. As \BBB $y_{0, \eps} \in \Yidreg$,  the above statements together with \eqref{eq: yidre} directly lead to $\hat{y}  \in H^1(I;  W^{4,\frac{q}{p-1}}    (\Omega; \R^d))$. By Sobolev embedding we get that $\overline{y} \in  L^2(I;  W^{2, r}   (\Omega; \R^d))$, where $r = (\frac{q}{p-1})^{**}$.  Since $q \ge p$ and thus $\frac{q}{p-1} \ge \frac{6}{5}$, we get $r=\infty$ for $d=2$ and for $d=3$ we have $r = (3\frac{q}{p-1})(3-2\frac{q}{p-1})^{-1}  \ge \frac{5q}{p-1} =  q (1+\eta)$. This concludes the proof of \eqref{neumann_timeder0}.

\ZZZ   \textit{Step 2 (Proof of the statement):}  \EEE From now on we can suppose that $\hat{y} \in H^1(I; H^4  (\Omega; \R^d))$. Due to this  \EEE improved regularity of $\hat{y}$, we have for each $t \in I$ \BBB and any \EEE $z \in C^\infty_c(\Omega; \R^d)$ that
\begin{equation}\label{another integration by parts}
  \int_\Omega  D H(\Delta \overline{y} \BBB (t) \EEE ) \cdot \Delta z \di x
  = -\int_\Omega \nabla ( D H(\Delta \overline{y} \BBB (t) \EEE  )) : \nabla z \di x.
\end{equation}
\BBB Recalling \eqref{e:Hdef}--\eqref{e:psi}, an elementary computation for a general $v \in H^4 (\Omega;\R^d)$  yields that pointwise a.e.~in $ \Omega$ \EEE  it holds that
\begin{equation}\label{nabla2x_hypot}
  \nabla ( D H(\Delta v)) = \begin{cases}
    2 \nabla \Delta v &\text{if } p |\Delta v|^{p-2} \leq 2, \\
    p (p-2) |\Delta v|^{p-4} \Delta v \otimes \big( ( \nabla  \Delta v)^T \Delta v \big)    + p |\Delta v|^{p-2} \nabla \Delta v &\text{else}.
  \end{cases}
\end{equation}
\ZZZ Taking \EEE  the Frobenius norm on both sides of \eqref{nabla2x_hypot} we see \BBB a.e.~in $ \Omega$ that \EEE
\begin{equation}\label{nabla2x_hypot2}
  |\nabla ( D H(\Delta v))| \leq \begin{cases}
    2 |\nabla \Delta v| &\text{if } p |\Delta v|^{p-2} \leq 2, \\
    p (p-1) |\Delta v|^{p-2}  |\nabla \Delta v| &\text{else}.
  \end{cases}
\end{equation}
\ZZZ By \EEE Sobolev embedding we have that \BBB $H^2(\Omega; \R^d) \subset  L^\infty(\Omega; \R^d)$ for $d=2,3$. \EEE
Hence, \ZZZ $\Delta \overline{y} \in L^\infty(I \times \Omega; \R^d)$. \EEE In particular, this shows $\nabla ( D H(\Delta \overline{y})) \in L^2(I \times \Omega; \BBB  \R^{d\times d}) \EEE $ and therefore, \BBB recalling the definition in \eqref{eq: g def} and \eqref{another integration by parts}, we derive \EEE $g^{\rm 3rd} \in L^2(I; H^{-1}(\Omega; \R^d))$ by arbitrariness of $z$, where we have used that $C^\infty_c(\Omega; \R^d)$ is dense in $H^1_0(\Omega; \R^d)$.  \BBB (Now, $\langle \cdot, \cdot \rangle$  stands for the dual pairing between $H^{1}_0(\Omega; \R^d)$ and  $H^{-1}(\Omega; \R^d)$.) \EEE By Lemma \ref{lem:ellip_reg_spec}\ref{item:ellip_reg_H5}, \eqref{neumann_timeder}, \BBB and the bounds from Proposition \ref{prop:a_priori_bounds} \EEE we derive that $\partial_t \hat{y} \in L^2(I; H^5(\Omega; \R^d))$ and \BBB that $\partial_t \hat{y}$ satisfies
\begin{align}\label{a bound}
\Vert \partial_t \hat{y} \Vert_{L^2(I; H^5(\Omega))}  \le C_{\eps,h}
\end{align}
for a constant $C_{\eps,h}$ depending on $\eps$ and $h$. Moreover, we have  
 \EEE the boundary condition
\begin{equation}\label{dirichlet_laplace2x_timeder}
\Deltatwo \partial_t \hat{y}(t) = 0 \qquad \text{$\haus^{d-1}$-a.e.~on } \partial \Omega \ \ \text{for $t \in I$.}
\end{equation}
\BBB To conclude the proof, it \EEE remains to show  $\Delta \partial_t \hat{y}(t) =   0$  $\haus^{d-1}$-a.e.~on $\partial \Omega$ for $t \in \ZZZ I \EEE$.  \BBB Defining for $z \in H\defas H^3(\Omega;\R^d) \cap H^1_0(\Omega;\R^d)$
\begin{align*}
\langle g^{\rm 2nd}(t), z \rangle &  \defas  \EEE  -\eps \int_{\Om} \nabla \Delta \partial_{t} \hat{y}(t) : \nabla \Delta z \di x
  - \int_\Omega \Big( 
      \pl_F \felpot(\nabla \overline{y}(t), \underline{\theta}(t))
      + \pl_{\dot F} \disspot(\nabla \underline{y}(t), \partial_t \nabla \hat{y}(t), \underline{\theta}(t))
    \Big) : \nabla z \di x \\
  &\phantom{\quad=}\quad
    + \int_\Omega \overline{f}(t) \cdot z \di x
    - \frac{\rho}{h} \int_{\Omega} (\partial_t \hat{y}(t) - \partial_t \hat{y}(t - h)) \cdot z \di x,
 \end{align*}
we can rewrite \eqref{eq: g def}--\eqref{z_pde} \BBB as  
\begin{equation}\label{eq: second order}
  \int_\Omega  D H(\Delta \overline{y}(t)) \cdot \Delta z \di x = 
  \langle g^{\rm 2nd}(t), z \rangle
\end{equation}
for $t \in I$, i.e., $g^{\rm 2nd}$ represents the term with second derivative. \EEE Using the improved regularity of $\partial_t\hat{y}$ and the boundary conditions \eqref{neumann_timeder} and  \eqref{dirichlet_laplace2x_timeder}, the first integral \BBB in the definition of $g^{\rm 2nd}(t)$ \EEE can be written as
\begin{align*}
  \int_\Omega \nabla \Delta  \partial_t \hat{y}(t) : \nabla \Delta z \di x
  &= - \int_\Omega \Deltatwo  \partial_t \hat{y}(t) \cdot \Delta z \di x + \int_{\partial \Omega} \partial_\nu \Delta   \partial_t \hat{y}(t) \cdot \Delta z \di \haus^{d-1}\notag \\
  &= \int_\Omega \nabla \Deltatwo   \partial_t \hat{y}(t) : \nabla z \di x - \int_{\partial \Omega} \Deltatwo   \partial_t \hat{y}(t) \cdot \partial_\nu z \di \haus^{d-1}  = \int_\Omega \nabla \Deltatwo   \partial_t \hat{y}(t) : \nabla z \di x
\end{align*}
\BBB for each $z \in H$.  As $\partial_t \hat{y} \in L^2(I; H^5(\Omega; \R^d))$,  this shows   $g^{\rm 2nd}(t) \in H^{-1}(\Omega; \R^d)$ \ZZZ for each $t \in I$. \EEE Now, it is standard to find $v(t) \in H^1_0(\Omega;\R^d)$ such that $\langle g^{\rm 2nd}(t) , z \rangle =   \int_\Omega v(t) \cdot \Delta z \di x$
for all $z \in H^2(\Omega;\R^d ) \cap H^1_0(\Omega;\R^d)$, see 
\eqref{NNNNN}--\eqref{NNNNN2} below for details. Given an arbitary $\varphi \in L^2(\Omega;\R^d)$ and choosing $z \in H^2(\Omega;\R^d) \cap H^1_0(\Omega;\R^d)$ with $\Delta z = \varphi$, this along with \eqref{eq: second order} shows
$$  \int_\Omega  \big( D H(\Delta \overline{y}(t)) - v(t) \big) \cdot \varphi \di x = 0. $$
This yields $D H(\Delta \overline{y}(t)) = v(t)$ a.e.\ in $\Omega$ and thus   it holds that ${D H}(\Delta \ZZZ \bar{y} \EEE (t)) = 0$ $\mathcal{H}^{d-1}$-a.e.~on~$\partial \Omega$ for $t \in I$. Since  \EEE $${D H}(v) = \max\{2, p |v|^{p-2}\}v = 0 \ \Leftrightarrow \ v = 0 \in \R^d$$ 
\BBB and $y_{0,\eps} \in \Yidreg$, \EEE this concludes the proof of~\eqref{higher_bdry_cond_y}.  
\end{proof}
\ee

\begin{proof}[Proof of Lemma \ref{l:laplace-regularity}] 
\BBB  For notational convenience,  we define the weighted Bochner space 
$$\Vert \cdot \Vert _{L^2_T(I; L^q(\Omega))} \defas  \Vert \sqrt{T-t}  \, \cdot \Vert_{L^2(I; L^q(\Omega))},$$ and employ a similar notation for Sobolev spaces.  We first note that it is not restrictive to establish the bounds \eqref{e:lap-reg-mu}--\eqref{e:lap-reg-mu2} only for the weighted space. Indeed, by extending $\theta_\flat$ and $f$ suitably on the time interval $[T,T+\eta]$ for some $\eta>0$ small, we can establish time-discrete solutions on the time interval $[-h,T+\eta]$, see Proposition~\ref{prop:minmoves_existence}. Then, a control on $\Vert \cdot \Vert _{L^2_{T+\eta}([0,T+\eta]; L^q(\Omega))}$ will directly imply a control on $\Vert \cdot \Vert _{L^2(I; L^q(\Omega))}$ for a constant additionally depending on $\eta$. To simplify notation, we use the interval $\ZZZ I= \EEE [0,T]$ instead of $[0,T+\eta]$ in the sequel.  As in the  proof of Proposition \ref{cor:higher_reg_y}, we omit the index $\tau$.

\BBB
  \textit{Step 1 (Proof of \eqref{e:lap-reg-mu} and \eqref{e:lap-reg-mu1.7}):}   
As $\partial_t \hat{y} \in L^2(I; H^4(\Omega; \R^d))$ with $\partial_\nu\Delta \partial_t \hat{y}(t) = 0$ $\haus^{d-1}$-a.e.~on $\partial \Omega$ for $t \in I$ by Proposition \ref{cor:higher_reg_y}, by an integration by part we can rewrite the $\eps$-dependent term in  \eqref{mechanical_step_el} as 
\begin{align}\label{IBP}
  \int_\Omega\eps   \nabla\Delta \partial_t  \hat{y} : \nabla \Delta z \ee \di x & =  -    \int_\Omega\eps     \Deltatwo  \partial_t \hat{y} : \Delta z \ee \di x +   \int_{\partial \Omega}\eps    \partial_\nu\Delta \partial_t  \hat{y} :  \Delta z \ee \di \haus^{d-1}   =  -    \int_\Omega\eps    \Deltatwo \partial_t  \hat{y} : \Delta z \ee \di x
  \end{align}
 for each $t \in    ((k-1)\tau, k\tau)$ and $z \in \ZZZ H^3(\Omega;\R^d) \EEE \cap H^1_0(\Omega;\R^d)$. \ZZZ Thus, by approximation we can test \eqref{mechanical_step_el} with functions in $H^2(\Omega;\R^d) \cap H^1_0(\Omega;\R^d)$. \EEE As   $\Delta \hat{y} \in L^2(I; H^2(\Omega; \R^d))$ with $\Delta \hat{y}(t) = 0$ $\haus^{d-1}$-a.e.~on $\partial \Omega$ for   $t \in I$ by Proposition \ref{cor:higher_reg_y} \ZZZ  and the fact that $y_{0,\eps} \in \Yidreg$ (see \eqref{eq: yidre}), \EEE   we discover that $z(t, x) \defas (T - t)\Delta \hat{y}(t,x)$ for $t \in ((k-1)\tau, k\tau)$  is a  valid test function \BBB in  \eqref{mechanical_step_el}. After summation and  rearranging terms, and employing \eqref{IBP} \EEE this  yields
\begin{equation}\label{first_test_higher_reg}
\begin{aligned}
  &\eps \int_I \int_{\Om}   \Deltatwo \partial_t \hat{y} :  \Deltatwo \hat{y} (T - t) \di x \di t
  - \intQ  D H(\Delta \overline{y}) \cdot \Deltatwo \hat{y} (T - t) \di x \di t \\
  &\quad= \intQ
    \big(
      \pl_F \felpot(\nabla \overline{y}, \underline{\theta})
      + \pl_{\dot F} \disspot(\nabla \underline{y}, \partial_t \nabla \hat{y}, \underline{\theta})
    \big) : \nabla \Delta\hat{y} (T - t) \di x \di t \\
  &\phantom{\quad=}\quad - \intQ \overline{f} \cdot \BBB \Delta \hat{y}(t) \EEE (T - t) \di x \di t
    + \frac{\rho}{h} \int_I \int_\Omega \big(\partial_t \hat{y}(t) - \partial_t \hat{y}(t - h)\big) \cdot \Delta \hat{y}(t) (T - t) \di x \di t.
\end{aligned}
\end{equation}
\BBB (The advantage of multiplying with $(T-t)$ will become apparent in \eqref{e:lap-reg-mu-int} below.) Using \eqref{higher_bdry_cond_y} \ZZZ  and   $y_{0,\eps} \in \Yidreg$  \EEE we integrate by parts in the second term on the left-hand side above, which  leads to
\BBB
\begin{align*}
  &-\intQ  D H(\Delta \overline{y}) \cdot \Deltatwo \hat{y}(T - t) \di x \di t \\
  &\quad= \intQ \nabla ( D H(\Delta \overline{y})) : \nabla \Delta \hat{y} (T - t) \di x \di t
    - \int_I \int_{\partial \Omega}  D H(\Delta \overline{y}) \cdot \partial_\nu \Delta \hat{y} (T - t) \di \haus^{d-1} \di t \nonumber \\
      &\quad \ge  \intQ \nabla ( D H(\Delta \overline{y})) : \nabla \Delta \overline{y} (T - t) \di x \di t
    - C\tau\Vert  \nabla ( D H(\Delta \overline{y})) \Vert_{L^2(I \times \Omega )} \Vert \nabla \Delta \partial_t \hat{y} \Vert_{L^2(I \times \Omega )},
 \end{align*}
where in the last step we exploited the definition in \eqref{interpolations}. In view of \eqref{nabla2x_hypot}, we get  that 
\begin{align*}
& \intQ \nabla ( D H(\Delta \overline{y})) : \nabla \Delta \overline{y} (T - t) \di x \di t \\ 
& = \intQ \max\{2, p |\Delta \overline{y}|^{p-2} \} |\nabla \Delta \overline{y}|^2 (T - t) \di x \di t    + \intQ \indic_{\{p |\Delta \overline{y}|^{p-2} \geq 2\}} p(p-2) |(\nabla \Delta \overline{y})^{T} \Delta \ZZZ \overline{y}|^{2} |\Delta \overline{y}|^{p-4}  \EEE  (T - t)   \, \di x \, \di t \nonumber \\  
   & \geq
  \intQ \max\{2, p |\Delta \overline{y}|^{p-2} \} |\nabla \Delta \overline{y}|^2 (T - t) \di x \di t. \nonumber
\end{align*}
 By $y_{0,\eps} \in \Yidreg$  and \eqref{a bound} we find \ste $\Vert \Delta \overline{y} \Vert_{L^\infty (I;L^\infty(\Omega))} \le C_{\eps,h}$ \BBB and $\Vert \overline{y} \Vert_{L^2(I; H^3(\Omega))} \le C_{\eps,h}$. Then, again using  \ZZZ \eqref{nabla2x_hypot}--\eqref{nabla2x_hypot2} \EEE and \eqref{H1_def} we eventually find  
\begin{align}\label{ibp_H_term_higher_reg}
 -\intQ  D H(\Delta \overline{y}) \cdot \Deltatwo \hat{y}(T - t) \di x \di t \ge  \intQ \max\{2, p |\Delta \overline{y}|^{p-2} \} |\nabla \Delta \overline{y}|^2 (T - t) \di x \di t - \tau C_{\eps,h}.
 \end{align} \EEE
By the chain rule we write 
\begin{align}\label{ibp_timeder_higher_reg}
   \int_I \int_{\Om} \Deltatwo \partial_t \hat{y} \cdot \Deltatwo \hat{y} (T - t) \di x \di t  & \BBB  =
   \int_I \frac{\di}{\di t}\Big( \frac{1}{2} \int_{\Om} |\Deltatwo \hat{y}|^2 (T - t) \di x \Big) \di t
  + \frac{1}{2} \int_I \int_{\Omega} |\Deltatwo \hat{y}|^2 \di x \di t \nonumber \\
  &= \frac{1}{2} \int_I \int_{\Omega} |\Deltatwo \hat{y}|^2 \di x \di t - \frac{T}{2} \|\Deltatwo y_{0, \eps}\|_{L^2(\Omega)}^2\,.   
\end{align}
By the definition of $\Yid$ and $\hat{y} \in L^\infty( \ZZZ I_h; \EEE \Yid)$, it follows that $\partial_t \hat{y}(t) = 0$ \BBB $\haus^{d-1} $-a.e.~in \EEE  $\partial \Omega$ for  $t \in I$. Hence, integrating by parts  leads to
\begin{align*}
  \int_I \int_\Omega (\partial_t \hat{y}(t) - \partial_t \hat{y}(t - h)) \cdot  \Delta \hat{y} \BBB (t) \EEE (T - t) \di x \di t   
   & = 
  -   \int_I \int_\Omega (\partial_t \nabla \hat{y}(t) - \partial_t \nabla \hat{y}(t - h)) : \nabla \hat{y}(t) (T - t) \di x \di t.
\end{align*}
Now, using the substitution \BBB $t \mapsto  t - h$   \BBB we derive \EEE
\begin{align}\label{ibp_acceleration_higher_order}
  \frac{\rho}{h} \int_I \int_\Omega (\partial_t \hat{y}(t) - \partial_t \hat{y}(t - h)) \cdot \Delta \hat{y}  \BBB (t) \EEE  (T - t) \di x \di t  & = \BBB \Pi \EEE    + \rho  \int_\Omega \nabla y_{0, \eps}' : \mint_0^h \nabla \hat{y}(t) (T - t) \di t \di x \\ & \quad 
  -  \BBB \rho \mint_{T - h}^T \EEE \int_\Omega \partial_t \nabla \hat{y}(t) : \nabla \hat{y}(t) \BBB (T - t) \EEE \di x \di t, \nonumber
\end{align} \BBB
where 
$$ \Pi \defas \rho \int_0^{T - h} (T-t) \int_\Omega \partial_t \nabla \hat{y}(t) : \frac{\nabla \hat{y}(t + h) - \nabla \hat{y}(t)}{h} \di x \di t \BBB - \rho \int_0^{T-h} \int_\Omega \partial_{t} \nabla \hat{y} (t)  \colon  \nabla \hat{y} (t+h) \di x \di t. \EEE $$
Using that $\frac{\nabla \hat{y}(t + h) - \nabla \hat{y}(t)}{h} = \mint_{t}^{t+h} \partial_{t} \nabla \hat{y} (s) \, \di s  $ and applying H\"older's inequality, we can check that
$$|\Pi| \le  C\| \partial_{t} \nabla \hat{y} \|^2_{L^{2}(I \times \Om)} +  C\| \partial_{t} \nabla \hat{y} \|_{L^{2}(I \times \Om)} \|  \nabla \hat{y} \|_{L^{2}(I \times \Om)} \le  C\|  \nabla \hat{y} \|^2_{H^{1}(I; L^{2} (\Om))}.  $$
\EEE Combining \eqref{first_test_higher_reg}--\eqref{ibp_acceleration_higher_order} \BBB then \EEE leads   to
\sa
\begin{equation}
\label{e:1111}
\begin{aligned}
&\intQ \max\{2, p |\Delta \overline{y}|^{p-2} \} |\nabla \Delta \overline{y}|^2 (T - t) \di x \di t
  + \frac{\eps}{2}  \int_I \int_{\Omega} |\Deltatwo \hat{y}|^2   \di x \di t -  \frac{T\eps}{2} \|\Deltatwo y_{0, \eps}\|_{L^2(\Omega)}^2 \\
&\quad\leq 
  \intQ
    \big(
      \pl_F \felpot(\nabla \overline{y}, \underline{\theta})
      + \pl_{\dot F} \disspot(\nabla \underline{y}, \partial_t \nabla \hat{y}, \underline{\theta})
    \big) : \nabla \Delta \hat{y} (T - t) \di x \di t  
  - \intQ \overline{f} \cdot \Delta \hat{y} (T - t) \di x \di t \\ & \qquad 
  \BBB + \tau C_{\eps,h} + C \EEE  \|  \nabla \hat{y} \|^{2}_{\BBB H^{1}(I; L^{2} (\Om))}  
  + \rho \int_\Omega \nabla y_{0, \eps}' : \mint_0^h \nabla \hat{y}(t) (T - t) \di t \di x \\
  &\qquad\quad- \rho \mint_{T - h}^T \int_\Omega \partial_t \nabla \hat{y}(t) : \nabla \hat{y}(t) (T - t) \di x \di t. 
\end{aligned}
\end{equation}
  In view of the bounds   \eqref{Linfty_W2p_def}--\eqref{W1r_temp},   we have that 
\begin{align}\label{that is easier} 
& \sup_{h>0, \, \eps \in (0, 1)} \| \partial_{F} W(\nabla \overline{y}, \underline{\theta}) + \partial_{\dot{F}} R(\nabla \underline{y}, \partial_{t} \nabla \hat{y}, \underline{\theta}) \|_{L^{2} (I \times \Omega)} <+\infty\,,\notag \\
& \sup_{h>0, \, \eps \in (0, 1)} \|  \overline{y}\|_{L^{\infty} (I; W^{2, p} (\Om))} + \BBB \| \hat{y}\|_{H^1 (I; H^{1}(\Om))} \EEE <+\infty\,,
\end{align}
\BBB where we use \ref{W_regularity}, \ref{C_lipschitz}, and \ref{D_quadratic}. Moreover, we choose $\tau_0$ small enough such that $\tau_0 C_{\eps, h} \le 1$.  \EEE Hence, by H\"older's inequality,  \BBB the weighted \EEE Young's inequality, and by the boundary condition $\Delta \hat{y} (t) = 0$ for a.e.~$t \in [0, T]$ and $\mathcal{H}^{d-1}$-a.e.~in~$\partial\Om$ \BBB (see \eqref{eq: yidre} and Proposition \ref{cor:higher_reg_y}),    \EEE we infer from \eqref{e:1111} that   
\begin{equation}\label{e:lap-reg-mu-int}
\BBB \sqrt{\eps} \|   \Deltatwo \hat{y} \|_{L^{2}(I \times \Om)} + \EEE \|   \Delta \overline{y} \|_{L^{2}_T (I; H^{1}(\Om))}
  + \|  \, |\Delta \overline{y}|^{\frac{p-2}{2}} \nabla(\Delta \overline{y}) \|_{L^{2}_T (I; L^2(\Om))} \\
  \leq
  C \big( 1 + \sqrt{\eps } \| y_{0, \eps}\|_{H^4(\Om)} + \| y'_{0, \eps}\|_{H^{1}(\Om)} \big)
\end{equation}
for some constant $C>0$ independent of~$\eps \in (0, 1)$ and $h>0$,  \BBB where the notation $L^{2}_T (I; H^{1}(\Om))$ has been introduced at the beginning of the proof. Here, the factor $(T-t)$ guarantees that the last term on the right-hand side of \eqref{that is easier}  can be controlled \ZZZ uniformly in $h$. \EEE  This shows \eqref{e:lap-reg-mu} and \eqref{e:lap-reg-mu1.7}.  \EEE

\BBB      \textit{Step 2 (Proof of \eqref{e:lap-reg-mu1.5}):} We proceed with \eqref{e:lap-reg-mu1.5}.  In view of   \eqref{nabla2x_hypot2},  it holds that  
\begin{align*}
| \, \nabla ({D H}(\Delta \overline{y}) ) | & \leq C  \big(1+|\Delta \overline{y}|\big)^{\frac{p-2}{2}} \, \Big|\big(1+|\Delta \overline{y}|\big)^{\frac{p-2}{2}} \nabla \Delta \overline{y} \Big| 
\end{align*}
a.e.\ in $\Omega$. By   using H\"older's inequality for $\frac{p'}{2}+\frac{p-2}{2(p-1)}=1$ \BBB we estimate \begin{align}\label{e:lap-reg-mu-int2}
\|    \nabla {D H}(\Delta \overline{y})\|_{L^2_T(I;L^{p'}(\Omega))}^2
&\leq C   \int_I\bigg(\int_\Omega \big(1+ |\Delta \overline{y}|\big)^{\frac{(p-2)p}{2(p-1)}}\big(\sqrt{T - t} \,\big(1+|\Delta \overline{y}|\big)^{\frac{p-2}{2}}
|\nabla \Delta \overline{y}|\big)^{p'}\, \di x\bigg)^\frac{2}{p'}\, \di t \notag
\\
&\leq C   \|1+|\Delta \overline{y}|\|_{L^\infty(I;L^p(\Omega))}^{p-2}\|  (1+
|\Delta \overline{y}|)^{\frac{p-2}{2}} \nabla \Delta \overline{y}\|_{L^2_T(I; L^2(\Omega))}^2.
\end{align}
Note that $|\Delta \overline{y}|$ is  uniformly bounded in $L^{\infty} (I; L^{p}(\Om; \R^{d}))$ by  \eqref{Linfty_W2p_def}. Hence,~\eqref{e:lap-reg-mu1.5} follows from \eqref{e:lap-reg-mu-int}.

 \BBB
  \textit{Step 3 (Proof of \eqref{e:lap-reg-mu2}):}  \EEE
Finally, we prove \eqref{e:lap-reg-mu2}. In \EEE  view of Corollary~\ref{cor:higher_reg_y}, we may test the mechanical equation \BBB  \eqref{mechanical_step_el}  with $ z = (T - t) \partial_{t}\Delta \hat{y}(t)$ for $t \in ((k-1)\tau, k\tau)$.  After summation and  rearranging terms  this  yields \EEE 
\begin{align*}
      &-\eps \int_I \int_{\Om}\partial_{t} \nabla \Delta \hat{y} : \partial_{t} \nabla \Deltatwo \hat{y}(T - t) \di x \di t
  - \intQ  DH(\Delta \overline{y}) \cdot \partial_{t}\Deltatwo \hat{y} (T - t) \di x \di t \\
  &\quad= \intQ
    \big(
      \pl_F \felpot(\nabla \overline{y}, \underline{\theta})
      + \pl_{\dot F} \disspot(\nabla \underline{y}, \partial_t \nabla \hat{y}, \underline{\theta})
    \big) : \partial_{t} \nabla \Delta \hat{y} (T - t) \di x \di t \\
  &\phantom{\quad=}\quad - \intQ \overline{f} \cdot \partial_{t} \Delta \hat{y} (T - t) \di x \di t
    + \frac{\rho}{h} \ZZZ \int_I \EEE \int_\Omega (\partial_t \hat{y}(t) - \partial_t \hat{y}(t - h)) \cdot \partial_{t} \Delta \hat{y}(t) (T - t) \di x \di t.
\end{align*}
Integrating by parts the left-hand side and the last term on the right-hand side  above, by the boundary conditions~$\partial_{\nu}  \Delta \partial_{t} \hat{y} = 0$ and $\partial_{t} \hat{y}= 0$ on $\partial \Om$ for  $t \in I$, we get that
\begin{equation}
    \label{e:last-step-2}
    \begin{aligned}
            &\eps \int_I \int_{\Om}| \partial_{t} \Deltatwo \hat{y} |^{2} (T - t) \di x \di t
  + \intQ \nabla (  DH(\Delta \overline{y}) )  : \ee \partial_{t}\nabla \Delta \hat{y} (T - t) \di x \di t \\
  &\quad= \intQ
    \big(
      \pl_F \felpot(\nabla \overline{y}, \underline{\theta})
      + \pl_{\dot F} \disspot(\nabla \underline{y}, \partial_t \nabla \hat{y}, \underline{\theta})
    \big) : \partial_{t} \nabla \Delta \hat{y} (T - t) \di x \di t \\
  &\phantom{\quad=}\quad - \intQ \overline{f} \cdot \partial_{t} \Delta \hat{y} (T - t) \di x \di t
    - \frac{\rho}{h} \int_0^T \int_\Omega (\partial_t \nabla \hat{y}(t) - \partial_t \nabla \hat{y}(t - h)) \BBB : \EEE \partial_{t} \nabla \hat{y}(t) (T - t) \di x \di t.
    \end{aligned}
\end{equation}
 We denote  the last term on the right-hand side of \eqref{e:last-step-2} (without negative sign) by $\Pi_T$. An expansion yields 
\begin{align} \label{e:last-step-3}
\Pi_T & =     \frac{\rho}{2h} \ZZZ \int_I \EEE \big(  \Vert \partial_t \nabla \hat{y}(t) \Vert^2_{L^2(\Omega)} - \Vert \partial_t \nabla \hat{y}(t-h) \Vert^2_{L^2(\Omega)} + \Vert \partial_t  \nabla \hat{y}(t) - \partial_t  \nabla \hat{y}(t-h) \Vert^2_{L^2(\Omega)}  \big)  (T-t)  \di t \notag \\
&        =  \frac{\rho}{2} \int_{0}^{T-h} \int_{\Om} | \partial_{t} \nabla \hat{y} (t) |^{2}\, \di x \, \di t  + \ee \frac{\rho}{2} \mint_{T-h}^{T} \int_{\Om}  | \partial_{t} \nabla \hat{y}(t)|^{2} ( T - t) \, \di x \, \di t - \frac{\rho}{2} \| \nabla y'_{0, \varepsilon}\|^{2}_{L^{2} (\Om)} \mint_0^{h} (T-t) \di t  \ee \nonumber
        \\
        &
        \quad + \frac{\rho}{2 h} \intQ | \partial_{t} \nabla \hat{y} (t) - \partial_{t} \nabla \hat{y} (t - h) |^{2} ( T - t) \di x \di t\,.
\end{align}
Combining~\eqref{e:last-step-2} and \eqref{e:last-step-3} and applying H\"older inequality, \BBB we deduce that
\begin{align*}
        \varepsilon \|   \partial_{t} \Deltatwo  \hat{y} \|_{L^{2}_T(I; L^2(\Om))}^{2}     & \leq\|  \nabla (DH(\Delta \overline{y})) \|_{L^{2}_T (I; L^{p'}(\Omega))} \| \partial_{t} \nabla \Delta \hat{y} \|_{L^{2}_T(I; L^p(\Omega))}
      \notag  \\
        &\quad \phantom{\leq} + \sqrt{T} \| (  \pl_F \felpot(\nabla \overline{y}, \underline{\theta})
      + \pl_{\dot F} \disspot(\nabla \underline{y}, \partial_t \nabla \hat{y}, \underline{\theta})
    ) \|_{L^{2}(I \times \Omega)} \|  \partial_{t} \nabla \Delta \hat{y} \|_{L^{2}_T(I; L^2(\Omega))}
  \notag  \\
    &\quad \phantom{\leq} + \sqrt{T} \| f \|_{L^{2} ( I \times \Omega ) } \|  \partial_{t}   \Delta \hat{y}\|_{ L^{2}_T(I; L^2(\Omega))}
    + \BBB \frac{\rho T }{2}\| y'_{0, \varepsilon} \|_{H^{1}(\Om)}^{2}\,.  \EEE
\end{align*}
Then, by \eqref{that is easier}--\eqref{e:lap-reg-mu-int2}, H\"older's inequality,  and Poincar\'e's inequality along with $\Delta \ZZZ \partial_t \EEE \hat{y} = 0$ on $I \times \partial \Omega$ we derive 
\begin{align}    \label{e:last-step-4}
 \varepsilon    \|   \partial_{t} \Deltatwo \hat{y} \|_{L^{2}_T(I;L^2(\Omega))}^{2} 
                \leq \ C   ( 1  + \sqrt{\eps } \| y_{0, \varepsilon}\|_{H^{4} (\Om)} + \| y'_{0, \varepsilon}\|_{H^{1}(\Om)})  \|  \partial_{t}  \nabla  \Delta \hat{y}\|_{L^{2}_T(I; L^p(\Om))} + C   \| y'_{0, \varepsilon}\|_{H^{1}(\Om)}^{2}. 
                \end{align}
By  the  Gagliardo-Nirenberg interpolation inequality with $\theta = \frac{p-2}{2p}d$ ($\theta \in (0,1)$ as $p \in (3,6)$ for $d=3$),  \ZZZ see e.g.~\cite[Theorem 1.24]{Roubicek-book}) for \EEE  $r= p $, $\beta = 0$, $k=1$, and  $p = q = 2$, we get
$$ \|   \partial_{t} \nabla  \Delta \hat{y}\|_{L^{2}_T(I; L^p(\Om))} \le  \|   \partial_{t} \nabla  \Delta \hat{y}\|_{L^{2}_T(I; L^2(\Om))} +  \|   \partial_{t} \nabla  \Delta \hat{y}\|_{L^{2}_T(I; L^2(\Om))}^{1-\theta} \|   \partial_{t}  \nabla^2 \Delta \hat{y}\|_{L^{2}_T(I; L^2(\Om))}^\theta. $$
By elliptic regularity  for the operator $\Delta$ and the fact that $\Delta \partial_t \hat{y} = 0$ on $I \times \partial \Omega$ we get $\|   \partial_{t}  \nabla^2 \Delta \hat{y}\|_{L^{2}_T(I; L^2(\Om))} \le C \|   \partial_{t}   \Deltatwo \hat{y}\|_{L^{2}_T(I; L^2(\Om))}$. Then, using the weighted Young's inequality for $\lambda>0$ and exponent $\frac{2}{\theta}$ we get 
$$ \|   \partial_{t}  \nabla  \Delta \hat{y}\|_{L^{2}_T(I; L^p(\Om))} \le  \|   \partial_{t} \nabla  \Delta \hat{y}\|_{L^{2}_T(I; L^2(\Om))} +  C(\eps\lambda)^{-\frac{\theta}{2-\theta}}  \|   \partial_{t}  \nabla  \Delta \hat{y}\|_{L^{2}_T(I; L^2(\Om))}^{\frac{2(1-\theta)}{2-\theta}}     +  C\eps\lambda \|   \partial_{t}  \Deltatwo \hat{y}\|_{L^{2}_T(I; L^2(\Om))}^2. $$
The bound $\sqrt{\eps}   \| \partial_{t} \hat{y} \|_{L^2(  I;  H^3(\Omega))} \le C$ given by   \eqref{H1_def}  and the fact that $\frac{1}{2-\theta} =  \varrho  \in (\frac{1}{2},1)$ (see \eqref{r-def}) yield 
$$ \|   \partial_{t} \nabla  \Delta \hat{y}\|_{L^{2}_T(I; L^p(\Om))} \le C\eps^{-1/2} +  C_\lambda \eps^{-\frac{\theta}{2-\theta}}  \eps^{-\frac{1-\theta}{2-\theta}}     +  C\eps\lambda \|   \partial_{t}  \Deltatwo \hat{y}\|_{L^{2}_T(I; L^2(\Om))}^2 \le C_\lambda \eps^{-\varrho}  +  C\eps\lambda \|  \partial_{t}  \Deltatwo \hat{y}\|_{L^{2}_T(I; L^2(\Om))}^2.   $$
Multiplying inequality~\eqref{e:last-step-4} by \ste $\eps^{\varrho}$ \BBB and  choosing $\lambda$ sufficiently small we get 
\begin{align*}
 \varepsilon^{1+\varrho}   \|  \partial_{t} \Deltatwo \hat{y} \|_{L^{2}_T(I;L^2(\Omega))}^{2} 
               & \leq \ C   ( 1  + \sqrt{\eps } \| y_{0, \varepsilon}\|_{H^{4} (\Om)} + \| y'_{0, \varepsilon}\|_{H^{1}(\Om)})  + C \eps^\varrho \| y'_{0, \varepsilon}\|_{H^{1}(\Om)}^{2} 
    \end{align*}
for some constant~$C>0$ independent of~$\varepsilon \in (0, 1)$ and of~$ h>0$.  \EEE This concludes the proof of~\eqref{e:lap-reg-mu2} \BBB by using \eqref{higher_bdry_cond_y} and an elliptic regularity estimate. \EEE
\end{proof}

\subsection{Proof of Theorem \ref{thm:weak_sol_time_delayed}}

\BBB

In this subsection, we separately discuss the limiting \ZZZ passage \EEE $\tau \to 0$ for the mechanical and the heat-transfer equation, leading to \eqref{weak_sol_time_del_y} and \eqref{weak_sol_time_del_theta}, respectively. 
 
\begin{lemma}[Convergence of the mechanical equation]\label{lemma:convmechequ}
Let $(y_h,\theta_h)$ be as in Proposition \ref{prop:a_priori_bounds}. Then, for any test function  $z \in C^\infty(\BBB I \EEE \times \overline \Omega; \R^d)$ satisfying $z = 0$ on $\BBB I \EEE \times \partial \Omega$ we have that \eqref{weak_sol_time_del_y} holds.
\end{lemma}

\EEE

\begin{proof}
\BBB By \eqref{def_H1_conv} we have that   $\partial_t \hat{y}_{\tau}  \weakly y_h$  weakly in $L^2(I_h;  H^3(\Omega;\R^d))$ \EEE and by  \BBB \eqref{def_Linfty_W1infty_conv} \EEE  it holds that $\hat{y}_{\tau}(t), \overline{y}_{\tau}(t)$, $\underline{y}_{\tau}(t) \to y_h$ strongly in $W^{2, p}(\Om; \R^{d})$ for a.e.~$t \in I$.
  Moreover, $(\hat{y}_{\tau})_\tau$, $(\overline{y}_{\tau})_\tau$, and $(\underline{y}_{\tau})_\tau$ are bounded in $L^\infty(I; W^{2, p}(\Om; \R^{d}))$ independently of $\tau$ and $h$.
  Hence, we deduce \eqref{weak_sol_time_del_y}  \ZZZ by \EEE testing \eqref{mechanical_step_el} with $z$, \BBB summing over $k$, \EEE and passing to the limit as $\tau \to 0$  using the \BBB generalized  dominated \EEE convergence theorem. \BBB Here, we crucially use that $ \pl_{\dot F} \disspot$ is linear in the second entry, cf.\ \ref{D_quadratic}.  
\end{proof}

\BBB 
Before we concern ourselves with the limiting passage in the heat-transfer equation, we establish a mechanical energy balance for $(y_h,\theta_h)$. \EEE

\BBB

\begin{lemma}[Mechanical energy balance]\label{lemma: chain rull}
Let $(y_h,\theta_h)$  \BBB be  as in Proposition \ref{prop:a_priori_bounds}  with $y_h(0, \cdot) = y_{0, \eps}$   satisfying \eqref{weak_sol_time_del_y}.  Then, for any $t \in I$ we have the mechanical energy balance 
\begin{align}\label{eq: mechenebal}
  \mathcal{M}(y_h(t))  &   +  \frac{\rho}{2}  \mint_{t-h}^t \Vert \partial_t y_h(s) \Vert^2_{L^2(\Omega)}  \di s   + \int_0^t    2\diss_\eps   (y_h,  \partial_t y_h, \theta_h)   \di s   +  \frac{\rho}{2h}\int_0^t  \Vert \partial_t y_h(s) - \partial_t y_h(s-h) \Vert^2_{L^2(\Omega)}     \di s   \nonumber
          \\
          & =  \mechen(y_{0,\eps})   +   \frac{\rho}{2}   \Vert y_{0,\eps}' \Vert^2_{L^2(\Omega)} 
             -  \int_0^t \int_\Omega 
       \pl_{F} W^{\rm cpl} (\nabla y_{h}, \theta_{h}) 
         : \partial_{t}\nabla y_{h}   \di x \di s
          +  \int_0^t \int_\Omega f \cdot \partial_{t} y_{h}   \di x \di s .       
\end{align}
\EEE
 
\end{lemma}

\begin{proof}
By the regularity $y_h \in H^1(I_h;H^3(\Omega;\R^d))$, using the chain rule for $\Lambda$-convex functionals (see \cite[Proposition 3.6]{tve_orig} and Lemma \ref{lemma_ lambda convex}), we first observe that the mechanical energy defined in \eqref{mechanical} satisfies  that $t \mapsto \mechen(y_h(t))$ lies in $W^{1,1}(I)$ and  that
\begin{align}\label{the cain rule}
\frac{\di}{\di t }   \mathcal{M}   (y_{h}) = \int_{\Om}    {D H}( \Delta y_{h})  : \partial_{t} \Delta y_{h} \di x + \int_{\Om} \partial_{F} \elpot ( \nabla y_{h} )  \cdot\partial_{t} \nabla y_{h} \di x    \quad \text{for a.e.\ $t \in I$}.
\end{align}
We  test  the equation  \eqref{weak_sol_time_del_y} with $z \defas\ee   \partial_{t} y_{h}  \indic_{[0,t]}  $, and obtain by \eqref{diss_rate} 
\begin{align*}
          &   \int_0^t \int_\Omega 
       \big(
        \pl_F \elpot(\nabla y_h) + \pl_{F} W^{\rm cpl} (\nabla y_{h}, \theta_{h}) \big) 
         : \partial_{t}\nabla y_{h}   \di x \di s +  \int_0^t \int_\Omega   D      \hypot( \Delta  y_h)   \cdot \partial_{t} \Delta  y_{h}   \di x \di s   
        \\
      &  
          -  \int_0^t \int_\Omega f \cdot \partial_{t} y_{h}   \di x \di s   +\frac{\rho}{h} \int_0^t \int_\Omega \big( \partial_t y_h(s) - \partial_t y_h(s - h) \big) \cdot \partial_{t} y_{h} (s)   \di x  \di s \nonumber
          \\
          &\qquad
           =  -\int_0^t \int_\Omega 2 \disspot(\nabla y_h, \partial_t \nabla y_h, \theta_h)  \di x \di s -  \eps \int_{0}^{t} \int_{\Om}   |\partial_t \nabla \ZZZ \Delta \EEE y_{h}|^2    \di x \di s  .     \nonumber
\end{align*}
Applying the chain rule \eqref{the cain rule} we find 
\begin{align}\label{MMMMM}
 & \mathcal{M} (y_h(t)) - \mechen(y_{0,\eps})       +  \int_0^t \int_\Omega 
       \pl_{F} W^{\rm cpl} (\nabla y_{h}, \theta_{h}) 
         : \partial_{t}\nabla y_{h}   \di x \di s
          -  \int_0^t \int_\Omega f \cdot \partial_{t} y_{h}   \di x \di s   
          \\
          &\quad
           =  -\int_0^t \int_\Omega \Big( 2 \disspot(\nabla y_h, \partial_t \nabla y_h, \theta_h)     +  \eps    |\partial_t \nabla \Delta y_{h}|^2  \Big)   \di x \di s   - \frac{\rho}{h} \int_0^t \int_\Omega \big( \partial_t y_h(s) - \partial_t y_h(s - h) \big) \cdot \partial_{t} y_{h} (s)  \di x    \di s.   \notag  
  \end{align}
\BBB Denoting the last term on the right-hand side  by $\Pi$ (without negative sign), and expanding it as in \eqref{e:last-step-3}, we derive 
\begin{align*}
\Pi & =     \frac{\rho}{2h} \int_0^t \big(  \Vert \partial_t y_h(s) \Vert^2_{L^2(\Omega)} - \Vert \partial_t y_h(s-h) \Vert^2_{L^2(\Omega)} + \Vert \partial_t y_h(s) - \partial_t y_h(s-h) \Vert^2_{L^2(\Omega)}  \big)    \di s \notag \\
&  =   \frac{\rho}{2}  \mint_{t-h}^t \Vert \partial_t y_h(s) \Vert^2_{L^2(\Omega)}  \di s  -  \frac{\rho}{2}  \mint_{-h}^0 \Vert \partial_t y_h(s) \Vert^2_{L^2(\Omega)}  \di s  + \frac{\rho}{2h}\int_0^t  \Vert \partial_t y_h(s) - \partial_t y_h(s-h) \Vert^2_{L^2(\Omega)}     \di s.  
\end{align*}
Plugging this into \eqref{MMMMM} and using    \eqref{Repps} as well as $ \partial_t y_h(s)  = y_{0,\eps}'$ for $s \in (-h,0)$, the proof of the mechanical energy balance is concluded. 
\end{proof}

\BBB 
\begin{lemma}[Convergence of the heat-transfer equation]\label{lemma: heattraf1}
Let $(y_h,\theta_h)$ be as in Proposition \ref{prop:a_priori_bounds}. Then, for any test function  $\varphi \in C^\infty(I \times \overline \Omega)$     equation \eqref{weak_sol_time_del_theta} is satisfied.
\end{lemma}

\BBB
 
\begin{proof}
The essential point is to show strong convergence of the strain rates, namely
\begin{equation}\label{eq: strngi conv}
 \nabla  \partial_t  \hat y_\tau  \to   \nabla   \partial_t y_h \text{ and }   \nabla \Delta  \partial_t \ZZZ \hat{y}_\tau \EEE  \to   \nabla  \Delta \partial_t y_h  \quad\text{strongly in } L^2(I;   L^2(\Omega; \R^{d\times d})  ). 
 \end{equation}
Once this is achieved, using the convergences in  \eqref{def_H1_conv}--\eqref{temp_Ls_conv}, we can almost verbatim follow the proof of  \cite[Proposition 4.6]{tve}, recalling that the  scheme for the heat-transfer equation differs from the one in  \cite{tve} only by the regularizing term $\eps |\partial_t \nabla \Delta \hat y_\tau|^{2} \wedge \tau^{-1}$, see \eqref{thermal_step_el} and \cite[Equation (3.11)]{tve}. The only difference is that for the term
$$ \int_\Omega  \partial_t \hat{w}_\tau \vphi \di x $$
in \eqref{thermal_step_el} we do not perform an integration by parts in time, but directly pass to the limit using that $\partial_t\hat{w}_\tau \weakly \partial_t w_h$ in $L^2( \ZZZ I; \EEE (H^1(\Omega))^*)$ by Corollary \ref{prop:a_priori_bounds2}. This gives   \eqref{weak_sol_time_del_theta}.

The argument for showing \eqref{eq: strngi conv} is along the lines of \cite[Lemma 4.5]{tve} or \cite[Proposition 5.1]{tve_orig}, and relies on passing to the limit in the mechanical energy balance. We briefly sketch the argument. In view of the notation in \eqref{interpolations}, we can write the discrete mechanical energy estimate in Lemma \ref{lemma: mechenbound}  as
\begin{align*} 
 & \mechen(\overline{y}_\tau(T))   + \frac{\rho}{2} \mint_{T-h}^T \Vert \partial_t \hat{y}_\tau(t) \Vert_{L^2(\Omega)}^2 \, \di t +  \int_0^T 2 \diss_\eps( \underline{y}_\tau, \partial_t  \hat{y}_\tau, \underline{\theta}_\tau)   \di t  +    \frac{\rho }{2h} \int_0^T  \Vert  \partial_t \hat{y}_\tau(t) -  \partial_t \hat{y}_\tau(t-h)\Vert_{L^2(\Omega)}^2 \, \di t
     \\ &\leq
      \mechen(y_{0,\eps})     +   \frac{\rho}{2} \Vert \BBB y_{0,\eps}' \EEE  \Vert_{L^2(\Omega)}^2 -     \int_0^T  \int_\Omega
      \pl_F \cplpot(\underline{y}_\tau,  \underline{\theta}_\tau) :  \partial_t \nabla  \hat{y}_\tau \di x \, \di t
      + \int_0^T \big( \overline{f}_\tau(t), \partial_t \hat{y}_\tau(t)  \big)_2 \, \di t    + C_M \tau V_{T/\tau}.
  \end{align*} 
  Employing \eqref{def_H1_conv}--\eqref{temp_Ls_conv}, standard lower semicontinuity arguments (see \cite[Equation (4.15)]{tve} and also \cite[Theorem~7.5]{FonsecaLeoni07Modern} for a general result) imply,  
  \begin{subequations}
  \begin{align}
 \liminf_{\tau \to 0} \mechen(\overline{y}_\tau(T)) &  \geq \mechen(y_h(T)), \label{two ini1} \\
 \liminf_{\tau \to 0} \int_0^T
      \int_\Omega\drate(\nabla \underline{y}_\tau, \partial_t \nabla \hat{y}_\tau, \underline{\theta}_\tau) \di x  \di t
     & \geq \int_0^T \int_\Omega\drate(\nabla y_h, \partial_t  \nabla  y_h, \theta_h) \di x \di t, \label{two ini2} \\
 \liminf_{\tau \to 0} \int_0^T
      \int_\Omega    \eps |\partial_t \nabla \Delta  \hat{y}_{\tau}|^2 \di x  \di t
     & \geq \int_0^T \int_\Omega \eps |\partial_t \nabla \Delta  {y}_h|^2   \di x \di t, \label{two ini3}\\          
      \liminf_{\tau \to 0}   \frac{\rho}{2} \mint_{T-h}^T\Vert \partial_t \hat{y}_\tau( \ZZZ t) \EEE \Vert_{L^2(\Omega)}^2 \, \di t 
     & \geq  \frac{\rho}{2}  \mint_{T-h}^T \Vert \partial_t y_h( \ZZZ t) \EEE \Vert^2_{L^2(\Omega)}  \di t, \label{two ini4}\\
           \liminf_{\tau \to 0}  \frac{\rho }{2h} \int_0^T  \Vert  \partial_t \hat{y}_\tau(t) -  \partial_t \hat{y}_\tau(t-h)\Vert_{L^2(\Omega)}^2 \, \di t 
     & \geq        \frac{\rho}{2h}\int_0^T  \Vert \partial_t y_h(t) - \partial_t y_h(t-h) \Vert^2_{L^2(\Omega)}     \di t. \label{two ini5}
  \end{align}
  \end{subequations} 
Since $\int_0^T ( \overline{f}_\tau(t), \partial_t \hat{y}_\tau(t) )_2 \, \di t \to \int_0^T \int_\Omega f \cdot \partial_{t} y_{h}   \di x \di t $, as well as $\int_0^T  \int_\Omega
      \pl_F \cplpot(\underline{y}_\tau,  \underline{\theta}_\tau) :  \partial_t \nabla  \hat{y}_\tau \di x \, \di t$ converges to $\int_0^T \int_\Omega 
       \pl_{F} W^{\rm cpl} (\nabla y_{h}, \theta_{h}) 
         : \partial_{t}\nabla y_{h}   \di x \di t$   by \eqref{def_H1_conv}--\eqref{temp_Ls_conv}, and  $\tau V_{T/\tau} \to 0$  as $\tau \to 0$ by \eqref{H1_def}, using the mechanical energy balance \eqref{eq: mechenebal}, we find that all estimates  \eqref{two ini1}--\eqref{two ini5} are actually equalities, see \cite[Equation (4.15)]{tve} for details on this argument. In particular, the convergence in \eqref{two ini2} implies  the first part of \eqref{eq: strngi conv},    by repeating the arguments in  \cite[Equation (4.16)ff.]{tve}. Eventually,   the convergence in \eqref{two ini3} provides the second part of \eqref{eq: strngi conv}. 
\end{proof}

\begin{proof}[Proof of Theorem \ref{thm:weak_sol_time_delayed}]
The statement follows by collecting the regularity for $y_h$, $\theta_h$, and $w_h$ given in Proposition~\ref{prop:a_priori_bounds} and Corollary~\ref{prop:a_priori_bounds2}, and the identification of the limiting equations in Lemmas \ref{lemma:convmechequ} and \ref{lemma: heattraf1}.
\end{proof} 
  \EEE

\section{Vanishing time-delay}\label{sec:vanish_h}

We recall that for each $T, \, h , \, \eps> 0$, Theorem \ref{thm:weak_sol_time_delayed} guarantees the existence of $(y_h, \theta_h)$ such that   $\BBB y_h  \in L^\infty( \BBB I; \EEE \Yidreg) \cap H^1(I_h; H^3(\Omega; \R^d))$,   $\theta_h \in L^2(I; H^1_+(\Omega))$, and \BBB such that equations \eqref{weak_sol_time_del_y}--\eqref{weak_sol_time_del_theta} hold. We also recall that $y_h(t) = \BBB y_{0,\eps} + t y_{0,\eps}' \EEE $ for $t \in [-h, 0]$. \EEE
The goal of this section is to pass to the limit $h \to 0$ in \eqref{weak_sol_time_del_y}--\eqref{weak_sol_time_del_theta}.
\BBB We start by a compactness result. \EEE

\begin{lemma}[Compactness]\label{h_lem:compactness}
   \BBB There exist $  y_\eps    \in L^{\infty}(I;  \Yidreg)   \cap  H^1(I; H^3(\Omega; \R^d))$ with $y_\eps(0) = \ZZZ y_{0,\eps} \EEE $, and  $ \theta_\eps   \in L^2(I; H^1_+(\Omega))  $, as well as    $w_\eps  \defas \inten(\nabla y_\eps, \theta_\eps) \in L^2(I; H^1(\Omega)) \cap H^1(I; (H^1(\Omega))^*)$ with  \ZZZ $ w_\eps(0) =   w_{0,\eps}  $  \EEE such that, up to selecting a subsequence (not relabeled), it holds that
  \begin{subequations}
  \begin{align}
    y_h &\weakly y_\eps \quad\text{weakly in } \BBB H^1(I; H^3(\Omega; \R^d)), \EEE  \label{h_def_H1_conv} \\
    y_h & \to y_\eps \quad\text{strongly in } \BBB C \EEE (I; W^{1, \infty}(\Omega; \R^{d})) \BBB \ \ \text{ and strongly in } C(I; W^{2, p}(\Omega; \R^{d})), \label{h_def_L2_W2p_conv} \\
    \theta_h &\to \theta_\eps \quad \text{ and } \quad
    w_h \to w_\eps
    \text{ strongly in } \ZZZ  L^s(I \times \Omega) \text{ for any } s \in [1, \tfrac{d+2}{d}), \EEE \label{h_temp_Lr_W1r_conv} \\
  \BBB      \theta_h &\weakly \theta_\eps \quad \text{ and } \quad
    w_h \weakly w_\eps
    \text{ weakly in }   L^2(I; H^1(\Omega))  , \EEE \label{h_temp_Lr_W1r_conv2} \\
        w_h& \weakly w_\eps
    \text{ weakly in } H^1(I; (H^1(\Omega))^*) . \label{h_temp_Ls_conv}
  \end{align}
  \end{subequations}
\end{lemma}

\BBB

\begin{proof}
As the a priori bounds derived in Proposition \ref{prop:a_priori_bounds} and Corollary \ref{prop:a_priori_bounds2} are independent of $h$, we see by the lower semicontinuity of norms that the same bounds hold true for $(y_h, \theta_h)$ \ZZZ in \EEE place of $(y_\tau, \theta_\tau)$. Then, \ZZZ \eqref{h_def_H1_conv}--\eqref{h_temp_Lr_W1r_conv} \EEE can be obtained exactly as in the proof of Proposition \ref{prop:a_priori_bounds}, and \eqref{h_temp_Lr_W1r_conv2} --\eqref{h_temp_Ls_conv} follow from the bounds in Corollary \ref{prop:a_priori_bounds2}  by weak compactness and the  Aubin-Lions' lemma, see also \cite[Proposition 5.1]{tve_orig}.
\end{proof}

\BBB
We collect a priori bounds for the limits $(y_{\eps}, \theta_{\eps})$ \EEE which directly follow from~Proposition \ref{prop:a_priori_bounds}, Lemma~\ref{l:laplace-regularity}, Lemma \ref{h_lem:compactness},   and lower semicontinuity: for each $q \in [1, \frac{d + 2}{d})$ and $r \in [1, \frac{d+2}{d+1})$ we can find constants $C$, $C_q$, and $C_r$ independent of~$\eps$  such that
\begin{subequations}
\label{e:epsilon_a_priori_1}
\begin{align}
  \Vert y_{\eps} \Vert_{L^\infty(I; W^{2, p}(\Omega))}
  + \Vert \det(\nabla y_{\eps})^{-1} \Vert_{L^\infty(I \ee \times \Omega)}
    &\leq C \label{mu_Linfty_W2p_def}, 
    \\
  \Vert y_{\eps} \Vert_{H^1(I \ee \times \Omega)} + \BBB  \sqrt{\eps}  \| \partial_{t}  {y}_{\eps}\|_{L^{2}(I;H^3(\Om))} \EEE 
    &\leq C \label{mu_H1_def}, \\
  \Vert \theta_{\eps} \Vert_{L^\infty(I; L^1(\Omega))}
  + \Vert w_{\eps} \Vert_{L^\infty(I; L^1(\Omega))} &\leq C \label{mu_Linfty_L1_temp}\\
    \Vert \theta_{\eps} \Vert_{L^q(I \ee \times \Omega)}
  + \Vert  w_{\eps} \Vert_{L^q(I \ee \times \Omega)} \leq C_q \label{mu_Lq_temp}, \\
  \Vert \nabla \theta_{\eps} \Vert_{L^r(I \ee \times \Omega)}
  + \Vert \nabla w_{\eps} \Vert_{L^r(I \ee \times \Omega)} \leq C_r. \label{mu_W1r_temp}
\end{align}
\end{subequations}
Furthermore,  \BBB there exist $\varrho \in (\frac{1}{2},1)$ and $C>0$ independently of $\eps$ \EEE with 
\begin{subequations}
\label{e:lar-reg-eps-finale}
\begin{align}
    \label{e:lap-reg-mu-final}
  \|  \Delta y_{\eps}   \|_{L^{2} (I; H^{1}(\Om))}  + \|  |\Delta y_{\eps}|^{\frac{p-2}{2}} \nabla(\Delta y_{\eps})  \|_{L^{2} (I \times \Om)} & \leq C  \big( 1 + \sqrt{\eps } \| y_{0,\varepsilon}\|_{H^{4}(\Om)} + \| y'_{0, \varepsilon}\|_{H^{1}( \Om)} \big) ,  
 \\
\label{e:lap-reg-mu1.5-final}
\|  \nabla {D H}(\Delta y_\eps)\|_{L^2(I;L^{p'}(\Omega))} & 
 \leq C \big( 1 + \sa \sqrt{\varepsilon} \| y_{0, \eps}\|_{H^{4}(\Om)}  + \| y'_{0, \eps}\|_{H^{1}(\Om)} \ee \big) \,, 
 \\
\label{e:lap-reg-mu2-final}
 \BBB \|   \Delta \partial_{t} y_{\eps}  \|_{L^{2}(I; H^2(\Om))} \EEE  &
  \leq C \eps^{-\frac{1+\varrho}{2}} \big( 1 + \sa \sqrt[4]{\varepsilon} \| y_{0, \eps}\|_{H^{4}(\Om)}  + \ZZZ \| y'_{0, \eps}\|_{H^{1}(\Om)} \EEE \ee \big).
  \end{align}
\end{subequations}  
\MMM Moreover, we have \EEE
\begin{subequations}
\label{e:bdry_eps}
\begin{align}
\label{e:bdry_eps_1}
\Delta y_{\eps}(t) &= \partial_{\nu} \Delta y_{\eps}(t) = 0 \qquad \text{on $\partial \Omega$ for a.e.~$t \in I$,}
\\
\label{e:bdry_eps_2}
\Delta \partial_{t} y_{\eps} (t) &= \partial_{\nu} \Delta \partial_{t} y_{\eps}(t) = 0 \qquad \text{on $\partial\Om$ for a.e.~$t \in I$.}
\end{align}
\end{subequations}
In particular,~\eqref{e:bdry_eps_2} follows from the boundary conditions deduced in Proposition~\ref{cor:higher_reg_y}, the estimates obtained in Lemma~\ref{l:laplace-regularity} on~$(\partial_{t}\ee y_{h})_h$, and a compactness argument, while~\eqref{e:bdry_eps_1} is a consequence of~\eqref{e:bdry_eps_2} \BBB and the fact that $y_{0,\eps} \in \Yidreg$, see \eqref{eq: yidre}. \EEE

The rest of the  section is devoted to the proof of  \BBB Theorems \ref{thm:main-thermal-elasto-regu} and \ref{thm:main-thermal-elasto-regu2}. In particular, we show that $(y_\eps,\theta_\eps)$ is a solution to the regularized system \eqref{mechanical_equation}--\eqref{e:regularized-thermal}. As the second bound in \eqref{H1_def} transfers uniformly to $\partial_t y_h$, there exists $y_{T,\eps}' \in L^2(\Omega;\R^d)$ such that, \ZZZ as $h \to 0$, \EEE
\begin{align}\label{yTstrich}
\mint_{T-h}^T\partial_t y_h(s) \di s \weakly  y_{T,\eps}' \ \  \text{weakly in $L^2(\Omega;\R^d)$}. 
\end{align}   \EEE

\EEE


\begin{proposition}[Auxiliary mechanical equation]\label{prop:h_convergence_mechanical}
  Let \BBB $(y_\eps, \theta_\eps)$ \EEE be as in Lemma~\ref{h_lem:compactness}.
  Then, $(y_\eps, \theta_\eps)$ satisfies
\begin{equation} \label{mechanical_equationNNN}
  \begin{aligned}
    0= & \BBB \int_I \EEE \int_{\Om} \partial_{F} W(\nabla y_\eps, \theta_\eps) : \nabla z \di x \di t
    +   \BBB \int_I \EEE \int_{\Om} {D H} ( \Delta  y_\eps) \cdot \Delta z \di x \di t  + \BBB  \eps \int_I   \int_{\Om} \partial_t \nabla \Delta  y_\eps : \nabla \Delta z \di x \di t \EEE   \\
    &\quad - \BBB \int_I \EEE  \int_{\Om} f \cdot z \di x \di t
    + \BBB \int_I \EEE \int_{\Om} \partial_{\dot{F}} R(\nabla y_\eps, \partial_t \nabla y_\eps, \theta_\eps) : \nabla z \di x \di t  - \rho \BBB \int_I \EEE \int_{\Om} \partial_t y_\eps \cdot \partial_t z \di x \di t \\
    &\qquad + \rho \int_{\Om}  \BBB \big(y'_{T,\eps}  \cdot z(T)  \EEE   -  y'_{0,\eps}  \cdot z(0)\big) \di x,  
  \end{aligned}
  \end{equation}  
for every $z \in C^{\infty} ( \ZZZ I \EEE \times \overline{\Omega}; \R^{d})$ with $z = 0$ on $\ZZZ I \EEE\times \partial \Om$.
\end{proposition}

\begin{proof}
Notice that by a change of variables we can rewrite~\eqref{weak_sol_time_del_y} for every $z \in  C^\infty(I \times \overline \Omega; \R^{d})\ee$ with $z = 0$ on $I \times  \partial \Omega\ee$ as
\begin{equation*}
\begin{aligned}
  &\intQ  D \hypot(\Delta y_h) \cdot \Delta z
    + \Big(
      \pl_F \felpot(\nabla y_h, \theta_h)
      + \pl_{\dot F} \disspot(\nabla y_h, \partial_t \nabla y_h, \ZZZ \theta_h \EEE )
    \Big) : \nabla z \di x \di t  \\
  &\quad   + \eps \int_I \int_{\Om}  \partial_t  \nabla\Delta   y_{h} \cdot  \nabla \Delta \ee z \di x \di t \EEE
  - \rho \int_0^{T - h} \int_\Omega \partial_t y_h(t) \cdot \frac{z(t + h) - z(t)}{h} \di x \di t
     \\
  &\quad - \rho \int_\Omega  \BBB y_{0,\eps}' \EEE  \cdot \mint_0^h   z(t)\ee \di t \di x + \rho \mint_{T-h}^{T} \int_\Omega \partial_{t} y_{h} (t) \cdot z(t) \di x \di t = \intQ f \cdot z \di x \di t.
\end{aligned}
\end{equation*}
\BBB By the smoothness of $z$ and \eqref{yTstrich} it directly follows that
  \begin{equation*}
  \int_\Omega  y_{0,\eps}' \cdot \mint_0^h z(t) \di t \di x \to \int_\Omega  y_{0,\eps}' \cdot  z(0) \di x  \qquad    \mint_{T-h}^{T} \int_\Omega \partial_{t} y_{h} (t) \cdot z(t) \di x \di t \to \int_\Omega  y_{T,\eps}' \cdot  z(T) \di x    \ \  \text{ as } h \to 0.
  \end{equation*} \EEE
  Moreover, the convergence in \eqref{h_def_H1_conv} and the smoothness of $z$ show by weak-strong convergence
  \begin{equation*}
    \int_0^{T-h\ee} \int_\Omega \partial_t y_h(t) \cdot \frac{z(t + h) - z(t)}{h} \di x \di t \to \int_0^T \int_\Omega \partial_t y_\eps \cdot \partial_t z \di x \di t \qquad \text{as } h \to 0.
  \end{equation*}
  By \BBB \eqref{h_def_L2_W2p_conv}, \EEE  assumption \ref{H_bounds}, \ee and   \BBB the generalized  dominated  convergence theorem, \EEE we derive that
  \begin{equation*}
    \intQ   D H \ee(\Delta  y_h) \cdot \Delta z \di x \di t \to \intQ   D H \ee(\Delta y_\eps) \cdot \Delta z \di x \di t \qquad \text{as } h \to 0.
  \end{equation*}
  Finally, by \BBB  \eqref{h_def_H1_conv}--\eqref{h_temp_Lr_W1r_conv}  \EEE the convergence of all remaining terms follows, leading to the desired equation \eqref{mechanical_equationNNN}.
\end{proof}

\BBB
\begin{proposition}[Mechanical equation]\label{prop:h_convergence_mechanical2}
  Let  $(y_\eps, \theta_\eps)$   be as in Lemma~\ref{h_lem:compactness}.
  Then, $(y_\eps, \theta_\eps)$ satisfies \eqref{mechanical_equation}. Moreover, it holds that $\partial_{tt}^2 y_\eps \in L^2(I; (H^3(\Omega;\R^d) \cap H^1_0(\Omega;\R^d))^*)$ with $\Vert \partial_{tt}^2 y_\eps \Vert_{L^2(I; (H^3(\Omega) \cap H^1_0(\Omega))^*)} \le C $ for a constant $C>0$ independent of $\eps$. In particular, $\partial_t y_\eps \in C(I;L^2(\Omega;\R^d))$ with $\partial_t y_\eps(0) = y'_{0,\eps}$ and $\partial_t y_\eps(T) = y'_{T,\eps}$. 
\end{proposition}

\begin{proof}
In view of  \ZZZ \eqref{mu_Linfty_W2p_def}--\eqref{mu_W1r_temp}, all time integrals \EEE  in \eqref{mechanical_equationNNN} except for $\rho    \int_I  \int_{\Om} \partial_t y_\eps \cdot \partial_t z \di x \di t$ lie in $ L^2(I;(H^3(\Omega) \cap H^1_0(\Omega))^*)$, where for the nonlinear term $  \int_I  \int_{\Om} {D H} ( \Delta  y) \cdot \Delta z \di x \di t $ we particularly use \ref{H_bounds}, \eqref{mu_Linfty_W2p_def},  and the fact  that $\Delta z \in L^2(I; L^p(\Omega))$ for each $z \in L^2(I;H^3(\Omega) \cap H^1_0(\Omega))$ ($p \in (3,6)$ for $d=3$). More precisely, the corresponding operator norms are uniformly bounded independently of $\eps$.  Then, by definition of weak derivatives we get that $\partial_{tt}^2 y_\eps \in L^2(I; (H^3(\Omega;\R^d) \cap H^1_0(\Omega;\R^d))^*)$ exists and  is bounded independently of $\eps$. Moreover,  an integration by parts in time shows  \eqref{mechanical_equation} for $z \in C^{\infty} ( \ZZZ I \EEE  \times \overline{\Omega}; \R^{d})$ with $z = 0$ on $ \ZZZ I \EEE \times \partial \Om$ and $z(0)=z(T) = 0$. Then, by a density argument we observe that the assumption $z(0)=z(T)=0$ can be dropped.   Eventually, $\partial_t y_\eps \in C(I;L^2(\Omega;\R^d))$ follows from \cite[Lemma 7.3]{Roubicek-book},  and $\partial_t y_\eps(0) = y'_{0,\eps}$ as well as $\partial_t y_\eps(T) = y'_{T,\eps}$ follow by integration by parts in time of  \eqref{mechanical_equation} for general  $z \in C^{\infty} ( \ZZZ I \EEE \times \overline{\Omega}; \R^{d})$  and a comparison with \eqref{mechanical_equationNNN}.  
\end{proof}

\EEE 

\BBB We proceed with a mechanical energy balance which is the analog to the one   in \eqref{eq: mechenebal}. We note that the balance can be formulated for each $t \in I$ since $y_\eps \in C(I; W^{2, p}(\Omega; \R^{d}))$ and  $\partial_t y_\eps \in C(I;L^2(\Omega;\R^d))$.

\BBB 
\begin{lemma}[Mechanical energy balance]\label{lemma: chain rull-2}
  Let  $(y_\eps, \theta_\eps)$ \EEE be as in Lemma~\ref{h_lem:compactness}  satisfying~\eqref{mechanical_equation}.  Then, for any $t \in I$ we have the mechanical energy \ste balance~\eqref{e:limit-energy-2-NNN-for result }.
\end{lemma}
\BBB

\begin{proof}
Since $\partial^{2}_{tt} \ZZZ y_\eps \EEE$ lies in $L^{2} (I; ( H^{3} (\Om; \R^{d}) \cap H^{1}_{0} (\Om; \R^{d}))^{*})$, see Proposition \ref{prop:h_convergence_mechanical2},  by an approximation argument  we can test~\eqref{mechanical_equation} with $z =    \partial_{t} y_\eps  \indic_{[0,t]} $ and  obtain
\begin{align}
\label{e:limit-energy}
  & \int_0^t \int_\Omega
       \Big(
        \pl_F \elpot(\nabla y_\eps) + \pl_{F} W^{\rm cpl} (\nabla y_\eps , \theta_\eps ) \Big) 
         : \partial_{t} \nabla y_\eps   \di x \di t + \int_0^t \int_\Omega \BBB D \EEE 
      \hypot( \Delta  y_\eps)  : \partial_{t}\Delta y_\eps  \di x \di s         
       \nonumber       \\
         &\quad\quad
             =  -   \int_0^t 
       2 \diss_\eps( y_\eps , \partial_t  y_\eps , \theta_\eps )    \di s    - \rho \int_0^t \langle  \partial^2_{tt} y_\eps, \partial_t y_\eps \rangle \di s  +   \int_0^t \int_\Omega f \cdot \partial_{t} y_\eps   \di x \di s .
\end{align}
Using  the chain rule we find
\begin{align*}
\rho \int_0^t \langle  \partial^2_{tt} y_\eps, \partial_t y_\eps \rangle \di s & =  \frac{\rho}{2} \int_0^t \int_{\Om} \frac{\di}{\di t } |\partial_t y_\eps|^{2} \,    \di x \di s    =  \frac{\rho}{2}   \Vert \partial_t y_h(t) \Vert^2_{L^2(\Omega)}   - \frac{\rho}{2}   \Vert y_{0,\eps}' \Vert^2_{L^2(\Omega)} .
\end{align*}
Combining this with the chain rule in \eqref{the cain rule} (for $y_\eps$ in place of $y_h$) and plugging into \eqref{e:limit-energy}, the proof is concluded, \ZZZ using again \eqref{diss_rate}--\eqref{Repps}. \EEE
\end{proof}

 \EEE

\BBB
\begin{proposition}[Heat-transfer equation]\label{prop:h_convergence_heat}
  Let  $(y_\eps, \theta_\eps)$  be as in Lemma~\ref{h_lem:compactness}.
  Then, $(y_\eps, \theta_\eps)$ satisfies~\eqref{e:regularized-thermal}.
\end{proposition}
\EEE

\begin{proof}
\BBB

As in the proof of Lemma \ref{lemma: heattraf1}, the essential point is to show strong convergence of the strain rates, namely
\begin{equation}\label{eq: strngi conv-new}
\nabla   \partial_t   y_h  \to   \nabla   \partial_t  y_\eps \text{ and }   \nabla \Delta  \partial_t  y_h  \to    \nabla  \Delta  \partial_t  y_\eps  \quad\text{strongly in } L^2(I;   L^2(\Omega; \ZZZ  \R^{d\times d}) \EEE  ),
 \end{equation}
 as then one can pass to the limit in each term by using the convergences in  \eqref{h_def_H1_conv}--\eqref{h_temp_Ls_conv}. Rearranging the terms in \eqref{eq: mechenebal}, dropping one nonnegative term, and passing to the liminf as $h \to 0$, by the convergences in Lemma~\ref{h_lem:compactness} we get  
\begin{align}\label{a new lsc0} 
\liminf_{h \to 0}  & \Big(    \mathcal{M}   (y_{h} (T) )   + \frac{\rho}{2}  \mint_{T-h}^T \Vert \partial_t y_h(s) \Vert^2_{L^2(\Omega)}  \di s    +  \int_I 2\diss_\eps( y_h, \partial_t  y_h, \theta_h)  \di t    \Big)               \notag \\ 
&  \le  \mathcal{M}   (y_{0, \varepsilon})    +   \frac{\rho }{2}   \Vert y_{0,\eps}' \Vert^2_{L^2(\Omega)}      + \lim_{h \to 0}  \int_I \int_\Omega 
        f \cdot \partial_{t} y_{h}     \di x \di t  -  \lim_{h \to 0} \int_I \int_\Omega \pl_{F} W^{\rm cpl} (\nabla y_{h}, \theta_{h}) 
         : \partial_{t}\nabla y_{h}    \di x \di t \notag \\
         &  \ZZZ = \EEE  \mathcal{M}   (y_{0, \varepsilon})    +   \frac{\rho }{2}   \Vert y_{0,\eps}' \Vert^2_{L^2(\Omega)}   +  \int_I \int_\Omega 
         f \cdot \partial_{t} y_\eps  \di x \di t -     \int_I \int_\Omega  \pl_{F} W^{\rm cpl} (\nabla y_\eps, \theta_\eps) 
         : \partial_{t}\nabla y_\eps       \di x \di t.
\end{align}
By the convergences in Lemma~\ref{h_lem:compactness} and standard lower semicontinuity arguments  we get 
\begin{align}\label{a new lsc} 
 \liminf_{h \to 0}   \BBB  \mathcal{M} \EEE (y_{h} (T) )       &  \geq     \mathcal{M}  (y_\eps (T) )    , \notag \\
 \liminf_{h \to 0} \int_I 2\diss_\eps( y_h, \partial_t  y_h, \theta_h)    \di t      & \geq \int_I 2\diss_\eps( y_\eps, \partial_t  y_\eps, \theta_\eps)   \di t, \notag\\
 \liminf_{h \to 0}   \frac{\rho}{2}   \mint_{T-h}^T \Vert \partial_t y_h(s) \Vert^2_{L^2(\Omega)}  \di s    & \ge  \frac{\rho}{2}   \Vert \partial_t y_\eps(T) \Vert^2_{L^2(\Omega)}.   
  \end{align}
 For the first two estimates we also refer to  \cite[Equation (4.15)]{tve} and the last one follows from \eqref{yTstrich} and Proposition~\ref{prop:h_convergence_mechanical2}.
Combining \eqref{a new lsc0}--\eqref{a new lsc}  with the mechanical energy balance in Lemma \ref{lemma: chain rull-2}, we conclude that all inequalities in \eqref{a new lsc}  are actually equalities.  Then, \eqref{eq: strngi conv-new} follows exactly as in the final argument of the proof of Lemma~\ref{lemma: heattraf1}.  
\end{proof}

 \BBB

\begin{proof}[Proofs of Theorems \ref{thm:main-thermal-elasto-regu} and \ref{thm:main-thermal-elasto-regu2}]
The weak formulation, the regularity properties of $(y_\eps,\theta_\eps,w_\eps)$, and the initial conditions for $y_\eps,\partial_t y_\eps, w_\eps$ follow from Lemma \ref{h_lem:compactness}, Proposition \ref{prop:h_convergence_mechanical2}, and Proposition \ref{prop:h_convergence_heat}. The mechanical energy balance is given in Lemma \ref{lemma: chain rull-2}. 

Concerning \eqref{internal energy balance-for result }, we observe that by density we can test \eqref{e:regularized-thermal} with functions $\varphi \in W^{1,1}(I)$ which are independent of the space variable $x$ and satisfy $\varphi(T) = 0$. For $t \in (0,T)$ fixed, we define the test function $\varphi$ with $\varphi \equiv 1$ on $(0,t-\delta)$, $\varphi \equiv 0$ on $(t+\delta,T)$, and $\varphi' \equiv -\frac{1}{2\delta}$ on $(t-\delta, t + \delta)$. Then, by an integration by parts in time for the term $\int_I \langle \partial_t w_\eps, \vphi  \rangle \di t$, in the limit $\delta \to 0$,   \eqref{e:regularized-thermal} yields 
\begin{align*}
    0 = & \int_0^t -\Big( \drate(\nabla y_\eps, \partial_t \nabla y_\eps, \theta_\eps
        ) 
        + \pl_F \cplpot(\nabla y_\eps, \theta_\eps) : \pl_t \nabla y_\eps +  \eps  |\partial_t \nabla \Delta y_\eps|^2
      \Big)  \di x \di t \\
    &\quad       - \kappa \int_0^t \int_{\partial \Omega} (\bt - \theta_\eps)   \di \haus^{d-1} \di t
    + \lim_{\delta \to 0} \frac{1}{2\delta}\int_{t-\delta}^{t+\delta}  \int_\Omega w_\eps  \di x \di t   - \int_\Omega w_{0,\eps}  \di x. \nonumber
\end{align*} 
As $w_\eps \in C(I;L^2 (\Omega))$ by Lemma \ref{h_lem:compactness} and \cite[Lemma 7.3]{Roubicek-book}, we find $\lim_{\delta \to 0} \frac{1}{2\delta}\int_{t-\delta}^{t+\delta}  \int_\Omega w_\eps  \di x \di s =   \int_\Omega w_\eps(t,x)  \di x$, which concludes the proof of \eqref{internal energy balance-for result }. 
Eventually, \eqref{eq:energy-old} follows by summation of \eqref{e:limit-energy-2-NNN-for result } and \eqref{internal energy balance-for result }.
\end{proof}

\section{Vanishing regularization: Proof of Theorem \ref{thm:main-thermal-elasto-unregu}}\label{sec: epstozero}
\label{s:vanishing-regularisation}

This section is devoted to the analysis of the limit of the regularized thermo-elastodynamic system~\eqref{mechanical_equation}--\eqref{e:regularized-thermal} as $\eps \to 0$. This will show existence of solutions to the system~\eqref{mechanical_equation_final-def-unreg}--\eqref{e:new-thermal-equation-lim}, i.e.,  \BBB Theorem~\ref{thm:main-thermal-elasto-unregu}. First, \EEE given  initial data $y_{0} \in \Yid$ and $y'_{0} \in H^{1}_{0} (\Om; \R^{d})$, we consider  suitable regularizations $y_{0, \eps} \in \BBB \Yidreg\EEE$ \BBB (see \eqref{eq: yidre}) \EEE and $y'_{0, \varepsilon} \in H^{3} (\Om; \R^{d})$ such that  
\begin{align}
& y_{0, \eps} \to y_{0} \text{ in $W^{2, p}(\Om; \R^{d})$ \quad and \quad} y'_{0, \eps} \to y'_{0} \text{ in $H^{1}(\Om; \R^{d})$,} \label{e:hpmu0}
\\
& \limsup_{\eps \to 0} \BBB \sqrt[4]{\eps} \EEE \| y_{0, \eps}\|_{\ste H^{4} (\Om) \BBB} <+\infty\,. \label{e:hpmu}
\end{align}
\BBB This can be achieved by considering regularizations $(\varphi_\eps)_\eps \subset C^\infty_c(\Omega;\R^d)$ with $\varphi_\eps \to \Delta y_{0} \in L^p(\Omega;\R^d)$, and choosing $y_{0, \eps} \in C^\infty(\Omega;\R^d) \cap \mathcal{Y}_\id$ as the solution to $\Delta y_{0, \eps} = \varphi_\eps$. Then, $\partial_\nu \Delta  y_{0,\eps} = \Delta  y_{0,\eps}   = 0 $ on $\partial \Omega$ holds by construction and \eqref{e:hpmu0}--\eqref{e:hpmu} can be achieved by the elliptic regularity estimate 
$\Vert   y_{0, \eps} - y_0 \Vert_{W^{2,p}(\Omega)} \le C \Vert \Delta( y_{0, \eps} - y_0) \Vert_{L^p(\Omega)}  $, see \cite[Lemma~9.17]{gt}. 
\EEE

For every $\eps>0$,  in Theorem~\ref{thm:main-thermal-elasto-regu} we have shown the existence of a solution~$(y_{\eps}, \theta_{\eps})$ to the regularized thermo-elastodynamic system~\eqref{mechanical_equation}--\eqref{e:regularized-thermal}.  In the following lemma, we summarize the compactness properties of such \BBB sequences. \EEE

\begin{lemma}
\label{l:mu-compactness}
Let $(y_{\eps}, \theta_{\eps})$ be \BBB a \EEE sequence of solutions to \eqref{mechanical_equation}--\eqref{e:regularized-thermal} with initial data $(y_{0, \eps}, y'_{0, \eps}, \theta_{0})$  \BBB given by Theorem~\ref{thm:main-thermal-elasto-regu}. \EEE Then, there exists $(y, \theta) \in \BBB (L^\infty(I;\Yid) \cap \EEE H^{1}(\BBB I; \EEE H^{1} (\Om; \R^{d}))) \times L^{1} (\BBB I; \EEE W^{1, 1}_{+}(\Om))$ such that, up to a subsequence,
it holds \RB for any $q \in (1, \ste 2^{*} \RB)$ that \EEE
  \begin{subequations}
  \begin{align}
    y_{\eps} &\weaklystar y \text{ weakly* in } L^\infty(I; \BBB W^{2,p}(\Omega;\R^d))  \EEE \text{ and weakly in } H^{1}( \BBB I; \EEE  H^{1}(\Om; \R^{d})) ,\label{mu_def_Linfty_W2p_conv} \\
    y_{\eps} &\to y \text{ in } L^\infty(I; W^{1, \infty}(\Omega; \R^{d })) \RB \text{ and in } L^2(I; W^{2, p}(\Omega; \R^d))\EEE, \label{mu_def_Linfty_W1infty_conv} \\
\BBB \partial_{t} y_{\eps} & \to \partial_{t} y  \text{ in  $L^{2}(I; L^{q} (\Om; \R^{d}))$}, \EEE \label{mu_yt_L2_L2}\\
    \theta_{\eps} &\rightharpoonup \theta \quad \text{ and } \quad
     w_{\eps} \rightharpoonup w
    \text{ weakly in } L^r(I; W^{1,r}(\Omega)) \text{ for any } r \in [1, \tfrac{d+2}{d+1}), \label{mu_temp_Lr_W1r_conv} \\
    \theta_{\eps} &\to \theta \quad \text{ and } \quad
     w_{\eps} \to w
    \text{ in } L^s(I \times \Omega) \text{ for any } s \in [1, \tfrac{d+2}{d}). \label{mu_temp_Ls_conv}
  \end{align}
  \end{subequations}
    \BBB where    $w \defas \inten(\nabla y, \theta)$.  \EEE   
\end{lemma}

\begin{proof}
The \RB convergences in~\eqref{mu_def_Linfty_W2p_conv} and \eqref{mu_temp_Lr_W1r_conv}--\eqref{mu_temp_Ls_conv} as well as the first convergence in \EEE \eqref{mu_def_Linfty_W1infty_conv} \EEE are obtained arguing as in \BBB the proof of Proposition~\ref{prop:a_priori_bounds}, \EEE relying on the estimates \eqref{mu_Linfty_W2p_def}--\eqref{mu_W1r_temp}. As for~\eqref{mu_yt_L2_L2}, we apply \BBB the   Aubin-Lions' lemma as follows: \EEE by~\eqref{mu_H1_def} we have that $\partial_{t} y_{\eps}$ is bounded in $L^{2} (I; H^{1}(\Om; \R^{d}))$ \BBB and \RB Proposition~\ref{prop:h_convergence_mechanical2} yields that $\RB\partial^{2}_{tt}\EEE y_{\eps}$ is bounded in $L^{2} (I; ( H^{3} (\Om; \R^{d}) \cap H^{1}_{0} (\Om; \R^{d}))^{*})$. \EEE Hence,~\eqref{mu_yt_L2_L2} holds \BBB by the compact embedding $H^{1}(\Om; \R^{d}) \subset \subset L^q(\Omega;\R^d)$ for $q < 2^*$. \EEE  The identification $w  = W^{\rm in} (\nabla y, \theta)$ follows as in the proof of Proposition~\ref{prop:a_priori_bounds}, cf.\ also~\cite[Lemma~4.2]{tve}. \RB Note that by~\eqref{e:lap-reg-mu-final} and elliptic regularity we follow that $y_\eps$ is bounded in $L^2(I; H^3(\Omega; \R^d))$. Consequently, yet another application of the Aubin-Lions' lemma using also the boundedness of $\partial_t y_\eps$ in $L^2(I \times \Omega; \R^d)$ shows the second convergence in~\eqref{mu_def_Linfty_W1infty_conv}.\EEE
\end{proof}

\subsection{The mechanical equation} 
\BBB We recall that the mechanical  equation~\eqref{mechanical_equation} is equivalent to the formulation in \eqref{mechanical_equationNNN}.  \EEE

\begin{proposition}
\label{thm:mu_convergence_mechanical}
\ste  Let $(y, \theta)$ be as in Lemma~\ref{l:mu-compactness}. Then, $(y, \theta)$ satisfies~\eqref{mechanical_equation_final-def-unreg}. 
\end{proposition}

 \ste
 \begin{proof}
 We test~\eqref{mechanical_equationNNN} with $z \in C^{\infty} (I \times \overline{\Om}; \R^{d})$ with $z = 0 $ in $I \times \partial \Om$ and $z(T) = 0$. Thanks to the convergences in~\eqref{mu_def_Linfty_W2p_conv}--\eqref{mu_temp_Ls_conv} and to assumptions~\eqref{W_regularity}, \eqref{H_regularity} and \eqref{H_bounds}, \eqref{C_regularity}, \eqref{D_quadratic}, and \eqref{e:hpmu0}, we have that
 \begin{subequations}
 \begin{align}
    \int_{I} \int_{\Om} \partial_{F} W(\nabla y_{\eps}, \theta_{\eps}) : \nabla z \di x \di t & \to  \int_{I} \int_{\Om} \partial_{F} W(\nabla y, \theta) : \nabla z \di x \di t\,,\label{e:limW}\\
    \int_{I} \int_{\Om} DH(\Delta y_{\eps}) \cdot \Delta z \di x \di t & \to \int_{I} \int_{\Om} DH(\Delta y) \cdot \Delta z \di x \di t \,,\label{e:limH}\\
   \int_I  \int_{\Om} \partial_{\dot{F}} R(\nabla y_\eps, \partial_t \nabla y_\eps, \theta_\eps) : \nabla z \di x \di t & \to \int_I  \int_{\Om} \partial_{\dot{F}} R(\nabla y, \partial_t \nabla y, \theta) : \nabla z \di x \di t \,,\label{e:limdiss}\\
    \int_{I} \int_{\Om} \partial_{t} y_{\eps} \cdot \partial_t z\di x \di t & \to \int_{I} \int_{\Om} \partial_{t} y \cdot \partial_t z\di x \di t \,,\label{e:limpartial}\\
    \int_{\Om} y_{0, \eps} \cdot z(0) \di x & \to \int_{\Om} y_{0} \cdot z(0) \di x \,.\label{e:liminitial}
 \end{align}
 \end{subequations}
 In particular, in~\eqref{e:limH} we have used~\eqref{mu_def_Linfty_W1infty_conv} and in~\eqref{e:limdiss} we have exploited the linear structure of~$\partial_{\dot{F}} R$ with respect to~$\partial_{t} y_{\eps}$. Finally, estimate~\eqref{mu_H1_def} implies that
 \begin{displaymath}
     \eps \int_{I} \int_{\Om}  \partial_t \nabla \Delta  y_\eps : \nabla \Delta z \di x \di t \to 0\,.
 \end{displaymath}
 Hence, the pair~$(y, \theta)$ satisfies~\eqref{mechanical_equation_final-def-unreg}.
 \end{proof}

\EEE

\BBB

\subsection{The heat-transfer equation} \EEE

We are left to consider the limit as $\eps \to 0$ in the \BBB heat-transfer \EEE equation~\eqref{e:regularized-thermal}. To this purpose, we now derive a weaker form of the regularized heat equation which is suitable for the limit procedure. This  relies on integration by parts and on the chain rule for the mechanical \ZZZ energy. \EEE  For notational convenience, for $\psi \in C^\infty(\overline{\Omega})$ we define  
\begin{align}\label{speEnot}
\toten(y,\theta;\psi) = \int_\Omega  \Big(\elpot(\nabla y) +  \hypot(\Delta y) + W^{\rm in}(y,\theta) \Big) \psi  \di x.   
\end{align}
Notice that for the limiting passage \BBB we will also crucially use the bounds in  \eqref{e:lar-reg-eps-finale}.  
\EEE 

\begin{proposition}
\label{p:new_formulation}
For every $\eps >0$, every $\varphi \in C^{\infty} (I \times \overline{\Om})$   \BBB of the form $\varphi = \psi \eta $ for  \EEE $\psi \in C^{\infty} (\overline{\Om})$ and $\eta \in C^{\infty}(\ZZZ I \EEE )$ with $\eta (T) = 0$  it holds
\begin{align}
\label{e:nnn}
0 = & \intQ   \eta \,  \hcm(\nabla y_{\eps}, \theta_{\eps}) \nabla \theta_{\eps} \cdot \nabla \psi           - \kappa  \BBB \int_I \EEE \int_{\partial \Omega}  \eta \psi \, (\bt - \theta_{\eps})   \di \haus^{d-1} \di t   -  \intQ \psi \eta \, f \cdot \partial_{t} y_{\eps} \di x \di t \notag \\
 &    -   \int_I \partial_{t} \eta\, \Big( \toten(y_{\eps}, \theta_\eps; \psi  )   + \int_\Omega \frac{\rho}{2} | \partial_{t} y_{\eps}|^{2} \psi \di x   \Big) \di t 
-            \eta (0) \,  \Big( \toten(y_{0,\eps}, \theta_0; \psi)  +  \int_{\Om} \frac{\rho}{2}    | y'_{0, \eps}|^{2} \psi  \, \di x \Big) 
  \notag \\
 &  + \intQ \eta  \Big(   \pl_F W(\nabla y_{\eps}, \theta_{\eps}) + \pl_{\dot{F}} R (\nabla y_{\eps}, \partial_{t} \nabla  y_{\eps}, \theta_{\eps})  \Big) : (\partial_{t} y_{\eps} \otimes \nabla \psi) \, \di x \, \di t      \notag  \\
&  - \intQ \eta \, {D H}(\Delta y_{\eps}) \cdot \partial_{t} y_{\eps} \Delta \psi \, \di x \di t   - 2\intQ \eta \, \nabla ({D H}(\Delta y_{\eps}) ) : (\partial_{t} y_{\eps} \otimes \nabla \psi)\, \di x \di t  \notag  \\  
 & - \eps\intQ \eta \, \partial_{t} \Deltatwo  y_{\eps} \cdot \big( 2\partial_{t}  \nabla  y_{\eps} \nabla\psi + {\rm div} (\partial_{t} y_{\eps} \otimes \nabla \psi) \big) \di x \di t \notag \\ 
& +   \eps\intQ \eta \, \partial_{t} \nabla \Delta  y_{\eps} :  \partial_{t} \nabla y_{\eps} \Delta \psi \di x \di t  
      - 2\eps\intQ \eta \, \partial_{t} \nabla \Delta y_{\eps} :  \partial_{t} \nabla  y_{\eps} \nabla^{2}\psi \di x \di t .
\end{align}
\end{proposition}

\BBB Note that except for the $\eps$-dependent terms this formulation coincides with the one in \eqref{e:new-thermal-equation-lim}. \EEE

\begin{proof}
\BBB For $\eps>0$, we define \EEE  $z \defas \eta \psi \partial_{t} y_{\eps}$ and,  \ZZZ  by \eqref{mu_H1_def}, \EEE we note that $z \in L^{2} (I; H^{3}(\Om; \R^{d}))$ and $z = 0$ in $I \times \partial \Omega$. Recall that  $\partial^{2}_{tt} y_{\eps} \in L^{2}(I; (H^{3} (\Om; \R^{d}) \cap H^{1}_{0} (\Om; \R^{d}))^{*})$ \BBB by Proposition \ref{prop:h_convergence_mechanical2}.    Testing  the  heat-transfer \EEE equation~\eqref{e:regularized-thermal} and performing an integration by parts in time we may write
\begin{align}
\label{e:nnn3}
\Pi & \defas  \intQ     \eta \psi \, \BBB \xi \EEE ( \nabla y_{\eps} , \partial_t \nabla y_{\eps} , \theta_{\eps} ) \di x \di t
    + \eps \int_I \int_{\Om}  \eta \psi \,   | \partial_{t} \nabla \Delta  y_{\eps} |^{2} \di x \di t  + \intQ   \eta \psi \, \pl_F \cplpot(\nabla y_{\eps}, \theta_{\eps}) : \pl_t \nabla y_{\eps} \, \di x \di t \nonumber
    \\
    &
    = \intQ  \Big( \eta \,  \hcm(\nabla y_{\eps}, \theta_{\eps}) \nabla \theta_{\eps} \cdot \nabla \psi         -  \psi \, w_{\eps} \partial_{t} \eta  \Big) \, \di x \di t
    - \kappa  \BBB \int_I \EEE \int_{\partial \Omega}  \eta \psi \, (\bt - \theta_{\eps})   \di \haus^{d-1} \di t
     - \int_\Omega \BBB w_{0,\eps} \EEE \psi \eta ( 0 )  \di x. 
\end{align}
\BBB Our goal is to rewrite the terms on the left hand side, i.e., $\Pi$. To this end, \EEE we test the regularized mechanical equation~\eqref{mechanical_equation} \ZZZ with $z = \eta \psi \partial_{t} y_{\eps}$: \BBB   this yields \EEE 
\begin{align}
\label{e:bbb}
&\intQ 
       \eta \Big(
        \pl_F \elpot(\nabla y_\eps) + \pl_{F} W^{\rm cpl} (\nabla y_{\eps}, \theta_{\eps}) \Big) 
         : \nabla (\psi \partial_{t} y_{\eps})  \di x \di t + \intQ 
        \eta \pl_{\dot F} \disspot(\nabla y_\eps, \partial_t \nabla y_\eps, \theta_\eps)
       : \nabla (\psi \partial_{t} y_{\eps})  \di x \di t  \nonumber
         \\
         & \quad  +  \eps \int_I \int_{\Om}  \eta \, \partial_{t} \nabla \Delta y_{\eps}  :   \nabla \Delta (\psi \partial_{t} y_{\eps}) \di x \di t   +   \intQ   \eta  \BBB D \EEE
      \hypot( \Delta  y_\eps)  \cdot \Delta (\psi \partial_{t} y_{\eps} ) \, \di x \di t \notag \\ &   \quad     \quad     \quad     \quad     \quad     \quad     \quad     \quad     \quad     \quad     \quad     \quad     \quad   
        \          = 
            \intQ \psi \eta \, f \cdot \partial_{t} y_{\eps} \di x \di t      - \rho \int_I  \left\langle \partial^{2}_{tt} y_{\eps} , \partial_{t} y_{\eps} \, \eta \psi  \right\rangle  \di t  
  .\end{align} 
\BBB We expand the terms on the left-hand side of  \eqref{e:bbb} by expanding $\nabla (\psi \partial_{t} y_{\eps})$, $\nabla \Delta (\psi \partial_{t} y_{\eps})$, and $\Delta (\psi \partial_{t} y_{\eps})$. This yields 
\begin{align}\label{rhs}
\intQ \!
       \eta \Big(
        \pl_F \elpot(\nabla y_\eps) + \pl_{F} W^{\rm cpl} (\nabla y_{\eps}, \theta_{\eps}) \Big) 
         : \nabla (\psi \partial_{t} y_{\eps})  \di x \di t & =   \intQ \!
       \eta\psi \,   \pl_{F} W^{\rm cpl} (\nabla y_{\eps}, \theta_{\eps})
         : \partial_{t} \nabla  y_{\eps} \di x \di t + J_0  +  J_1,\notag \\ 
     \intQ         \eta \pl_{\dot F} \disspot(\nabla y_\eps, \partial_t \nabla y_\eps, \theta_\eps)
       : \nabla (\psi \partial_{t} y_{\eps})  \di x \di t  & =  \intQ 
        \eta\psi \,  \pl_{\dot F} \disspot(\nabla y_\eps, \partial_t \nabla y_\eps, \theta_\eps)
       : \partial_{t} \nabla y_{\eps} \di x \di t  + J_2,\notag \\
   \eps  \intQ \eta\,  \partial_{t} \nabla \Delta y_{\eps}  :   \nabla \Delta (\psi \partial_{t} y_{\eps}) \di x \di t  &     
     = \eps\intQ\eta \psi \, |\partial_{t} \nabla \Delta  y_{\eps}|^{2} \di x \di t + J_3 + J_4, \notag  \\
      \intQ   \eta D 
      \hypot( \Delta  y_\eps)  \cdot \Delta (\psi \partial_{t} y_{\eps} ) \, \di x \di t &= \intQ \eta \psi \, {D H}(\Delta y_\eps) \cdot \partial_{t}\Delta y_{\eps} \, \di x \di t   + J_5,   
\end{align}
where for brevity we \ste have written\BBB  
\begin{align}\label{allJ}
J_0 &\defas \intQ \eta \psi \, \partial_{F} W^{\rm el} (\nabla y_{\eps}) : \partial_{t} \nabla  y_{\eps} \di x \di t, \notag \\  J_1 & \defas \intQ \eta  \Big(        \pl_F \elpot(\nabla y_\eps) + \pl_{F} W^{\rm cpl} (\nabla y_{\eps}, \theta_{\eps}) \Big) : (\partial_{t} y_{\eps} \otimes \nabla \psi) \, \di x \, \di t, \notag \\
 J_2 &\defas   \intQ \eta \,  \pl_{\dot F} \disspot(\nabla y_\eps, \partial_t \nabla y_\eps, \theta_\eps)        : (\partial_{t} y_{\eps} \otimes \nabla \psi) \di x \di t, \notag \\ 
       J_3 & \defas  \eps \intQ \eta \, \partial_{t} \nabla \Delta  y_{\eps} : ( \partial_{t} \Delta  y_{\eps} \otimes \nabla \psi) \di x \di t   , \notag \\ 
        J_4 & \defas  \eps \intQ \eta \,   \partial_{t} \nabla \Delta y_{\eps} : \nabla (  \partial_{t} \nabla y_{\eps} \nabla \psi) \di x \di t + \eps \intQ\eta \,  \partial_{t} \nabla \Delta  y_{\eps} : \nabla ({\rm div} ( \partial_{t} y_{\eps} \otimes \nabla \psi)) \di x \di t , \notag \\    
 J_5 &\defas \intQ \eta \, {D H}(\Delta y_{\eps}) \cdot \partial_{t} y_{\eps} \Delta \psi \, \di x \di t + 2\intQ \eta \, {D H}(\Delta y_{\eps}) \cdot \partial_{t} \nabla y_{\eps} \nabla \psi \, \di x \di t.
\end{align}
Recalling \eqref{diss_rate}, we see that the first three terms on the right-hand sides of \eqref{rhs} correspond to the terms in \eqref{e:nnn3} whose sum is denoted by $\Pi$. This along with \eqref{e:bbb} yields 
\begin{align*}
&\Pi + \sum_{i=0}^4J_i  = 
           \intQ \psi \eta \, f \cdot \partial_{t} y_{\eps} \di x \di t   -   \intQ   \eta  \BBB D \EEE
      \hypot( \Delta  y_\eps)  \cdot \Delta (\psi \partial_{t} y_{\eps} ) \, \di x \di t    
 - \rho \int_I   \left\langle \partial^{2}_{tt} y_{\eps} ,  \partial_{t} y_{\eps} \, \eta \psi  \right\rangle   \di t
   \nonumber.
\end{align*}
Using the last equation in \eqref{rhs} and the definition of $J_0$, we get
\begin{align}   \label{e:ccccooo}
&\Pi + \sum_{i=1}^5J_i  = 
         \intQ \psi \eta \, \Big( f \cdot \partial_{t} y_{\eps}    -   \partial_{F} W^{\rm el} (\nabla y_{\eps}) : \partial_{t} \nabla  y_{\eps}  - {D H}(\Delta y_\eps) \cdot \partial_{t}\Delta y_{\eps} \Big) \, \di x \di t   
 - \rho \int_I   \left\langle \partial^{2}_{tt} y_{\eps} ,  \partial_{t} y_{\eps} \, \eta \psi  \right\rangle   \di t
    \,.
\end{align}
\BBB Using the chain rule and $\eta(T) = 0$, \EEE we rewrite the last term on the right-hand side   as
    \begin{align}
    \label{e:nnn2}
        \rho \int_I \left\langle \partial^{2}_{tt} y_{\eps},  \partial_t y_{\eps} \psi \eta \right\rangle  \di t & =
        \frac{\rho}{2} \intQ \frac{\di}{\di t} (| \partial_{t} y_{\eps}|^{2} \eta)\psi \di x \di t - \frac{\rho}{2} \intQ \psi | \partial_{t} y_{\eps}|^{2}  \partial_{t} \eta \di x \di t \\
        &
        =  -\frac{\rho}{2} \int_{\Om} \ZZZ \psi \EEE \eta(0)   \,  | y'_{0, \eps}|^{2} \, \di x  - \frac{\rho}{2} \intQ \psi | \partial_{t} y_{\eps}|^{2}  \partial_{t} \eta \di x \di t .\nonumber
\end{align}
By the chain rule \ZZZ for the mechanical energy, \EEE see \eqref{the cain rule} \BBB (for $y_\eps$ in place of $y_h$), and integration by parts \EEE we   have 
\begin{align}
\label{e:nnn4}
\intQ \eta \psi \, & \Big( \partial_{F} W^{\rm el} (\nabla y_{\eps}) : \partial_{t} \nabla  y_{\eps}  + {D H}(\Delta y_\eps) \cdot \partial_{t}\Delta y_{\eps} \Big) \, \di x \di t  
      = \int_I \eta \, \frac{\di}{\di t} \int_{\Om} \psi\big(  \elpot (\nabla y_{\eps}) + H(\Delta y_{\eps} ) \big)\di x \di t   \nonumber
      \\
      &
      = - \int_{\Om} \eta (0) \, \psi  \big(  \elpot (\nabla y_{0,\eps}) + H(\Delta y_{0,\eps} ) \big)\di x - \intQ \partial_{t} \eta\, \psi \big(   \elpot (\nabla y_{\eps}) + H(\Delta y_{\eps} ) \big)\di x \di t  .
\end{align}
\BBB Next, we manipulate $J_3$, $J_4$, $J_5$. \EEE Recall that $\Delta y_{\eps} = 0$ and $\partial_{t} y_{\eps} =0$ on~$\partial \Om$ for a.e.~$t \in I$ (see \eqref{e:bdry_eps}) \BBB and $DH(0)=0$ by \eqref{e:Hdef}--\eqref{e:psi}. \EEE We   \BBB   perform an integration by parts in  $J_5$  to get \EEE 
\begin{align}
\label{e:ccc0}
       J_5 &  =  \intQ \eta \,  {D H}(\Delta y_{\eps}) \cdot \partial_{t} y_{\eps} \Delta \psi \, \di x \di t - 2\intQ \eta \, \nabla ({D H}(\Delta y_{\eps}) ) : (\partial_{t} y_{\eps} \otimes \nabla \psi)\, \di x \di t \notag \\ & \quad - 2 \intQ \eta \, {D H}(\Delta y_{\eps}) \cdot \partial_{t} y_{\eps} \Delta \psi\, \di x \di t \nonumber
      \\
      &
      =   - \intQ \eta \, {D H}(\Delta y_{\eps}) \cdot \partial_{t} y_{\eps} \Delta \psi \, \di x \di t   - 2\intQ \eta \, \nabla ({D H}(\Delta y_{\eps}) ) : (\partial_{t} y_{\eps} \otimes \nabla \psi)\, \di x \di t .   
\end{align}
By integrating by parts and recalling the boundary conditions $ \Delta \partial_{t} y_{\eps} = \partial_{\nu}  \Delta  \partial_{t} y_{\eps} = 0$ on $\partial \Om$ for a.e.~$t \in I$ (see \eqref{e:bdry_eps}),   we  rewrite $J_4$ as 
\begin{align}
\label{e:eee}
J_4     = - \eps\intQ \eta \, \partial_{t} \Deltatwo  y_{\eps} \cdot \big( \partial_{t}  \nabla  y_{\eps} \nabla\psi + {\rm div} (\partial_{t} y_{\eps} \otimes \nabla \psi) \big) \di x \di t,
 \end{align}
 \BBB and, by elementary but tedious computations,  $J_3$ can be written as \EEE
 \begin{align}
    \label{e:fff}
\frac{1}{\eps} J_3 & =   \intQ \eta \, \partial_{t} \nabla \Delta  y_{\eps}   : ( \partial_{t} \Delta  y_{\eps} \otimes \nabla \psi) \di x \di t 
    = - \intQ  \eta \, \partial_{t} \Delta  y_{\eps} \cdot \big( (\partial_{t} \nabla \Delta  y_{\eps} \nabla \psi) + \partial_{t} \Delta  y_{\eps} \Delta \psi \big) \di x \di t \notag
    \\
    &
    = \intQ \eta \, \partial_{t}\nabla \Delta  y_{\eps} : \big( \partial_{t} \nabla^{2}  y_{\eps} \nabla \psi + \partial_{t}  \nabla  y_{\eps} \Delta \psi\big) \di x \di t   + \intQ \eta \, \partial_{t} \Delta  y_{\eps} \cdot \big(\partial_{t} \nabla^{2} y_{\eps} \nabla^{2} \psi + \partial_{t} \nabla  y_{\eps} \nabla \Delta \psi \big) \di x \di t \nonumber
    \\
    &
    = - \intQ \eta \, \partial_{t} \Deltatwo  y_{\eps} \cdot ( \partial_{t} \nabla  y_{\eps} \nabla  \psi)  \di x \di t + \intQ \eta \, \partial_{t} \nabla \Delta  y_{\eps} :  \partial_{t} \nabla y_{\eps} \Delta \psi \di x \di t \nonumber
    \\
    &
    \quad - 2\intQ \eta \, \partial_{t} \nabla \Delta y_{\eps} :  \partial_{t} \nabla  y_{\eps} \nabla^{2}\psi \di x \di t
    \,.  
\end{align}
\BBB Recalling the definition in \eqref{speEnot} and combining \eqref{e:nnn3} with  \eqref{allJ}--\eqref{e:fff} we obtain the statement. More precisely, \eqref{e:nnn2} and \eqref{e:nnn4} contribute to  the \ZZZ  second  line \EEE in \eqref{e:nnn}, the third line corresponds to $J_1+J_2$ (see \eqref{allJ}) and the last three lines correspond to $J_3+J_4+J_5$ (see \eqref{e:ccc0}--\eqref{e:fff}).  
\end{proof}

We are now in a position to pass to the limit as $\eps \to 0$ in the modified heat-transfer equation~\eqref{e:nnn}. This will conclude the \ZZZ proof \BBB of Theorem~\ref{thm:main-thermal-elasto-unregu}. \EEE

\begin{proposition}
\label{thm:mu_convergence_thermal}
Let $(y, \theta)$ be  \BBB given by \EEE Lemma~\ref{l:mu-compactness}. Then, $(y, \theta)$ satisfies~\eqref{e:new-thermal-equation-lim}. 
\end{proposition}

\begin{proof}
\BBB In order to show \eqref{e:new-thermal-equation-lim}, by a density argument it suffices to consider test functions of the form   $\varphi = \psi \eta $ for    $\psi \in C^{\infty} (\overline{\Omega})$ and $\eta \in C^{\infty} (I)$ with $\eta(T) = 0$. \EEE We test~\eqref{e:nnn} with $\varphi$ and pass to the limit term by term.  

Thanks to the convergences stated in Lemma~\ref{l:mu-compactness}, 
the estimates~\eqref{e:epsilon_a_priori_1}--\eqref{e:lar-reg-eps-finale}, and the assumptions~\ref{W_regularity}--\ref{W_lower_bound}, \ref{H_regularity}--\ref{H_bounds},~\ref{C_regularity}--\ref{C_bounds}, \ref{D_quadratic}--\ref{D_bounds}, and \eqref{e:hpmu0}, we have that \BBB each of the terms in the first three lines of \eqref{e:nnn} converges to the respective term in the first three lines of \eqref{e:new-thermal-equation-lim} (with $\eta \psi$ in place of $\varphi$).  Here, we againg exploit that $ \pl_{\dot F} \disspot$ is linear in the second entry, cf.\ \ref{D_quadratic}.   Therefore, it suffices to show
\begin{align}\label{STS1}
 &\lim_{\eps \to 0} \Big( \intQ \eta \, {D H}(\Delta y_{\eps}) \cdot \partial_{t} y_{\eps} \Delta \psi \, \di x \di t   + 2\intQ \eta \, \nabla ({D H}(\Delta y_{\eps}) ) : (\partial_{t} y_{\eps} \otimes \nabla \psi)\, \di x \di t  \Big) \notag \\
& =  \intQ \eta \, {D H}(\Delta y) \cdot \partial_{t} y \Delta \psi \, \di x \di t   + 2\intQ \eta \, \nabla ({D H}(\Delta y) ) : (\partial_{t} y \otimes \nabla \psi)\, \di x \di t ,
 \end{align}
\begin{align}\label{STS12}
\lim_{\eps \to 0} \eps\intQ \eta \, \partial_{t} \Deltatwo  y_{\eps} \cdot \big( 2\partial_{t}  \nabla  y_{\eps} \nabla\psi + {\rm div} (\partial_{t} y_{\eps} \otimes \nabla \psi) \big) \di x \di t  = 0,
\end{align}
\begin{align}\label{STS13} 
\lim_{\eps \to 0}  \Big( \eps\intQ \eta \, \partial_{t} \nabla \Delta  y_{\eps} :  \partial_{t} \nabla y_{\eps} \Delta \psi \di x \di t  
      - 2\eps\intQ \eta \, \partial_{t} \nabla \Delta y_{\eps} :  \partial_{t} \nabla  y_{\eps} \nabla^{2}\psi \di x \di t  \Big) \ste = 0 .
\end{align} \BBB
We start with \eqref{STS1}. Recalling that $\partial_t y_\eps$ converges strongly in $L^2(I;L^p(\Omega;\R^d))$ by \eqref{mu_yt_L2_L2} (recall $p <6$ for $d=3$), the key point is to show \ste the strong (resp.~weak) convergence of ${D H}(\Delta y_{\eps})$ (resp.~$\nabla ({D H}(\Delta y_{\eps})) $)  in $L^2(I; L^{p'} (\Om))$. \EEE 
Since \ste $ y_{\eps} \to  y$ in $L^{2} (I ; W^{2, p} (\Om; \R^{d}))$ (cf.~\eqref{mu_def_Linfty_W1infty_conv}),
we infer \ZZZ by \ref{H_bounds} \EEE that \EEE 
\begin{align}\label{LH1}
    \eta \,{D H}(\Delta y_{\eps}) \ \to \ \eta {D H}(\Delta y) \qquad \text{in $L^{2} (I; L^{p'} (\Om; \R^{d}))$.}
\end{align}
\BBB In view of  \eqref{e:lap-reg-mu1.5-final}, by weak compactness  \EEE  we also have that
\begin{align}\label{LH2}
\eta \,\nabla ({D H}(\Delta y_{\eps}))  \ \rightharpoonup \ \eta \nabla ( {D H}(\Delta y) )\qquad \text{in $L^{2}(I; L^{p'} (\Om; \R^{d \times d}))$.} 
\end{align}
Now, \eqref{LH1}--\eqref{LH2} along with \eqref{mu_yt_L2_L2} show \eqref{STS1}.

\BBB Eventually, \EEE by H\"older's inequality, we get that
\begin{align*}
 \bigg|\intQ   \eta \, \partial_{t} \Deltatwo  y_{\eps} & \cdot \big( 2\partial_{t} \nabla  y_{\eps} \nabla \psi + {\rm div} ( \partial_{t} y_{\eps} \otimes \nabla \psi) \big) \di x \di t \bigg| 
 \\
 &
 \leq C \|  \partial_{t} \Deltatwo  y_{\eps} \|_{L^{2}(I \times \Om)}  \big( \| \eta \nabla \psi \|_{L^{\infty} (\Omega)}  \ZZZ + \| \eta \nabla^2 \psi \|_{L^{\infty} (\Omega)} \big) \EEE \big(   \|  \partial_{t} y_{\eps}\|_{L^{2}(I \times \Om)} +  \ZZZ  \|    \, \partial_{t} \nabla y_{\eps}\|_{L^{2}(I \times \Om)}  \EEE \big) \,,
 \\
 \bigg|\intQ  \eta \,  \partial_{t} \nabla \Delta y_{\eps}  & : \big( 2 \partial_{t} \nabla  y_{\eps} \nabla^{2} \psi - \partial_{t} \nabla  y_{\eps} \Delta \psi) \di x \di t \bigg|
 \\
 &
\leq C  \|   \partial_{t} \nabla \Delta  y_{\eps} \|_{L^{2}(I \times \Om)} \BBB \| \eta  \nabla^{2} \psi \|_{L^{\infty} (\Omega)}  \EEE   \|   \partial_{t} \nabla  y_{\eps}\|_{L^{2}(I \times \Om)}\,.
\end{align*}
Thus, we infer from \BBB \eqref{mu_H1_def} and   \eqref{e:lap-reg-mu2-final} (recall $\varrho <1$) that \eqref{STS12}--\eqref{STS13} hold. This concludes the proof. \EEE 
\end{proof}

\BBB
\begin{proof}[Proof of Theorem \ref{thm:main-thermal-elasto-unregu}]
The weak formulation follows from Propositions \ref{thm:mu_convergence_mechanical} and \ref{thm:mu_convergence_thermal}. \ste We deduce the   regularity  $y\in L^2(I; H^3(\Omega;\R^d))$ and     $(1+\abs{\Delta y})^\frac{p-2}{2}\abs{\nabla \Delta y}^2\in L^2(I\times \Omega)$ by \BBB the bound \eqref{e:lap-reg-mu-final} and lower semicontinuity of norms as $\eps \to 0$, again applying an elliptic regularity estimate.
\end{proof}
\EEE

\BBB 

\subsection*{Acknowledgements} RB  and MF  have been supported by  the DFG project FR 4083/5-1 and  by the Deutsche Forschungsgemeinschaft (DFG, German Research Foundation) under Germany's Excellence Strategy EXC 2044 -390685587, Mathematics M\"unster: Dynamics--Geometry--Structure, as well as by the DAAD projects 57600633 and 57702972. SS  was supported by the ERC-CZ Grant CONTACT LL2105 funded by the Ministry of Education, Youth and Sport of the Czech Republic, and  also acknowledges the support of the VR Grant 2022-03862 of the Swedish Science Foundation and the supported by Charles University Research Centre program No. UNCE/24/SCI/005. SS is a member of the Necas center for mathematical modeling. SA has been supported by the Austrian Science Fund through the Stand Alone project 10.55776/P35359, by the Italian Ministry of Research (MUR) through the PRIN 2022 Project No.~2022HKBF5C ``Variational Analysis of Complex Systems in Materials Science, Physics and Biology'', and by the University of Naples Federico II through the FRA Project ``ReSinApas''. SA is also member of the Gruppo Nazionale per l'Analisi Matematica, la Probabilit\`a e le loro Applicazioni (INdAM--GNAMPA).

The authors would like to thank {\sc M.~Kru\v{z}\'ik} and {\sc U.~Stefanelli} for fruitful discussions on the content of the work. The authors are thankful to the University of Vienna for its hospitality, and to  the Erwin Schr\"odinger Institute in Vienna where part of this work has been developed during the workshop ``Between Regularity and Defects: Variational and Geometrical Methods in Materials Science".

\appendix 

\section{A special case of elliptic regularity}

\BBB We formulate and prove the lemma used in Subsection \ref{sec: APB}. \EEE 

\begin{lemma}[A special case of elliptic regularity]\label{lem:ellip_reg_spec}
  \BBB Consider the Banach space  $X \defas W^{2,q}(\Omega; \R^d) \cap H^1_0(\Omega; \R^d)$ for some $q>1$. \EEE   Moreover, let $u \in H^3(\Omega; \R^d) \cap H^1_0(\Omega; \R^d)$ and $g \in X^*$ be such that for all $z \in C^\infty(\overline \Omega; \R^d)$ with $z = 0$ on $\partial \Omega$ it holds that
  \begin{equation}\label{ellip_reg_weak_form}
    \int_\Omega \nabla \Delta u : \nabla \Delta z = \langle g, z \rangle,
  \end{equation}
  where $\langle \cdot, \cdot \rangle$ denotes the dual pairing between $X$ and $X^*$.
  Then, the following holds true:
  \begin{enumerate}[label=(\alph*)]
  \item \label{item:ellip_reg_H4}
We have that $u \in \BBB W^{4,q'} \EEE (\Omega; \R^d)$ with
    \begin{equation}\label{u_H4_bound}
      \|u  \|_{\BBB W^{4,q'}(\Omega)} \leq C \|g\|_{X^*} + \BBB C|\mu|
    \end{equation}
    for a constant $C \BBB > 0 \EEE $ only depending on $\Omega$, \BBB where $\mu \defas \mint_\Omega \Delta u \di x$.  \EEE
    Moreover, the following boundary condition holds true:
    \begin{equation}\label{neumann_laplace_u}
      \partial_\nu \Delta u = 0 \qquad \text{$\haus^{d-1}$-a.e.~on } \partial\Omega.
    \end{equation}
  \item \label{item:ellip_reg_H5}
    If we additionally have $g \in H^{-1}(\Omega; \R^d)$, then $u \in H^5(\Omega; \R^d)$ with
    \begin{equation}\label{u_H5_bound}
      \|u    \|_{H^5(\Omega)} \leq C \|g\|_{H^{-1}(\Omega)} \BBB + C|\mu|, \EEE
    \end{equation}
    satisfying the boundary condition
    \begin{equation}\label{dirichlet_laplace2x_u}
      \Deltatwo u = 0 \qquad \text{$\haus^{d-1}$-a.e.~on } \partial\Omega.
    \end{equation}
  \end{enumerate}  
\end{lemma}

\begin{proof}
  \textit{Step 1 ($W^{4,q'}$-regularity):}
  Using \eqref{ellip_reg_weak_form} and integrating by parts we see that for all $z \in C^\infty(\overline \Omega; \R^{d})$ with \BBB $z=\partial_\nu \Delta z = 0$ \EEE on $\partial \Omega$ \BBB it holds that \EEE
  \begin{equation}\label{ellip_reg_weak_form_2}
    \langle g, z \rangle
    = \int_\Omega \nabla \Delta u : \nabla \Delta z \di x
    = - \int_\Omega \Delta u \cdot \Deltatwo z \di x + \int_{\partial \Omega} \Delta u \cdot \partial_\nu \Delta z \di \haus^{d-1}
    = - \int_\Omega \Delta u \cdot \Deltatwo z \di x.
  \end{equation}
\BBB By representation of the dual space we find  $G,G_j,G_{ij} \in L^{q'}(\Omega;\R^d)$ such that 
\begin{align}\label{weak_form_vk0}
\langle g ,z \rangle = \int_\Omega G\cdot z \di x + \sum_{j=1}^d G_j\cdot \partial_jz \di x + \sum_{i,j=1}^d G_{ij}\cdot \partial^2_{ij}z \di x
\end{align}
with  $\Vert G \Vert_{L^{q'}(\Omega)} + \sum_{j=1}^d\Vert G_j \Vert_{L^{q'}(\Omega)}+ \sum_{i,j=1}^d\Vert G_{ij} \Vert_{L^{q'}(\Omega)} \le C \Vert g \Vert_{X^*}$.   We approximate $G$, $G_j$,  and $G_{ij}$ by sequences $G^k, G_j^k, G_{ij}^k \in C^\infty_c(\Omega;\R^d)$ converging to the respective functions in $L^{q'}(\Omega;\R^d)$. We  set
$$g_k \defas G_k  - \sum_{j=1}^d \partial_j G_j^k + \sum_{i,j=1}^d \partial_{ij}^2 G_{ij}^k, $$
\EEE     
and let $v_k \in H^1_0(\Omega; \R^d)$ be \BBB the \EEE weak solution to
  \begin{equation*}
  \left\{
  \begin{aligned}
    - \Delta v_k &= g_k &&\text{in } \Omega, \\
    v_k &= 0 &&\text{on } \partial \Omega.
  \end{aligned}
  \right.
  \end{equation*}
  In particular, \BBB by integration by parts, \EEE for all $z \in C^\infty(\overline \Omega; \R^d)$ with $z = 0$ on $\partial \Omega$ it holds that
  \begin{equation}\label{weak_form_vk}
    \int_\Omega v_k \cdot \Delta z \di x = - \BBB \int_\Omega g_k \cdot z \, \di x = -  \int_\Omega  \Big( G_k \cdot z  + \sum_{j=1}^d G_j^k \cdot \partial_j z + \sum_{i,j=1}^d G_{ij}^k \cdot \partial^2_{ij} z \Big) \, \di x. \EEE
  \end{equation}
  Moreover, let $\BBB w_k \EEE \in H^1_0(\Omega; \R^d)$ be \BBB the \EEE weak solution to
  \begin{equation*}
  \left\{
  \begin{aligned}
    \Delta w_k &= \BBB |v_k|^{\frac{2-q}{q-1}} \EEE v_k &&\text{in } \Omega, \\
    w_k &= 0 &&\text{on } \partial \Omega.
  \end{aligned}
  \right.
  \end{equation*}
  As $\Omega$ has   $C^5$-boundary,  by elliptic regularity we see that \BBB $w_k \in W^{2,q}(\Omega; \R^d)$ with
  \begin{equation*}
   \|w_k\|_{W^{2,q}(\Omega)} \leq C \||v_k|^{\frac{2-q}{q-1}}  v_k\|_{L^{q}(\Omega)} = C \|  v_k\|^{1/(q-1)}_{L^{q'}(\Omega)},
  \end{equation*}
  where the constant $C$ only depends on $\Omega$.
  With \eqref{weak_form_vk}, this shows \BBB
  \begin{align}\label{vk_L2_bound}
    \|v_k\|_{L^{q'}(\Omega)}^{q'}
   & = \int_\Omega v_k \cdot \Delta w_k \di x
    =- \int_\Omega g_k \cdot w_k \di x \notag\\
    & \leq \BBB C \Big( \Vert  G_k \Vert_{L^{q'}(\Omega)}  + \sum_{j=1}^d \Vert G_j^k \Vert_{L^{q'}(\Omega)} + \sum_{i,j=1}^d  \Vert G_{ij}^k \Vert_{L^{q'}(\Omega)}   \Big) \Vert w_k \Vert_{W^{2,q}(\Omega)}  \notag\\ &  \le   C \Vert g \Vert_{X^*}\Vert v_k \Vert^{1/(q-1)}_{L^{q'}(\Omega)}.
  \end{align} \EEE
 Consequently, there exists $v \in \BBB L^{q'} \EEE (\Omega; \R^d)$ such that, up to selecting a subsequence, $v_k \weakly v$ weakly in $\BBB L^{q'} \EEE (\Omega; \R^d)$.
  Passing to the limit $k \to \infty$ in \eqref{weak_form_vk} and \eqref{vk_L2_bound}, \BBB and recalling \eqref{weak_form_vk0}, \EEE we discover that $v$ satisfies
  \begin{align}
    \int_\Omega v \cdot \Delta z \di x &=   -\langle g, z \rangle, \label{weak_form_v} \\
    \BBB \|v\|_{L^{q'}(\Omega)} &\leq C \|g\|_{X^*}, \EEE \label{v_L2_bound}
  \end{align}
\BBB for all $z \in C^\infty(\overline \Omega; \R^d)$ with $z = 0$ on $\partial \Omega$. \EEE
\BBB Now, \EEE let $w \in H^1(\Omega; \R^d)$  \BBB with $\mint_\Omega w \di x = 0$ be  the  \EEE weak solution to
  \begin{equation}\label{neumann_w}
  \left\{
  \begin{aligned}
     \Delta w &= v \BBB - m \EEE  &&\text{in } \Omega, \\
    \partial_\nu w &= 0 &&\text{on } \partial \Omega,
  \end{aligned}
  \right. 
  \end{equation}\BBB
  with $m \defas \mint_\Omega v \di x$, \EEE
  i.e.,  $\int_\Omega w \di x = 0$ and $
 -  \int_\Omega \nabla w : \nabla z \di x = \int_\Omega (v-m) \cdot z \di x$   for all $z \in H^1(\Omega; \R^d)$.  
  As $\Omega$ has  $C^5$-boundary  and $v \in \BBB L^{q'} \EEE (\Omega; \R^d)$, by 
   elliptic regularity (see for instance \cite[Chapter 2, Section 5]{LionsM}) and \eqref{v_L2_bound} we derive that $w \in \BBB W^{2,q'}\EEE(\Omega; \R^d)$  \BBB (i.e., \eqref{neumann_w} holds in a pointwise sense) \EEE
  and
  \begin{equation}\label{w_H2_bound}
    \|w\|_{W^{2,q'}(\Omega)} \leq C \|v - \BBB m \EEE \|_{L^{q'}(\Omega)} \leq C \|g\|_{X^*}.
  \end{equation}
  Consequently, for $z$  with \BBB $z = \partial_\nu \Delta z = 0$  on $\partial \Omega$  \EEE we derive, due to $ \partial_\nu w = 0$ for $\haus^{d-1}$-a.e.~point in $\partial \Omega$, \BBB  \eqref{weak_form_v}, and \eqref{neumann_w} \EEE that  
  \begin{align*}
    \int_\Omega w \cdot \Deltatwo z \di x
    &= -\int_\Omega \nabla w : \nabla \Delta z \di x + \int_{\partial \Omega} w \cdot \partial_\nu \Delta z \di \haus^{d-1}  = \int_\Omega \Delta w \cdot \Delta z \di x \BBB - \EEE \int_{\partial \Omega} \partial_\nu w \cdot \Delta z \di \haus^{d-1} \nonumber \\
    &=   \int_\Omega v \cdot \Delta z \di x \BBB -  \int_\Omega \ste m \cdot \Delta z \di x  
    =  - \langle g, z \rangle - \int_{\Omega} m \cdot \Delta z \di x\,. \EEE  
  \end{align*}
  \BBB Using \eqref{ellip_reg_weak_form_2} we get  \EEE
  \begin{equation}\label{diff_w_weak}
    \int_\Omega (\Delta u - w) \cdot \Deltatwo z \di x = 0
  \end{equation}
  for all $z \in \BBB C^\infty \EEE(\Omega; \R^d)$ with $z = \partial_\nu \Delta z = 0$ at $\haus^{d-1}$-a.e.~point in $\partial \Omega$ \ste and $\mint_\Omega \Delta z \di x = 0$.   \BBB We now show that $\Delta u - w$ constant a.e.~in $\Omega$. To this end, let \EEE $\vphi \in \BBB C^\infty \EEE (\Omega; \R^d)$ \BBB with $\mint_\Omega \varphi \di x = 0$ \EEE  be arbitrary.
  As $\Omega$ has  $C^5$-boundary, \EEE by elliptic regularity we can find $\tilde z \in H^2(\Omega; \R^d)$ such that  $\mint_\Omega \tilde z \di x = 0$ and \BBB
  \begin{equation*}
  \left\{
  \begin{aligned}
    - \Delta \tilde z &= \vphi &&\text{in } \Omega, \\
    \partial_\nu \tilde z  &= 0 &&\text{on } \partial \Omega,
  \end{aligned}
  \right. 
  \end{equation*}
  \BBB and, subsequently, \EEE we can find $z \in H^4(\Omega; \R^d)$ satisfying
  \begin{equation*}
  \left\{
  \begin{aligned}
    - \Delta z &= \tilde z &&\text{in } \Omega, \\
    z  &= 0 &&\text{on } \partial \Omega.
  \end{aligned}
  \right. 
  \end{equation*}
  Consequently, with \eqref{diff_w_weak} this leads to $    \int_\Omega (\Delta u - w) \cdot \vphi \di x = 0$, and by the arbitrariness of $\vphi$ to  \BBB $\Delta u - w$ constant  a.e.~in $\Omega$.  As $\mint_\Omega w \di x = 0$, we get    $\Delta u - w = \mu = \mint_\Omega \Delta u \di x$. We let $\nu \in H^1_0(\Omega;\R^d)$ such that $\Delta \nu = \mu$.    \EEE   As \BBB by assumption \EEE  $u \in H^1_0(\Omega; \R^d)$ and $\Omega$ has   $C^5$-boundary, and since $w \in W^{2,q'}(\Omega; \R^d)$, \EEE by elliptic regularity we see that $u \in \BBB W^{4,q'} \EEE(\Omega; \R^d)$ and \BBB 
  \begin{equation*}
    \|u -\nu \|_{W^{4,q'}(\Omega)} \BBB \leq C \| \Delta u - \mu \|_{W^{2,q'}(\Omega)} =  \EEE  C \|w\|_{W^{2,q'}(\Omega)}.
  \end{equation*}
 \BBB This along with \eqref{w_H2_bound} \BBB and the elliptic regularity estimate $    \|\nu \|_{W^{4,q'}(\Omega)}   \le C   \| \Delta \nu \|_{W^{2,q'}(\Omega)} \le C|\mu| $  shows \EEE  \eqref{u_H4_bound}.    Finally, \eqref{neumann_laplace_u} directly follows from \BBB $\Delta u - w$ constant \EEE  and \eqref{neumann_w}.  \EEE  
  This concludes the proof of \ref{item:ellip_reg_H4}.

  \textit{Step 2 ($H^5$-regularity):}
\BBB From now on, we  assume that  \EEE $g \in H^{-1}(\Omega; \R^d)$. \BBB Since then also $g \in H^*$, Step 1 yields \ZZZ $u \in W^{4,q'} \EEE (\Omega; \R^d)$. \EEE Thus, we can integrate by parts in \eqref{ellip_reg_weak_form} and use \eqref{neumann_laplace_u} to derive
  \begin{equation}\label{ellip_reg_weak_form_3}
    -\langle g, z \rangle
    = -\int_\Omega \nabla \Delta u : \nabla \Delta z \di x
    = \int_\Omega \Deltatwo u : \Delta z \di x - \int_{\partial \Omega} \partial_\nu \Delta u \cdot \Delta z \di \haus^{d-1}
    = \int_\Omega \Deltatwo u : \Delta z \di x
  \end{equation}
 \BBB   for all $z \in C^\infty(\overline{\Omega}; \R^d)$ \ZZZ with $z = 0$ on $\partial \Omega$, \EEE where $\langle \cdot, \cdot \rangle$ now denotes the dual pairing between $H^1_0(\Omega; \R^d)$ and $H^{-1}(\Omega; \R^d)$.
  Let $\BBB v \EEE  \in H^1_0(\Omega; \R^d)$ be \BBB the \EEE weak solution to
  \begin{align}\label{NNNNN}
  \left\{
  \begin{aligned}
    - \Delta v &= g  &&\text{in } \Omega, \\
   v &= 0 &&\text{on } \partial \Omega.
  \end{aligned}
  \right. 
  \end{align}
  In particular, we have with Poincar\'e's inequality that \BBB
  \begin{equation}\label{tildew_H1_bound}
    \|v\|_{H^1(\Omega)}^2
    \leq C \|\nabla v\|_{L^2(\Omega)}^2
    = \langle g, v \rangle
    \leq \|g\|_{H^{-1}(\Omega)} \|v\|_{H^1(\Omega)}.
  \end{equation} \EEE
  Furthermore, for all $z \in C^\infty(\overline{\Omega} ; \R^d)$ with $z = 0$ on $\partial \Omega$ it holds that
   \begin{align}\label{NNNNN2}
    -\langle g, z \rangle
    = -\int_\Omega \nabla v : \nabla z \di x
    = \int_\Omega v \cdot \Delta z - \int_{\partial \Omega} v \cdot \BBB \partial_{\nu} \EEE z \di \haus^{d-1}
    = \int_\Omega v\cdot \Delta z.
  \end{align}
  Subtracting \BBB this \EEE from \eqref{ellip_reg_weak_form_3}, we arrive at
  \begin{equation*}
    \int_\Omega (\Deltatwo u - v) \cdot \Delta z \di x = 0
  \end{equation*}
  for all $z \in C^\infty(\overline{\Omega}; \R^d)$ with $z = 0$ on $\partial \Omega$.
  By an argument similar to the one from Step 1 this leads to $\Deltatwo u = v$ a.e.~in $\Omega$. \ZZZ This also shows $\int_\Omega v \di x = \int_{\partial \Omega}  \partial_\nu \Delta u \di \mathcal{H}^{d-1} = 0  $ by \eqref{neumann_laplace_u}.    Then, \EEE    in view of  \ZZZ  \eqref{neumann_laplace_u} and \EEE \eqref{tildew_H1_bound}, \EEE we derive by elliptic regularity \ZZZ for Neumann problems (see for instance \cite[Chapter 2, Section~5]{LionsM}) \EEE that $\Delta u \in H^3(\Omega; \R^d)$ such that
  \begin{equation*}
    \|\Delta u   \BBB -\mu \EEE \|_{H^3(\Omega)} \leq  \BBB C  \|v\|_{H^1(\Omega)}\le \EEE  C \|g\|_{H^{-1}(\Omega)},
  \end{equation*}
 \BBB  where as before $\mu = \mint_\Omega \Delta u \di x$.  \EEE
  Hence, as $u = 0$ for $\haus^{d-1}$-a.e.~point on $\partial \Omega$ and $\Omega$ has $C^5$-boundary, yet another application \BBB of elliptic regularity \EEE leads to $u \in H^5(\Omega; \R^d)$ and the bound \eqref{u_H5_bound}. \ZZZ Then, \EEE 
    \eqref{dirichlet_laplace2x_u} follows from $\Deltatwo u = v$ a.e.~in $\Omega$ and \eqref{NNNNN}. \EEE 
\end{proof}

\typeout{References}

\end{document}